\documentclass[12pt]{article}
\usepackage[english]{babel}
\usepackage{amsmath,amsthm,amsfonts,amssymb,epsfig, titlesec, epsfig, float, caption, mathrsfs, enumitem}
\usepackage[hidelinks]{hyperref}
\usepackage{tikz-cd}
\usepackage{tikz}
\usetikzlibrary{calc, decorations.pathreplacing,shapes.misc}
\usepackage[left=1in,top=1in,right=1in]{geometry}

\newtheorem{thm}{Theorem}[section]
\newtheorem{lem}[thm]{Lemma}
\newtheorem{prop}[thm]{Proposition}
\newtheorem{cor}[thm]{Corollary}
\newtheorem{conj}[thm]{Conjecture}
\newtheorem{ex}[thm]{Example}
\newtheorem{assump}[thm]{Assumption}

 \theoremstyle{remark} \newtheorem{rem}[thm]{Remark}
 
\theoremstyle{definition} \newtheorem{df}[thm]{Definition} 
\newtheorem{qs}[thm]{Question}

\titleformat*{\section}{\normalsize \bfseries \filcenter}
\titleformat*{\subsection}{\normalsize \bfseries }
\newcommand{\eps}{\varepsilon}
\newcommand{\wt}{\widetilde}
\newcommand{\bb}{\mathbb}

\DeclareMathOperator{\id}{id}
\DeclareMathOperator{\dg}{dg}
\DeclareMathOperator{\Hom}{Hom}
\DeclareMathOperator{\Log}{Log}
\DeclareMathOperator{\Ext}{Ext}
\DeclareMathOperator{\Simp}{Simp}
\DeclareMathOperator{\Morse}{Morse}
\DeclareMathOperator{\Cech}{\check{C}ech}
\DeclareMathOperator{\Coh}{Coh}
\DeclareMathOperator{\Pic}{Pic}

\makeatletter
\renewcommand*\env@matrix[1][*\c@MaxMatrixCols c]{%
  \hskip -\arraycolsep
  \let\@ifnextchar\new@ifnextchar
  \array{#1}}
\makeatother

\begin{document}

\title{\normalsize \textbf{Monodromy of monomially admissible Fukaya-Seidel categories mirror to toric varieties}}
\author{ \normalsize Andrew Hanlon}
\date{}
\maketitle

\begin{abstract}
Mirror symmetry for a toric variety involves Laurent polynomials whose symplectic topology is related to the algebraic geometry of the toric variety. We show that there is a monodromy action on the monomially admissible Fukaya-Seidel categories of these Laurent polynomials as the arguments of their coefficients vary that corresponds under homological mirror symmetry to tensoring by a line bundle naturally associated to the monomials whose coefficients are rotated. In the process, we introduce the monomially admissible Fukaya-Seidel category as a new interpretation of the Fukaya-Seidel category of a Laurent polynomial on $(\bb{C}^*)^n$, which has other potential applications, and give evidence of homological mirror symmetry for non-compact toric varieties.
\end{abstract}

\section{Introduction}

Let $X^n$ be a smooth complete toric variety given by a fan $\Sigma \subset N_\bb{R} \simeq \bb{R}^n$ where $N \simeq \bb{Z}^n$ is a lattice and $N_\bb{R} = N \otimes_\bb{Z} \bb{R}$. Let $A \subset N \simeq \bb{Z}^n$ be the set of primitive generators of $\Sigma$. Associated to this data is the Hori-Vafa superpotential which is a Laurent polynomial on $(\bb{C}^*)^n$ of the form 
\[ W_\Sigma =  \sum_{\alpha \in A} c_\alpha z^\alpha \] 
where $z^\alpha = z_1^{\alpha_1} \hdots z_n^{\alpha_n}$ for $\alpha = (\alpha_1, \hdots, \alpha_n) \in A$. The coefficients $c_\alpha$ are traditionally real positive and determined by a choice of K\"{a}hler form on $X$ although they can be more generally interpreted as elements of a Novikov ring. The pair $( (\bb{C}^*)^n, W_\Sigma)$ is a mirror to the toric variety $X$ when $X$ is Fano. This statement can be justified as follows in terms of SYZ mirror symmetry. Viewing $X$ as a compactification of $(\bb{C}^*)^n$, which is self-mirror, the mirror superpotential is a count of Maslov index 2 discs that intersect the boundary divisor. The boundary divisor in $X$ is a union of divisors $D_\alpha$ corresponding to $\alpha \in A$. When $X$ is Fano, each of the $D_\alpha$ has positive Chern number and thus each Maslov index 2 disc intersects exactly one $D_\alpha$ resulting in the superpotential having the form above (for more details see for instance \cite{Ch}). In general, $W_\Sigma$ is only the leading term in the correct disc counting potential which is modified by the presence of nonpositive divisors. However, $W_\Sigma$ still encodes much of the information relevant to homological mirror symmetry (HMS) and has been used successfully in establishing HMS theorems for toric varieties; we now review some of the prior work on this topic.

Here, we treat $X$ as the B-model in the HMS conjecture. In other words, we look at a derived equivalence of categories between a Fukaya-Seidel type category for  $W_\Sigma$ and the category of coherent sheaves $\Coh(X)$.\footnote{A proof of HMS for toric varieties in the other direction is the subject of work in progress by M. Abouzaid, K. Fukaya, Y.-G. Oh, H. Ohta, and K. Ono} HMS in the direction that we consider was first established in certain examples by Seidel \cite{MoreVan} and Auroux-Katzarkov-Orlov \cite{AKO2}, but the general case was studied extensively by Abouzaid \cite{Ab1, Ab2} and Fang-Liu-Treumann-Zaslow \cite{FLTZ1, FLTZ2, FLTZ3, FLTZS}. In \cite{Ab1, Ab2}, the model for the Fukaya category is a slight variation on the Fukaya-Seidel category $\mathcal{FS}(W_\Sigma)$ introduced by Seidel \cite{SeBook, Lef1} following ideas of Kontsevich. Roughly speaking, the objects of the Fukaya-Seidel category are Lagrangians that fiber over $\bb{R}_{>0} \subset \bb{C}$ under the map $W_\Sigma$ outside of a compact subset. Abouzaid shows in \cite{Ab2} that there is a full subcategory of the derived Fukaya-Seidel category quasi-equivalent to a $\dg$-enhancement of $D^b \Coh(X)$ when $X$ is projective. This subcategory is expected to split-generate in the case that $X$ is Fano (see Lemma 5.2 of \cite{AuSpec} for some justification). In general, it is known that this subcategory cannot split-generate. For example, it is shown in \cite{AKO1} that for the non-Fano Hirzebruch surfaces $\bb{F}_m$ for $m \geq 3$ that the derived Fukaya-Seidel category is quasi-equivalent to $D^b \Coh (\bb{P}^2(1,1,m) )$ in which $D^b \Coh (\bb{F}_m)$ sits as a strict subcategory. This example is also discussed from our perspective in Section \ref{tropvcomb}.

In a different direction, it is shown in \cite{FLTZ2} that there is a derived equivalence between the category of torus equivariant coherent sheaves on $X$ and the category of constructible sheaves on $\bb{R}^n$, which should be viewed as the universal cover the base torus of $T^*T^n \simeq (\bb{C}^*)^n$, with microsupport in a Lagrangian skeleton $\bb{L}_\Sigma$ determined by $\Sigma$. The nonequivariant case is shown in certain cases in \cite{Ku1, SS, Tr} and has recently been shown to hold in great generality (without any Fano assumptions and in the singular case with coherent sheaves replaced by perfect complexes) in \cite{Ku2}. Categories of microlocal sheaves were first related to infinitesimally wrapped Fukaya category by \cite{N} and \cite{NZ}. More recently, an equivalence between partially-wrapped Fukaya categories and microlocal sheaf categories has been established by Ganatra-Pardon-Shende in \cite{GPS3}. Thus, we see that much progress has been made on homological mirror symmetry for toric varieties, but the relationship between the different models of the Fukaya-Seidel category used remains somewhat unclear as illustrated in further detail in Section \ref{relate}.

Here, we will introduce yet another Fukaya-Seidel type category $\mathcal{F}_\Delta (W_\Sigma)$ in order to study monodromy in the space of mirror Landau-Ginzburg models to $X$ where $\Delta$ is extra data called a monomial division and defined in Definition \ref{div}. In particular, we look at families of functions
\begin{equation} \label{moneqintro} W_\Sigma^{\theta, \alpha} = c_{\alpha} e^{i\theta}z^{\alpha} + \sum_{\beta \in A \setminus \{ \alpha \}} c_\beta z^{\beta} \end{equation}
and the autoequivalences that they induce on $\mathcal{F}_\Delta (W_\Sigma)$. We obtain the following result which is a combination of Example \ref{tensoronlines} and Theorem \ref{mirrorfunctors} in Section \ref{catmono}. 

\begin{thm} \label{mainthm} Suppose that $\Delta$ is adapted to $\Sigma$. The monodromy of the family $W_\Sigma^{\theta, \alpha}$ as $\theta$ goes from $0$ to $2 \pi$ induces an autoequivalence of $\mathcal{F}_\Delta(W_\Sigma)$. On the subcategory of Lagrangian sections, $\mathcal{F}^s_\Delta(W_\Sigma)$, this autoequivalence corresponds under mirror symmetry to $(\cdot) \otimes \mathcal{O}(D_\alpha)$. Moreover, the autoequivalences fit together to give an action of $\text{Pic}(X)$ on $\mathcal{F}_\Delta(W_\Sigma)$. 
\end{thm}

The definition of a monomial division being adapted to the fan appears as Definition \ref{adapt}. A monomial division adapted to $\Sigma$ exists for surfaces and projective spaces for the standard K\"{a}hler form on $(\bb{C}^*)^n$.\footnote{By standard K\"{a}hler form on $(\bb{C}^*)^n$, we will always mean $\sum d\log|z_i| \wedge d\theta_i$. It should be noted that such a specific choice of ``standard K\"{a}hler form" is unavoidable due to the fact that monomial divisions do not behave well with respect to even the SL$(n,\bb{Z})$ action on $(\bb{C}^*)^n$ by monomial coordinate changes as evidenced by the second example in Figure \ref{notadapted}.} We do not know the extent to which monomial divisions adapted to the fan exist in general for the standard K\"{a}hler form, but there always exist toric K\"{a}hler forms on $(\bb{C}^*)^n$ for which monomial divisions adapted to the fan exist as shown in Corollary \ref{adaptedexists}.

The autoequivalences induced by the monodromy in \eqref{moneqintro} and described by Theorem \ref{mainthm} can be generalized (or composed) to give a functor $F_D$ on $\mathcal{F}_\Delta(W_\Sigma)$ corresponding under mirror symmetry to $(\cdot) \otimes \mathcal{O}(D)$ for any toric divisor $D$ on $X$. When $D$ is an effective divisor, we can say more. 

\begin{thm} \label{introtrans} If $\Delta$ is a monomial division adapted to $\Sigma$ and $D$ is an effective toric divisor on $X$, there is a natural transformation from the identity to $F_D$ coming from Floer theory. On the subcategory of sections $\mathcal{F}_\Delta^s(W_\Sigma)$, this natural transformation corresponds under mirror symmetry to multiplication by a defining section of $D$. 
\end{thm}

Theorem \ref{introtrans} is the subject of Section \ref{defsect} and appears as Proposition \ref{transexist} and Theorem \ref{sectionthm}. We now move to discussing an illustrative example of the monodromy induced by \eqref{moneqintro}. 

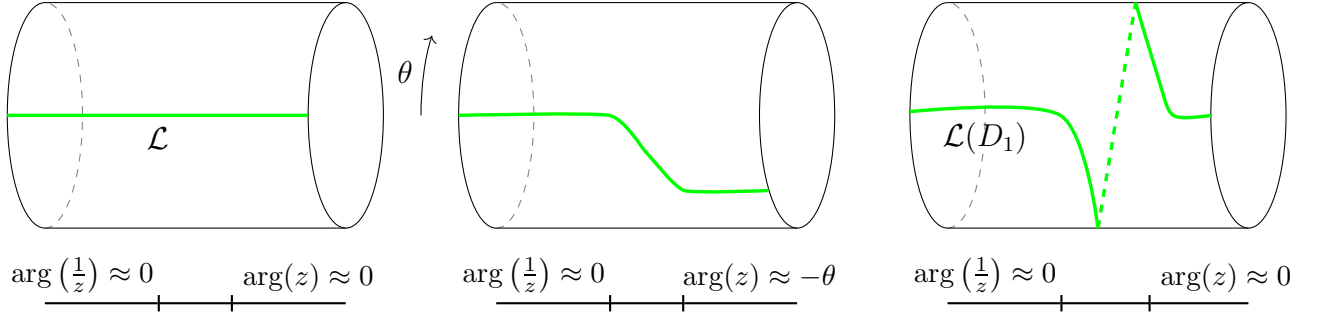
\begin{figure} 
\centering 
\begin{tikzpicture}
	\draw[dashed,color=gray] (-6,0) arc (-90:90:0.5 and 1.5);
	\draw (-6,0) -- (-2,0);
	\draw (-6,3) -- (-2,3);
	\draw (-6,0) arc (270:90:0.5 and 1.5);
	\draw (-2,1.5) ellipse (0.5 and 1.5);
	\draw (-6.5,1.5) -- (-2.5, 1.5) [color = green, line width = 1.3] node[font=\normalsize,midway,below, color= black] {$\mathcal{L}$};
	\draw[thick, |-|] (-4.5,-1) -- (-3.5,-1);
	\draw[thick] (-6,-1) -- (-4.5, -1) ;
	\draw (-5.5, -1) node [font=\small, above] {$\text{arg}\left(\frac{1}{z} \right) \approx 0$};
	\draw[thick] (-3.5, -1) -- (-2, -1); 
	\draw (-2.5, -1) node [font = \small, above] {$\text{arg}(z) \approx 0$};

	\draw[-> ,] (-1, 1.5 ) arc (180: 135 :0.5 and 1.5) node[font=\normalsize, midway, left] {$\theta$};
	\draw[dashed,color=gray] (0,0) arc (-90:90:0.5 and 1.5);
	\draw (0,0) -- (4,0);
	\draw (0,3) -- (4,3);
	\draw (0,0) arc (270:90:0.5 and 1.5);
	\draw (4,1.5) ellipse (0.5 and 1.5);
	\draw [green, line width = 1.3] plot [smooth, tension = 0.3] coordinates {(-0.5, 1.5) (1.5,1.5) (2,1) (2.5,0.5) (3.62, 0.5)};
	\draw[thick, |-|] (1.5,-1) -- (2.5,-1);
	\draw[thick] (0,-1) -- (1.5, -1) ;
	\draw (0.5, -1) node [font=\small, above] {$\text{arg}\left(\frac{1}{z} \right) \approx 0$};
	\draw[thick] (2.5, -1) -- (4, -1); 
	\draw (3.5, -1) node [font = \small, above] {$\text{arg}(z) \approx -\theta$};

	\draw[dashed,color=gray] (6,0) arc (-90:90:0.5 and 1.5);
	\draw (6,0) -- (10,0);
	\draw (6,3) -- (10,3);
	\draw (6,0) arc (270:90:0.5 and 1.5);
	\draw (10,1.5) ellipse (0.5 and 1.5);
	\draw [green, line width = 1.3] plot [smooth, tension = 0.5] coordinates {(5.5, 1.55) (7.5,1.5)  (8,0)  };
	\draw[dashed, green, line width = 1.3] plot [smooth, tension = 0.5] coordinates{(8, 0) (8.5,3)};
	\draw[green, line width = 1.3] plot [smooth, tension = 0.5] coordinates{(8.5, 3) (8.8, 2) (9, 1.5) (9.5, 1.5)} ;
	\draw (6.5, 1.2) node {$\mathcal{L}(D_1)$};
	\draw[thick, |-|] (7.5,-1) -- (8.7,-1);
	\draw[thick] (6,-1) -- (7.5, -1) ;
	\draw (6.5, -1) node [font=\small, above] {$\text{arg}\left(\frac{1}{z} \right) \approx 0$};
	\draw[thick] (8.7, -1) -- (10, -1); 
	\draw (9.7, -1) node [font = \small, above] {$\text{arg}(z) \approx 0$};
	
\end{tikzpicture}
\caption{Monodromy action of $e^{i\theta}z + \frac{1}{z}$, taking the mirror of $\mathcal{O}_{\bb{P}^1}$, $\mathcal{L} = \bb{R}_{>0}$, to the mirror of $\mathcal{O}_{\bb{P}^1}(1)$, $\mathcal{L}(D_1)$.} 
\label{p1ex}
\end{figure}

\subsection{Example: Projective Space} \label{introexample}

 In this section, we will analyze the monodromy action on a basis of thimbles which generate $\mathcal{FS}(W_\Sigma)$ even though our eventual proof will use the monomially admissible Fukaya-Seidel category which does not have thimbles as objects. Conjectural and established relationships between these categories are discussed in Section \ref{relate}. In fact, the embedding of the monomially admissible Fukaya-Seidel category of sections into Abouzaid's category of tropical Lagrangian sections exists for the mirror to projective space, and there is a more straightforward path to showing that Abouzaid's category generates $\mathcal{FS}(W_\Sigma)$ (see Section 7 of \cite{AuSpec}). See also Section \ref{thimblesec} for a further relation of thimbles to monomial admissibility.

The easiest example in which to see the monodromy of $W_\Sigma^{\theta, \alpha}$ in action is on the mirror to $\bb{P}^1$ (or a product of projective lines). In this case, the autoequivalence corresponds to a Dehn twist along one of the circles which is a fiber of the moment map. Figure \ref{p1ex} shows the twisting of the zero section, i.e., $\bb{R}_{>0}$, as one of the monomials is rotated in $W_\Sigma = z + \frac{1}{z}$. In that figure, we are using that $\text{arg}(e^{i\theta}z + 1/z) \approx \text{arg}(e^{i\theta} z)$ for $\log | z | \gg 0$ and $\text{arg}(e^{i\theta}z + 1/z) \approx \text{arg}(1/z)$ for $\log | z| \ll 0$. A monomial division will serve to change the $\approx$'s into equalities.

It is also possible to see the monodromy action on $\bb{P}^n$ in a basis of thimbles generating the Fukaya-Seidel category mirror to the full exceptional collection 
\begin{equation} \label{pnexceptional} \langle \mathcal{O}, \hdots, \mathcal{O}(n) \rangle \end{equation}
in $D^b \Coh(X)$. The critical points of the function
\[ z_1 + \hdots + z_n + \frac{1}{z_1 \hdots z_n} \]
are given by $z_1 = \hdots = z_n$ is a $(n+1)$th root of unity. Let $\zeta = e^{2\pi i / (n+1)}$. The critical values are $(n+1)\zeta^j$ for $j =0, \hdots, n$. The thimbles expected to correspond to the exceptional collection in \eqref{pnexceptional} lie over paths $\gamma_j \colon [0,1] \to \bb{C}$ from $(n+1)\zeta^j$ to a large positive real number which can be described as follows. The path $\gamma_0$ is the straight-line path on the real axis. Then, $\gamma_j$ is a path such that $\text{Arg}(\gamma_j(t)) \in [0, 2 \pi)$ is decreasing and $|\gamma_j(t) | > | \gamma_k(s)|$ when $\text{Arg}(\gamma_j(t)) = \text{Arg}(\gamma_k(s))$ for any $k < j$ and $s, t \in [0, 1)$. See the paths in Figure \ref{pnex} for an example. The fact that this basis of thimbles corresponds to the exceptional collection in \eqref{pnexceptional} for $n = 2$ can be deduced from Example (3B) in \cite{MoreVan} by performing an appropriate mutation. Moreover, this basis of thimbles should agree with the construction of $\mathcal{O}(k)$ in \cite{Ab1} up to Hamiltonian isotopy as noted in Section 7 of \cite{AuSpec}. 

Further, the critical points of 
\[ z_1 + \hdots + z_n + \frac{e^{i\theta}}{z_1 \hdots z_n} \]
are $z_1 = \hdots = z_n = e^{i\theta/(n+1)}\zeta^j$ with critical values $(n+1)e^{i\theta/(n+1)} \zeta^j$ for $j = 0, \hdots, n$. Thus, following the critical points and vanishing paths as $\theta$ varies from $0$ to $2\pi$ as shown in Figure \ref{pnex} clearly takes the thimbles corresponding to $\mathcal{O}, \hdots, \mathcal{O}(n-1)$ to those corresponding to $\mathcal{O}(1), \hdots, \mathcal{O}(n)$. Further, the thimble corresponding to $\mathcal{O}(n)$ is sent to a thimble that goes to the critical point above $1$ but with vanishing path that twists once around. By the description of the Serre functor in \cite{SympHoch}, one can see that this thimble corresponds to $\mathcal{O}(n+1)$. Thus, we see the functor $(\cdot) \otimes \mathcal{O}(1)$ as expected. Continuing this process and increasing $\theta$ further also yields the expected result by the same argument. 

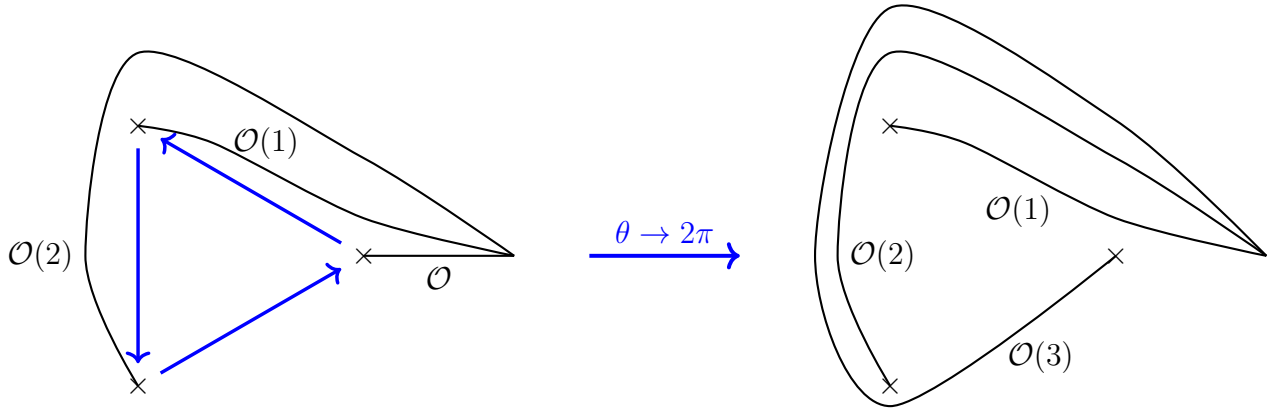
\begin{figure} 
\centering 
\begin{tikzpicture}

	\draw (2, 0) node {$\times$};
	\draw (-1, 1.73) node {$\times$};
	\draw (-1, -1.73) node {$\times$};
	\draw[thick] (2,0) -- (4,0) node[midway, below] {$\mathcal{O}$};
	\draw [thick] plot [smooth, tension = 0.5] coordinates {(-1, 1.73) (0, 1.5) (2,0.5) (4,0)};
	\draw (0.7, 1.5) node {$\mathcal{O}(1)$}; 
	\draw [thick] plot [smooth, tension = 0.5] coordinates {(-1, -1.73) (-1.7, 0) (-1, 2.7) (2,1.3) (4,0)};
	\draw (-2.3, 0) node {$\mathcal{O}(2)$}; 
	\draw [line width = 1.3 , blue, ->] (1.7, 0.17) -- (-0.7, 1.56);
	\draw [line width = 1.3, blue, ->] (-1, 1.43) -- (-1, -1.43);
	\draw [line width = 1.3, blue, ->] (-0.7, -1.56) -- (1.7, -0.17); 
	
	\draw[line width = 1.5, ->, blue] (5,0) -- (7,0) node[midway, above] {$\theta \to 2\pi$};
	
	\draw (12, 0) node {$\times$};
	\draw (9, 1.73) node {$\times$};
	\draw (9, -1.73) node {$\times$};
	\draw [thick] plot [smooth, tension = 0.5] coordinates {(9, 1.73) (10, 1.5) (12,0.5) (14,0)};
	\draw (10.7, 0.6) node {$\mathcal{O}(1)$}; 
	\draw [thick] plot [smooth, tension = 0.5] coordinates {(9, -1.73) (8.3, 0) (9, 2.7) (12,1.3) (14,0)};
	\draw (8.9, 0) node {$\mathcal{O}(2)$}; 
	\draw [thick] plot [smooth, tension = 0.5] coordinates {(12, 0) (9, -2) (8, 0) (9, 3.3) (12,1.8) (14,0)};
	\draw (11, -1.35) node {$\mathcal{O}(3)$};

\end{tikzpicture}
\caption{Monodromy action of $z_1 + z_2 + e^{i\theta}/z_1z_2$ on a basis of thimbles for the Fukaya-Seidel mirror of $\bb{P}^2$.} 
\label{pnex}
\end{figure}

\subsection{Outline of proof}

We now outline our approach to studying the monodromy of $W_\Sigma^{\theta, \alpha}$.  As already mentioned, we introduce Fukaya-Seidel categories particularly adapted to the problem using a monomial division to put a condition on the Lagrangians considered as objects. We call this condition monomial admissibility, and it appears in Definition \ref{adm}. As suggested by its name, monomial admissibility depends on the behavior of the monomials in $W_\Sigma$ allowing us to understand how the Lagrangians should change as the argument of a single coefficient varies. When the monomial division is adapted to the fan, we construct Hamiltonians $H_{D_\alpha}$ on $(\bb{C}^*)^n$ for each $\alpha \in A$ whose time $t$ flow takes Langrangians monomially admissible with respect to $W_\Sigma$ to Lagrangians monomially admissible with respect $W_\Sigma^{t, \alpha}$ for each $\alpha \in A$. We call these Hamiltonians twisting Hamiltonians (Definition \ref{twisthamdf}). By carefully setting up the Fukaya-Seidel category of monomially admissible Lagrangians, we obtain an action of the flows of the twisting Hamiltonians through essentially standard techniques in Floer theory. The natural transformations from Theorem \ref{introtrans} come from adapting to our setting the natural transformations coming from Hamiltonians in Section 10c of \cite{SeBook}. 

With the functors and natural transformations constructed, we still need to understand them through HMS. In order to do that, we restrict our attention to Lagrangian sections with respect to the moment map on $(\bb{C}^*)^n$. In fact, we show that any monomially admissible Lagrangian section can be obtained by flowing $\mathcal{L} = (\mathbb{R}_{>0})^n$ by a twisting Hamiltonian. Then, we translate the ideas of \cite{Ab2} to our setting in order to prove the HMS statement that the Fukaya-Seidel category consisting of monomially admissible Lagrangian sections is quasi-equivalent to a dg-enhancement of the category of line bundles on $X$. By closely examining this quasi-equivalence, we see that the functors induced by the flow of the twisting Hamiltonians act under the quasi-equivalence as tensor product with the appropriate line bundle by construction. Identifying the natural transformations requires the combination of an algebraic argument and restricting the presence of certain perturbed holomorphic disks using basic energy estimates in addition to the previously established understanding of the HMS quasi-equivalence.

\subsection{Relationship to other results and other definitions of Fukaya-Seidel categories}\label{relate}

The monodromy of $W_\Sigma^{\theta, \alpha}$ and its relationship to mirror symmetry can be viewed in several different lights. Theorem \ref{mainthm} is a generalization in the setting of mirrors to toric varieties of Seidel's description of the inverse Serre functor on Fukaya-Seidel categories as the total monodromy associated to the family of superpotentials $e^{i\theta} W_\Sigma$ in \cite{SympHoch}. It can also be seen as a part of the program of Diemer-Katzarkov-Kerr \cite{DKK1, DKK2} to understand the moduli space of Landau-Ginzburg models mirror to a Fano variety. From that perspective, it would be interesting to study the behavior of the monomially admissible Fukaya-Seidel category along other paths and loops in the moduli space. For instance, it may be possible to understand the mirrors to runs of the toric Mori program as described in \cite{BDFKK} in our framework. More abstractly, the framework of monomial divisions should be well-suited to understanding more of the natural B-side functors on the A-side of HMS. For example, blow-up functors can be described almost trivially for monomial divisions adapted to the fan, as in the microlocal setting. We plan to pursue these directions in future work. 

There are several different approaches to an A-side category for the mirror $W_\Sigma$ to a toric variety which have been mentioned earlier in the introduction. Assuming that $X$ is projective so that all categories considered are well-defined, we now recall or introduce notation for these categories and very roughly their definitions. Most traditionally, there is the Fukaya-Seidel category $\mathcal{FS}(W_\Sigma)$ as defined in \cite{SeBook, Lef1} whose objects are roughly Lagrangians that fiber over the real positive axis with morphisms computed by counterclockwise rotation of the input Lagrangian in the base. Seidel has shown that $\mathcal{FS}(W_\Sigma)$ is generated by thimbles when $W_\Sigma$ is Lefschetz. To $W_\Sigma$, Ganatra-Pardon-Shende in \cite{GPS1} associate a partially wrapped Fukaya category $\mathcal{W}( (\bb{C}^*)^n, W_\Sigma^{-1}(\infty))$ roughly defined as having Lagrangians that avoid $W_\Sigma^{-1}(\infty)$ and are conical on a Legendrian in the contact boundary and morphisms computed by pushing the input Lagrangian arbitrarily close to $W_\Sigma^{-1}(\infty)$ by positive wrapping, i.e., in the direction of Reeb flow in the contact boundary. Their definition is closely related to the partially wrapped Fukaya category associated to $W_\Sigma$ by Sylvan in \cite{Syl}. 

 While the previous two categories are defined for any Landau-Ginzburg model, there are categories more specifically defined for the toric setting. In \cite{Ab2}, Abouzaid primarily considers the $A_\infty$-pre-category $\mathscr{T}Fuk((\bb{C}^*)^n, M_{t,1})$ of Lagrangian sections of $\Log$ with boundary on a symplectic hypersurface $M_{t,1}$ which is the fiber of a tropical localization of $W_\Sigma$. On the other hand, Fang-Liu-Treumann-Zaslow in \cite{FLTZ1} define a singular Lagrangian skeleton 
 \[ \bb{L}_\Sigma = \bigcup_{\sigma \in \Sigma} \sigma^{\perp} \times \sigma \]
 of $(\bb{C}^*)^n \simeq T^*T^n$ and put forward the category $\mu\text{Sh}(\bb{L}_\Sigma)$ of constructible sheaves with microsupport in $\bb{L}_\Sigma$ as an A-side category for $W_\Sigma$.\footnote{Our sign conventions for $\bb{L}_\Sigma$ are opposite of those in \cite{FLTZ1} but rather agree with those used in \cite{GS}.} From \cite{GPS2}, a partially wrapped category Fukaya category $\mathcal{W}((\bb{C}^*)^n, \bb{L}_\Sigma)$ can also be defined for the singular Lagrangian $\bb{L}_\Sigma$ in a similar fashion to $\mathcal{W}( (\bb{C}^*)^n, W_\Sigma^{-1}(\infty))$. 

In general, we expect these categories to be related according to the following diagram where we label each arrow and discuss it below. 

\[ \begin{tikzcd} \mathcal{F}_{\Delta^t}(W_\Sigma) \arrow{d}{8} & & \mathcal{W}( (\bb{C}^*)^n, \bb{L}_\Sigma ) \arrow[leftrightarrow]{r}{3} & \mu\text{Sh}(\bb{L}_\Sigma) \\
\mathcal{FS}(W_\Sigma) \arrow[leftrightarrow]{r}{4} & \mathcal{W}( (\bb{C}^*)^n, W_\Sigma^{-1}(\infty)) \arrow{ur}{2} & \\
  \mathscr{T}\text{Fuk}((\bb{C}^*)^n, M_{t,1}) \arrow{u}{1}  & & 
\mathcal{F}^s_{\Delta^c}(W_\Sigma) \arrow{ll}{6} \arrow{r}{5} \arrow{uu}{7} & \mathcal{F}_{\Delta^c}(W_\Sigma) \end{tikzcd} \]

Some of the arrows are well-established and well-understood while others are less so. The reader should be aware that our assumption that $X$ is projective allows us to simplify things somewhat by ignoring certain dualities. In the diagram, $\Delta^t$ and $\Delta^c$ are tropical and combinatorial divisions as in Definitions \ref{tropdiv} and \ref{combdiv}, respectively. However, $\Delta^c$ could be replaced by any monomial division adapted to $\Sigma$ as the category of monomially admissible Lagrangian sections is invariant among such divisions (Proposition \ref{adaptedinvariant}).

Arrows 1-4 have been studied outside of the present work. The fact that arrow 1 represents an embedding essentially follows from the setup in \cite{Ab2} and the proof in \cite{Ab1} that the deformation of $W_\Sigma$ to its tropical localization is symplectically trivial. To construct the embedding, it is also necessary to deal with small technical differences in the definitions such as extending the objects of $\mathscr{T}\text{Fuk}((\bb{C}^*)^n, M_{t,1})$ to lie over arcs by parallel transport and dealing with the fact that $\mathscr{T}\text{Fuk}((\bb{C}^*)^n, M_{t,1})$ is an $A_\infty$-pre-category as in Appendix \ref{precat}. $\mathscr{T}\text{Fuk}((\bb{C}^*)^n, M_{t,1})$ is expected to split-generate $\mathcal{FS}(W_\Sigma)$ when $X$ is Fano. An embedding representing 2 comes from a pushforward functor defined in \cite{GPS1, GPS2} after seeing that the Legendrian piece of $\mathbb{L}_\Sigma$ in the contact boundary can be seen as a subset of the fiber at $\infty$ of a tropical localization of $W_\Sigma$. In fact, $\mathbb{L}_\Sigma$ is the core of this fiber when $X$ is Fano as shown in \cite{GS, Zhou}. It then follows from \cite{GPS2} that arrow 2 is a quasi-equivalence for Fano $X$. The derived equivalence represented by arrow 3 is a case of the main result of \cite{GPS3}. There is expected to be a quasi-equivalence corresponding to arrow 4. However, to carefully establish this quasi-equivalence, one has to deal with the differing requirements on objects at infinity. In particular, both categories are generated by thimbles when $W_\Sigma$ is Lefschetz (which is true for a generic choice of coefficients) as shown for $\mathcal{W}( (\bb{C}^*)^n, W_\Sigma^{-1}(\infty))$ in \cite{GPS2}, but Floer theory of the thimbles is computed with different requirements on the almost complex structure. To go between the two, a careful compactness argument for holomorphic disks is required which has yet to appear in the literature. Also, the arrow 4 appears more naturally when using $W_\Sigma^{-1}(-\infty)$ as a stop, but $\mathcal{W}( (\bb{C}^*)^n, W_\Sigma^{-1}(-\infty))$ can be (non-canonically) identified with $\mathcal{W}( (\bb{C}^*)^n, W_\Sigma^{-1}(\infty))$. 

Arrows 5-8 involve monomially admissible Fukaya-Seidel categories. Arrow 5 is an embedding that is simply the result of $\mathcal{F}_{\Delta^c}^s(W_\Sigma)$ being defined as a subcategory of $\mathcal{F}_{\Delta^c}(W_\Sigma)$. We expect that this subcategory always split-generates, but the appropriate closed-string theory to prove this has not been developed. There is an embedding representing arrow 6 when $\Delta^c$ is adapted to the fan with the standard K\"{a}hler form (it is not known whether this is always true as discussed in Section \ref{existencesect}). In fact, this embedding, when it exists, is a quasi-equivalence and is closely related to Conjecture \ref{thimbleconj} stating that mirrors to line bundles can be represented by thimbles. The construction of this map is outlined at the beginning of Section \ref{hms}. An embedding corresponding to arrow 7 can be constructed by modifying the construction of twisting Hamiltonians to make them homogeneous of degree 1 (see Remark \ref{conical}). This modification will cause angles that were previously controlled to lie in small intervals, but this is enough to construct sections that avoid $\bb{L}_\Sigma$. The details of this construction will appear in follow-up work in preparation. Morally, the embedding comes from the fact that the part of the contact boundary where the ends of monomially admissible Lagrangians with respect to $\Delta^c$ are required to lie can easily be seen to retract onto $\bb{L}_\Sigma$. Finally, the arrow 8 is the most conjectural. We expect that such an arrow exists and is a quasi-equivalence for an appropriate choice of toric K\"{a}hler form from Remark \ref{righthalf} and the computation and discussion in Section \ref{tropvcomb}. An explicit construction may need to pass through an embedding similar to arrow 7, but with respect to the FLTZ skeleton of an anticanonical model of $X$ that is actually a skeleton of the fiber. An alternative approach would be to pass through an embedding similar to that of arrow 6, but using a different tropical localization than Abouzaid.

Thus, we see that the monomially admissible Fukaya-Seidel category can act as an intermediary between the various existing A-model categories associated to $W_\Sigma$. In particular, the monomially admissible Fukaya-Seidel category for $\Delta^c$ is closely related to the combinatorially defined categories $\mathscr{T}\text{Fuk}((\bb{C}^*)^n, M_{t,1}), \mathcal{W}((\bb{C}^*)^n, \bb{L}_\Sigma),$ and $\mu\text{Sh}(\bb{L}_\Sigma)$, which are all known to be derived equivalent to $\Coh(X)$. For a tropical division, the monomially admissible Fukaya-Seidel category should correspond to the traditional Fukaya-Seidel category $\mathcal{FS}(W_\Sigma)$ and to $\mathcal{W}( (\bb{C}^*)^n, W_\Sigma^{-1}(\infty))$, which are expected to instead be derived equivalent to the category of coherent sheaves on an anticanonical model of $X$ following \cite{AKO1} and \cite{BDFKK}. In both cases, the monomially admissible Fukaya-Seidel category of Lagrangian sections gives a computable and entirely Floer-theoretic model that we expect can be naturally used in other HMS constructions.

\subsection{Organization}

The remainder of the paper is organized as follows. Section \ref{2} includes the basic definitions of and related to monomial divisions (\ref{monadmcond}), the definition and construction of twisting Hamiltonians (\ref{twistham}), the identification of Hamiltonian isotopy classes of monomially admissible Lagrangian sections with line bundles on $X$ (\ref{sectionclasses}), and the proof that monomial divisions adapted to the fan exist for an appropriate choice of toric K\"{a}hler form (\ref{existencesect}). Section \ref{3} provides background results on pseudoholomorphic discs (\ref{control} and \ref{Floer}) and defines the monomially admissible Fukaya-Seidel category (\ref{flocalize}). It also includes a simple example of HMS for ample line bundles (\ref{floercomp}), the proof of HMS for the category of monomially admissible Lagrangian sections (\ref{hms}), and some speculations on how the Fano condition relates to monomial divisions (\ref{tropvcomb}). Section \ref{4} contains the definition of the monodromy functors with the precise statement and proof of Theorem \ref{mainthm} (\ref{catmono}), the relevant definitions, statement, and proof of Theorem \ref{introtrans} (\ref{defsect}), HMS for the complements of toric divisors by localization (\ref{openhms}), and a conjecture that all line bundles are mirror to thimbles (\ref{thimblesec}). Finally, Appendix \ref{precat} shows that the localization approach to defining Fukaya-Seidel categories is equivalent to the approach using $A_\infty$-pre-categories.

\medskip

\noindent \textbf{Acknowledgements.} First and foremost, I would like to thank my advisor Denis Auroux for originally suggesting to work on this problem, his patience, and many useful suggestions and comments. I would also like to thank Mohammed Abouzaid for helpful discussions related to this work and the anonymous referee for various useful suggestions on improving the exposition. In addition, I am grateful to Harvard University for their hospitality while part of this work was completed. This work was partially supported by the Simons Foundation (grant \#385573, Simons Collaboration on Homological Mirror Symmetry), NSF grant DMS-1264662, and NSF RTG grant DMS-1344991.

\section{Monomially admissible Lagrangians and twisting Hamiltonians} \label{2}

We will set up a Fukaya-Seidel category for a Laurent polynomial $W = \sum_{\alpha \in A} c_\alpha z^\alpha$ on $(\bb{C}^*)^n$ where Lagrangians are subject to an admissibility condition constraining the behavior of the monomials in $W$ rather than that of $W$ itself.\footnote{It is not necessary in the definition to assume that $A$ is the set of primitive generators of a fan. It should also be noted that all of the information in $W$ plays a role even though Definition \ref{adm} only uses the elements of $A$ and the arguments of the $c_\alpha$. The norms of the $c_\alpha$ are important in Definition \ref{div}.} The fact that admissibility depends on the monomials makes this setup particularly apt for understanding the monodromy which is the subject of this paper. Before getting into the details of the definition of the category, we will define the notion of a monomial admissibility condition and study the geometry of monomially admissible Lagrangian sections. The precise setup of the Fukaya-Seidel category will be done in Section \ref{categories}.

Generally speaking, a monomial admissibility condition requires each monomial in $W$ to be real positive over a corresponding subset of $(\bb{C}^*)^n$ near infinity. The introduction of such a condition can be rationalized from the viewpoint of SYZ mirror symmetry. Recall that the SYZ construction produces the mirror of a manifold with a Lagrangian torus fibration as the total space of the dual torus fibration. The points of the mirror correspond to torus fibers with a unitary rank one local system. $(\bb{C}^*)^n$ admits a natural SYZ fibration over $\bb{R}^n$ and is self-mirror. It also sits as a dense open subset inside of the toric variety $X$. As we approach the boundary divisors, circles in the torus fibers collapse to a point. Thus, we should expect that the local systems comprising the SYZ dual fibers have to be trivial along the collapsing circle, that is, it is natural to require certain arguments to be zero or equivalently certain Laurent monomials to be real positive towards infinity. Note that the superpotential $W_\Sigma$ only naturally appears when considering SYZ mirror symmetry in the opposite direction (viewing $X$ as a symplectic manifold). This bit of philosophy also fits well with the partially wrapped Fukaya categories defined using stops by Sylvan \cite{Syl} or Liouville sectors by Ganatra-Pardon-Shende \cite{GPS1}.

With the basic definition of monomial admissibility in hand, we then turn to the construction of Hamiltonian functions that will rotate the arguments of Lagrangians at infinity in the desired way, which we will call twisting Hamiltonians. We will first see that the Hamiltonians are essentially forced upon us by the condition that they take $W_\Sigma$ monomially admissible Lagrangians to $W_\Sigma^{\theta, \alpha}$ monomially admissible Lagrangians when additional constraints are placed on the monomial division. We will then show that these Hamiltonians act appropriately on the Hamiltonian isotopy classes of monomially admissible Lagrangian sections. Finally, we will discuss the existence of monomial divisions satisfying the aforementioned constraints. 

\subsection{Monomial admissibility conditions} \label{monadmcond}

Fix a toric K\"{a}hler form $\omega$ on $(\bb{C}^*)^n$ with respect to the standard $T^n$ action and standard complex structure. In addition, fix a moment map $\mu \colon (\bb{C}^*)^n \to \bb{R}^n$. A monomial admissibility condition depends on a monomial division of the moment map image $\bb{R}^n$. Besides determining the notion of a monomially admissible Lagrangian, the subdivision will also be used to ensure the compactness of moduli spaces of discs needed to show that the Fukaya-Seidel category is well-defined.
\begin{df} \label{div} A monomial division $\Delta$ for $W$ and $\mu$ is an assignment of a closed set $C_\alpha \subset \bb{R}^n$ to each monomial in $W$ such that the following conditions hold.
\begin{enumerate}
\item The $C_\alpha$ cover the complement of a compact subset of $\bb{R}^n$.
\item There exist constants $k_\alpha \in \bb{R}_{>0}$ such that in the expression
\[ \max_{\alpha \in A} (|c_\alpha z^\alpha|^{k_\alpha}) \]
the maximum is always achieved by $|c_\alpha z^\alpha|^{k_\alpha}$ for an $\alpha$ such that $\mu(z) \in C_\alpha$.
\end{enumerate}
\end{df}

\begin{rem} The second condition in Definition \ref{div} makes sense as the moment map $\mu$ is $T^n$-invariant and thus a function of $|z_1|, \hdots, |z_n|$. In fact, $\mu = \Log$ for the standard symplectic form $\sum d\log|z_i| \wedge d\theta_i$ and changing toric K\"{a}hler form amounts to a change of coordinates on $\bb{R}^n$. 
\end{rem}

Both conditions in Definition \ref{div} are used in the argument in Section \ref{control} showing the compactness of moduli spaces of discs needed to obtain a well-defined Fukaya-Seidel category.  As we see in the examples below, the first condition arises naturally when working with complete toric varieties as we do in this paper. The second condition may at first seem somewhat artificial, but we will see in Section \ref{existencesect} that it also has geometric meaning. The following additional condition on a monomial division will play a central role in this paper.

\begin{df} \label{adapt} A monomial division $\Delta$ for $W_\Sigma$ is adapted to the fan $\Sigma$ if for each $\alpha \in A$, the moment map image of the complement of a compact subset of $\mu^{-1}(C_\alpha)$ is contained in the interior of the star of $\alpha$.\footnote{Here and throughout the paper, we use ``star of $\alpha$" to mean the star of the ray generated by $\alpha$ in the fan $\Sigma$.} 
\end{df}

\begin{rem} \label{ontomoment} If $\mu$ is onto, $\Delta$ being adapted to the fan reduces to the requirement that the complement of a compact subset of each $C_\alpha$ is contained in the interior of the star of $\alpha$.
\end{rem}

It should be noted that $\Delta$ being adapted to $\Sigma$ implies that for each $\alpha \in A$ and near infinity, the ray generated by $\alpha$ is contained in $C_\alpha$ and is not contained in $C_\beta$ for $\beta \neq \alpha$.  We will now see how a monomial division is used to constrain the asymptotic behavior of Lagrangians.

 A standard definition, originally due to Kontsevich, of an admissible Lagrangian is that outside of a compact subset its image under $W$ lies in $\bb{R}_{>0} \subset \bb{C}$. Motivated by the comments on SYZ mirror symmetry at the beginning of this section, we mimic this condition for monomials in their corresponding regions of the monomial division. 

\begin{df} \label{adm} A Lagrangian $L \subset (\bb{C}^*)^n$ is admissible with respect to a monomial division $\Delta$, or \textit{monomially admissible}, if over $\mu^{-1}(C_\alpha)$ the argument of $c_\alpha z^\alpha$ restricted to $L$ is zero outside of a compact set.
\end{df}

One can also imagine setting up a similar admissibility condition where the $C_\alpha$ cover $\bb{R}^n$ at infinity only in certain directions and Lagrangians are fully wrapped in the other directions. We will not undertake that geometric setup here, but we hint at it further in Section \ref{openhms}. 

For us, there are two key ``examples" of a monomial division and thus monomial admissibility when $W = W_\Sigma$ with $\Sigma$ a fan for a complete toric variety $X$ as in the introduction. The first of which we now discuss.

\begin{df} \label{tropdiv} A tropical division for $W_\Sigma$ has $C_\alpha$ equal to the region in $\bb{R}^n$ where 
\[ |c_\alpha z^\alpha| \geq (1- \delta) \max_{\beta \in A} (| c_\beta z^\beta|) \]
for a fixed $\delta \in [0,1]$.
\end{df}

\begin{rem} \label{righthalf} For small values of $\delta$, the $C_\alpha$ have smaller overlaps and there are more Lagrangians admissible with respect to the tropical division. For larger values of $\delta$, the monomial admissibility condition is harder to satisfy, but monomially admissible Lagrangians are closer to being admissible with respect to $W_\Sigma$, that is, closer to having $W_\Sigma$ be real positive on the Lagrangian outside of a compact set. For instance, $\delta \geq 1 - \frac{1}{N-1}$ where $N = | A |$ is the number of monomials in $W_\Sigma$ implies that the image of $W_\Sigma$ on the complement of a compact subset of a monomially admissible Lagrangian lies in the right half-plane.
\end{rem}

\begin{figure} 
\centering
\begin{tikzpicture}
\draw[yellow, fill=yellow, opacity = 0.5] (0,0) -- (-3, -2.5) -- (-3, -3) -- (-2.5, -3)  -- (0,0);
\draw [blue, fill=blue, opacity=0.5] (0,0) -- (3, 0.5) -- (3,-0.5) -- (0,0);
\draw[red, fill=red, opacity = 0.5] (0,0) -- (0.5,3) -- (-0.5, 3) -- (0,0);
\draw[green, fill=green, opacity = 0.5] (0,0) -- (3, -0.5) -- (3, -3) -- (-2.5, -3) -- (0,0); 
\draw[purple, fill=purple, opacity = 0.5] (0,0) -- (3, 0.5) -- (3,3) -- (0.5, 3) -- (0,0);
\draw[orange, fill = orange, opacity = 0.5] (0,0) -- (-0.5, 3) -- (-3, 3) -- (-3, -2.5) -- (0,0);
\draw (-2, -2.5) node {$C_{(-1,-1)}$};
\draw (2.4, 0) node {$C_{(1,0)}$};
\draw (0, 2.6) node {$C_{(0,1)}$};
\draw[->, thick] (0,0) --(1,0);
\draw[->, thick] (0,0) -- (0,1);
\draw[->, thick] (0,0) -- (-1,-1);
\draw[dashed] (1.97,0.33) arc (9.46:219.8:2);
\draw[dashed] (-0.96,-1.15) arc (-129.8:80.54:1.5);
\draw[dashed] (-0.41,2.47) arc (99.46:350.54:2.5);

\draw[yellow, fill=yellow, opacity = 0.5] (-8,0) -- (-6.5, -3) -- (-11, -3) -- (-11, 1.5)  -- (-8,0);
\draw [blue, fill=blue, opacity=0.5] (-8,0) -- (-6.5, -3) -- (-5,-3) -- (-5, 3) -- (-8,0);
\draw[red, fill=red, opacity = 0.5] (-8,0) -- (-11, 1.5) -- (-11, 3) -- (-5, 3) -- (-8,0);
\draw (-10, -2.5) node {$C_{(-1,-1)}$};
\draw (-5.6, 0) node {$C_{(1,0)}$};
\draw (-8, 2.6) node {$C_{(0,1)}$};
\draw[->,thick] (-8,0) --(-7,0);
\draw[->, thick] (-8,0) -- (-8,1);
\draw[->, thick] (-8,0) -- (-9,-1);

\end{tikzpicture}
\caption{Tropical and combinatorial divisions and fan for $\bb{P}^2$ with $\delta = 0, c_\alpha = 1$ for all $\alpha$, and $\mu = \Log$. In the second image, overlaps are colored by the mixture of the two colors and dashed arcs indicate the angles of the cones in the monomial division.} \label{firstexample}
\end{figure}
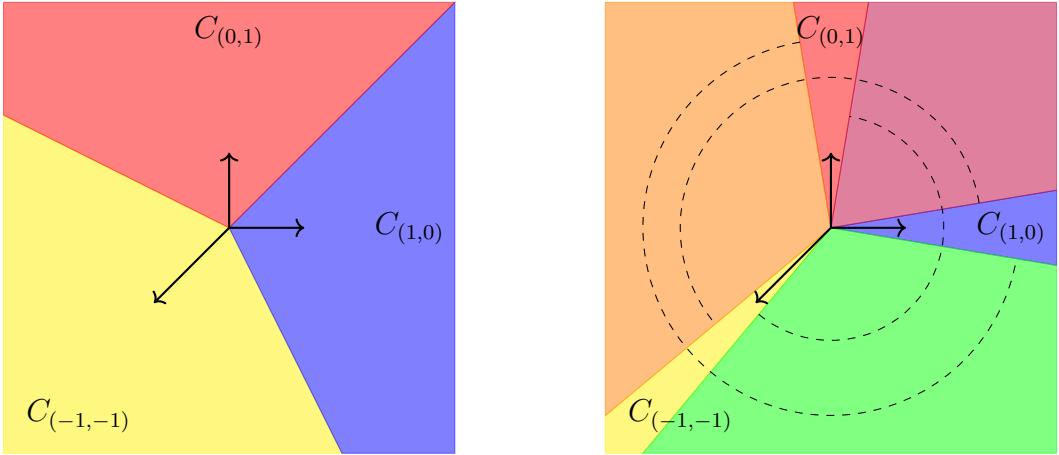

A tropical division is tautologically a monomial division with $k_\alpha = 1$ for all $\alpha \in A$ and should be thought of as a tropical interpretation of the Fukaya-Seidel category where only the dominating monomials are required to be real positive. When $\delta = 0$ and $\mu = \Log$, $C_\alpha$ is the region of the tropicalization of $W_\Sigma$ corresponding to $\alpha$. In particular, there will always be bounded or empty $C_\alpha$ when $X$ is not semi-Fano, that is, when there is a homology class which pairs negatively with $c_1(X)$. Our version of the Fukaya-Seidel category will then not ``see" the rays corresponding to those regions. We elaborate on this phenomenon in Section \ref{tropvcomb}. 

The possibility of bounded or empty $C_\alpha$ also shows that a tropical division is in general not adapted to the fan. Even when all $C_\alpha$ are unbounded (for example, when $X$ is Fano), a tropical division may still not be adapted to the fan. We will leave further discussion of when the tropical division is adapted to the fan to the more technical Section \ref{existencesect}. In particular, Corollary \ref{tropicalfanoadapts} in that section shows that there is always a toric K\"{a}hler form on $(\bb{C}^*)^n$ for which the tropical division is adapted to the fan for small $\delta$. 

We now move to introducing and discussing our second important ``example," which presents an issue that is in a sense dual to trying to find an adapted tropical division.

\begin{df} \label{combdiv} A combinatorial division for $W_\Sigma$ has $C_\alpha$ equal to a cone slightly inside the star of $\alpha$. More precisely, we have $x \in C_\alpha$ if $x$ is in the star of $\alpha$ and the angle between $x$ and any cone of $\Sigma$ not containing $\alpha$ is greater than or equal to $\eps$ for some small fixed $\eps > 0$.
\end{df}

The restrictions put on a monomially admissible Lagrangian with respect to the combinatorial division are a possible interpretation of what it means to be asymptotic to the Lagrangian skeleton in the setup of Fang-Liu-Treumann-Zaslow \cite{FLTZ3}. The slight shrinking of the star of $\alpha$ to obtain $C_\alpha$ in the combinatorial division is necessary for the existence of nontrivial smooth monomially admissible Lagrangians. The condition for $\Delta$ to be adapted to $\Sigma$ holds trivially for the combinatorial division. However, it is not clear when the combinatorial division satisfies the second condition of Definition \ref{div}, which is why we have been using quotes when referring to the combinatorial division as an example. Detailed discussion of this issue will again be postponed to Section \ref{existencesect}, which contains Corollary \ref{adaptedexists} showing that there is always a toric K\"{a}hler form on $(\bb{C}^*)^n$ for which the combinatorial division is adapted to the fan when $X$ is projective.

\subsection{Twisting Hamiltonians} \label{twistham}

Let $D = \sum_{\alpha \in A} n_\alpha D_\alpha$ be a toric divisor in $X$. Associated to $D$, we have a family of superpotentials 
\begin{equation} \label{monodromyeq} W_\Sigma^{\theta, D} = \sum_{\alpha \in A} c_\alpha e^{in_\alpha \theta} z^\alpha \end{equation}
as we considered earlier with $D = D_\alpha$. Suppose that the moment map coordinates on $\bb{R}^n$ for our toric K\"ahler form on $(\bb{C}^*)^n$ are $\mu = (\mu_1, \hdots, \mu_n)$ and that $H(\mu_1, \hdots, \mu_n)$ is a Hamiltonian function in the moment map coordinates. Since $\omega = \sum d\mu_i \wedge d\theta_i$ and $dH = \sum H_{\mu_i} d\mu_i$, we have that 
\[ X_H = \sum H_{\mu_i} \frac{\partial}{\partial \theta_i} \]
with the convention that $i_{X_H} \omega = - dH$. Therefore, the flow of $X_H$ is given by 
\[ \phi^t_H (\mu_1, \theta_1, \hdots, \mu_n, \theta_n) = \left( \mu_1, \theta_1 + tH_{\mu_1}, \hdots, \mu_n, \theta_n + tH_{\mu_n} \right). \]
Thus, we are led to define a special class of Hamiltonians whose time $t$ flow takes a $W_\Sigma$ monomially admissible Lagrangian to a $W_\Sigma^{2\pi t, D}$ monomially admissible Lagrangian.
\begin{df} \label{twisthamdf} A Hamiltonian $H = H(\mu_1, \hdots, \mu_n)$ is a twisting Hamiltonian for $D=\sum n_\alpha D_\alpha$ if 
\begin{equation} \label{twistcond} \nabla H \cdot \alpha = - 2\pi n_\alpha \end{equation}
outside of a compact set in each $C_{\alpha}$ where the gradient is taken with respect to the flat metric and $\cdot$ is the dot product. Moreover, we will call $H$ an \textit{admissible Hamiltonian} if $D$ is the empty divisor, i.e., $n_\alpha = 0$ for all $\alpha \in A$.
\end{df}

\begin{rem} The set of twisting Hamiltonians for a fixed divisor $D$ is an affine linear space in the sense that $tH_1 +(1-t)H_2$ satisfies \eqref{twistcond} for all $t \in \bb{R}$ when $H_1$ and $H_2$ satisfy \eqref{twistcond}. In particular, this set is convex. This convexity manifests itself in the fact that the images of a monomially admissible Lagrangian under the time-$1$ flows of any two twisting Hamiltonians for $D$ are Hamiltonian isotopic through monomially admissible Lagrangians.
\end{rem}

\begin{rem} The set of all twisting Hamiltonians form a group under addition with identity the Hamiltonian $H(\mu) = 0$, which is an admissible Hamiltonian.
\end{rem}

Note that $H$ being a twisting Hamiltonian for $D$ places strong restrictions on $H$ where the $C_\alpha$ overlap. In order to control these overlaps, we will assume that we are working with a combinatorial division and then observe that our construction produces a twisting Hamiltonian for $D$ for any monomial division adapted to the fan in Corollary \ref{twistexist}.\footnote{This is because the overlaps that occur in a monomial division adapted to the fan are always a subset of the overlaps in a combinatorial division.} In the case of a combinatorial division, $H$ is essentially determined up to smoothing outside of a compact set. For each maximal cone $\sigma = \langle \alpha_1, \hdots, \alpha_n \rangle$ of $\Sigma$, we must have that in the intersection $C_{\alpha_1} \cap \hdots \cap C_{\alpha_n}$, $H = 2\pi m_\sigma \cdot \mu$ where $m_\sigma$ is the unique vector satisfying $m_\sigma \cdot \alpha_i = - n_{\alpha_i}$ for $i = 1, \hdots, n$. The linear functions $2 \pi m_\sigma \cdot \mu$ naturally glue to a piecewise linear function $F_D$ given by $F_D(\mu) = 2 \pi m_\sigma \cdot \mu$ when $\mu \in C$ for each maximal cone $C$ of $\Sigma$. In other words, $F_D$ is the unique function that is linear on the maximal cones of $\Sigma$ and satisfies $F_D(\alpha) = -2 \pi n_\alpha$ for all $\alpha \in A$. After rescaling by $2\pi$, the piecewise linear function $F_D$ is more commonly known as the support function of the divisor $D$ or line bundle $\mathcal{O}(D)$ (see for instance Chapter 4 of \cite{CLS}). To construct a twisting Hamiltonian $H$, we need only to appropriately smooth the function $F_D$. 

\begin{rem} We are assuming throughout that $\Sigma$ is smooth, and we use that condition above to obtain $F_D$. The weakest condition on $\Sigma$ for which the construction will work for any divisor is that it is simplicial. 
\end{rem}

Before proceeding to the general construction of a smooth twisting Hamiltonian from $F_D$, let's consider the example of the divisor $D = D_{(1,0)}$ on $\bb{P}^2$. There are three maximal cones in the fan of $\bb{P}^2$: $\sigma_1 = \langle (1,0), (0,1) \rangle, \sigma_2 = \langle (1,0), (-1,-1) \rangle$, and $\sigma_3 = \langle (0,1), (-1, -1) \rangle$. In this case, we have that 
\[ F_D(\mu) = \begin{cases} - 2 \pi \mu_1 & \mu \in \sigma_1 \\ 2\pi ( \mu_2 - \mu_1) & \mu \in \sigma_2 \\ 0 & \mu \in \sigma_3 \end{cases}. \]
To construct $H$, we need to smooth $F_D$ near the rays of the fan and near the origin. Near the origin, there are no restrictions as the smoothing can be done in a compact set. Near a ray generated by $\alpha$, we need to perform the smoothing in a way that preserves equation \eqref{twistcond}. For instance, near the ray we $(1,0)$, we need to preserve $\partial H/ \partial \mu_1 = - 2 \pi$. Fortunately, the difference between $F_D$ on $\sigma_1$ and $F_D$ on $\sigma_2$ is function only of $\mu_2$ so such a smoothing is possible. Explicitly, we can take $H$ to be $-2\pi (u_1 +(-u_2)^+ )^+$ where $( \cdot )^+ \colon \bb{R} \to \bb{R}$ is a convex smooth function equal to $0$ when $t \ll 0$ and equal to $t$ when $t \gg 0$. Some level sets of the resulting twisting Hamiltonian are shown in Figure \ref{levelsets} along with the level sets of twisting Hamiltonians for other surfaces constructed in the same manner. In higher dimensions, it is possible to proceed similarly by observing that $F_D$ still varies between maximal cones by the normal coordinate and proceed by induction. However, we will use mollifier functions instead of a direct approach due to their general properties.

\begin{figure}
\centering 
\begin{tikzpicture}
\draw[thick, ->] (0,0) -- (1,0);
\draw[thick, ->] (0,0) -- (0,1);
\draw[thick, ->] (0,0) -- (-1,-1);
\draw[ultra thick, blue] plot [smooth, tension = 0.5] coordinates{(0.33, 1) (0.33, 0.1) (0, -0.33) (-0.67, -1)};
\draw[ultra thick, blue] plot [smooth, tension = 0.5] coordinates{(0.67, 1) (0.67, 0.1) (0, -0.67) (-0.33, -1)};
\draw[ultra thick, blue] plot [smooth, tension = 0.5] coordinates{(1, 1) (1, 0.1) (0.5, -0.5 ) (0, -1)};

\draw[thick, ->] (4,0) -- (5,0);
\draw[thick, ->] (4,0) -- (4,1);½
\draw[thick, ->] (4,0) -- (3,0); 
\draw[thick, ->] (4,0) -- (4,-1);
\draw[ultra thick, blue] (4.33,1) -- (4.33, -1);
\draw[ultra thick, blue] (4.67,1) -- (4.67, -1);
\draw[ultra thick, blue] (5,1) -- (5, -1);

\draw[thick, ->] (8,0) -- (9,0);
\draw[thick, ->] (8,0) -- (8,1);
\draw[thick, ->] (8,0) -- (9,1);
\draw[thick, ->] (8,0) -- (7,-1);
\draw[ultra thick, blue] (9.33,1) -- (7.33, -1);
\draw[ultra thick, blue] (9.67,1) -- (7.67, -1);
\draw[ultra thick, blue] (10,1) -- (8, -1);

\draw[thick, ->] (12,0) -- (13,0);
\draw[thick, ->] (12,0) -- (12,1);
\draw[thick, ->] (12,0) -- (13,1);
\draw[thick, ->] (12,0) -- (11,-1);
\draw[thick, ->] (12,0) -- (12,-1);
\draw[ultra thick, blue] plot [smooth, tension = 0.3] coordinates{(13.33, 1) (12.67, 0.33) (12.33, -0.1) (12.33, -1)};
\draw[ultra thick, blue] plot [smooth, tension = 0.3] coordinates{(13.67, 1) (13, 0.33) (12.67, -0.1) (12.67, -1)};
\draw[ultra thick, blue] plot [smooth, tension = 0.3] coordinates{(14, 1) (13.33 , 0.33) (13, -0.1) (13, -1)};

\end{tikzpicture}
\caption{Some level sets of a twisting Hamiltonian for the divisor $D_{(1,0)}$ on $\bb{P}^2, \bb{P}^1 \times \bb{P}^1$, and some blow-ups of $\bb{P}^2$ overlaid on the fans.} \label{levelsets}
\end{figure}
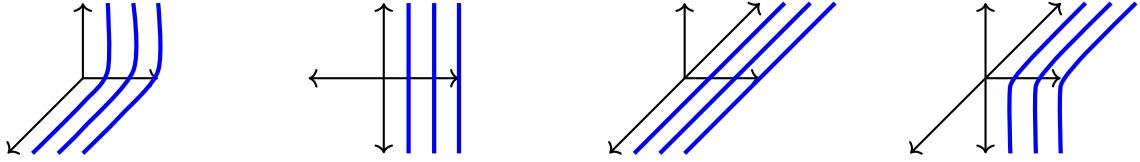

\begin{prop} \label{smooth} There exists a smooth function $H$ equal to $F_D$ outside of an arbitrarily small neighborhood of the facets of $\Sigma$. Moreover, $H$ can be constructed such that the directional derivatives of $F_D$ in the direction of the rays of the fan are preserved outside of arbitrarily small neighborhoods of the cones of the fan along which they are discontinuous.
\end{prop}
\begin{proof} Let $\eta \colon \bb{R}^n \to \bb{R}$ be a smooth mollifier function supported on the ball around the origin of radius $1$ and let $\eta_\eps (\mu) = \eps^{-n} \eta(\mu/\eps)$. Further suppose that $\eta$ is symmetric, i.e., $\eta(-\mu) = \eta(\mu)$. We claim that
\[ H (\mu) = (F_D * \eta_\eps) (\mu) = \int_{\bb{R}^n} F_D (\mu - x) \eta_\eps(x) \, dx \]
has the desired properties. Indeed, it easy to check that $H = F_D$ outside of an $\eps$-neighborhood of the facets of $\Sigma$ since $F_D$ is linear in the components of the complement of the facets and $\eta_\eps$ is symmetric. In addition, the convolution satisfies 
\[ \nabla H \cdot v = (\nabla F_D \cdot v) * \eta_\eps \]
for any $v \in \bb{R}^n$ so directional derivatives are preserved outside of $\eps$ neighborhoods of the cones where they are discontinuous by the same argument.
\end{proof}

When $D = D_\alpha$ for some primitive generator $\alpha$, the resulting twisting Hamiltonian $H$ can be understood geometrically. $H$ will be a smooth function whose level sets are tangent to the boundary faces of the star of $\alpha$ as is the case in Figure \ref{levelsets}. 

\begin{rem} \label{concave} When $F_D$ is a concave function, i.e., $D$ is basepoint free, we can guarantee that the smoothing $H$ is also concave by choosing $\eta_\eps$ to be nonnegative. If $F_D$ is strictly concave in the sense of Definition 6.1.12 of \cite{CLS} (except the notions of concave and convex are flipped in \cite{CLS}), which means that $F_D$ is concave and
\[ F_D(\mu) = m_\sigma \cdot \mu \]
if and only if $\mu \in \sigma$, then $D$ is ample and the smoothing of $H$ with $\eta_\eps$ nonnegative is strictly concave away from the domains of linearity of $F_D$ in the sense that 
\[ H(t \mu_1 + (1-t) \mu_2) > tH(\mu_1) + (1-t)H(\mu_2) \]
for all $t \in (0,1)$ unless $\mu_1$ and $\mu_2$ lie in the same maximal cone of $\Sigma$. 
\end{rem}

Note that our construction of a twisting Hamiltonian $H$ from $F_D$ using Proposition \ref{smooth} is only guaranteed to satisfy equation \eqref{twistcond} in the star of $\alpha$ (where $\nabla F_D \cdot \alpha = - 2\pi n_\alpha$) away from an arbitrarily small neighborhood of the boundary facets of the star. Thus, we obtain the following.

\begin{cor} \label{twistexist} If $\Delta$ is adapted to the fan $\Sigma$ of a smooth toric variety $X$, then twisting Hamiltonians exist for any toric divisor $D$ in $X$. 
\end{cor} 

\begin{rem} When $\Delta$ is not adapted to the fan, it is not possible to construct twisting Hamiltonians in general even when one can obtain a piecewise-linear approximation by examining the maximal overlaps. For example, there is no twisting Hamiltonian for the divisor $D_{(1,0)}$ for the tropical division on the right in Figure \ref{notadapted}. Any such function would have to be equal to $2\pi(-\mu_1 + 2\mu_2)$ on one boundary of  $C_{(1,0)}$ and equal to $2\pi(-\mu_1 + 3 \mu_2)$ on the other and vary from one boundary to the other in a way that depends only on $\mu_2$ in $C_{(1,0)}$. Comparing the boundary values along a segment parallel to the $\mu_1$-axis and recalling that $\nabla H \cdot (1,0) = - 2 \pi$ everywhere in $C_{(1,0)}$, we see that there is no such smooth function. In fact, the same argument will show that in this example there is no twisting Hamiltonian for any divisor with $n_{(1,0)} + n_{(2,1)} + n_{(-3,-1)} \neq 0$.
\end{rem}

\subsection{Hamiltonian isotopy classes of monomially admissible Lagrangian sections} \label{sectionclasses}

SYZ mirror symmetry predicts that Lagrangian sections of the torus fibration, which in our setting is the moment map on $(\bb{C}^*)^n$, should correspond to line bundles on $X$.  This is exactly the approach taken in \cite{Ab2} to construct mirror Lagrangians to line bundles on $X$, and we employ similar ideas here. Assuming that the monomial division is adapted to the fan, we show in this section that the Hamiltonian isotopy classes of monomially admissible Lagrangian sections (all of which are exact) are in bijection with line bundles on $X$ and that twisting Hamiltonians act as expected on the isotopy classes.

We will regard $(\bb{C}^*)^n$ as $T^* \bb{R}^n$ mod a fiberwise lattice. Under this identification, the moment map is projection to the base. In fact, we have a natural splitting $T^*\bb{R}^n = \bb{R}^n_{\mu} \times \bb{R}^n_\theta$ where $\mu = (\mu_1, \hdots, \mu_n)$ are the moment map coordinates and $\theta = (\theta_1, \hdots, \theta_n)$ are lifts of the arguments of $z_1, \hdots, z_n$. The fiberwise lattice is $2\pi \bb{Z}^n \subset \bb{R}^n_\theta$. 

 Any monomially admissible Lagrangian section $L$ lifts to a Lagrangian section of $T^*\bb{R}^n$. Further, any two lifts differ by a fixed element of $2 \pi \bb{Z}^n$ in each fiber. Also, any lift of $L$ over $C_\alpha$ satisfies $\alpha \cdot \theta|_L \in 2 \pi \bb{Z}$ at infinity where $\text{arg}(z^\alpha) = 0$. Thus, to any monomially admissible Lagrangian section and assuming that each $C_\alpha$ is unbounded, we can associate a function $\nu_L \colon A \to \bb{Z}$ defined up to an integral linear function on $\bb{R}^n$ by
 \begin{equation} \label{isotopyfunction} \nu_L (\alpha) = \frac{1}{2 \pi} \alpha \cdot \theta \end{equation}
for some $\theta$ such that $(\mu, \theta) \in L, \mu \in C_\alpha,$ and $\| \mu \| \gg 0$. 

\begin{prop} \label{bij} The Hamiltonian isotopy classes of admissible Lagrangian sections with respect to any monomial division $\Delta$ for $W_\Sigma$ adapted to $\Sigma$ are in bijection with $\text{Pic}(X)$. 
\end{prop}
\begin{proof} Since $X$ is smooth, a function $\nu \colon A \to \bb{Z}$ uniquely determines a piecewise linear integral function on $\bb{R}^n$ that is linear on the maximal cones of $\Sigma$. But, such a piecewise linear integral function on $\bb{R}^n$ is the support function for a unique toric divisor $\sum -\nu(\alpha) D_\alpha$ and two such divisors define the same line bundle if and only if their support functions differ by an integral linear function. Thus, we claim that the map $L \mapsto \nu_L$ gives the desired bijection. This map is well-defined since any Hamiltonian isotopy of monomially admissible Lagrangians lifts to a Hamiltonian isotopy preserving the function $\nu_L$ since this function is continuous and takes values in a discrete set. 

Now, for any support function $\nu$, we can construct a Lagrangian $\mathcal{L} \left( \sum - \nu(\alpha) D_\alpha \right)$ defined by applying the flow of the twisting Hamiltonian for $\sum - \nu(\alpha) D_\alpha$ to $\bb{R}_{>0}^n$. By first lifting $\bb{R}_{>0}^n$ to the zero section and following the isotopy induced by $H$, we get a lift of $\mathcal{L} \left( \sum - \nu(\alpha) D_\alpha \right)$ giving the desired support function.

Finally, suppose that $\nu_L$ and $\nu_{L'}$ differ by an integral linear function. That is, there exist lifts of $L$ and $L'$ such that the lifted arguments of $z^\alpha$ agree on all $C_\alpha$ for $|\mu(z)| \gg 0$. As exact Lagrangian sections of $T^* \bb{R}^n$, these lifts are of the form $df$ and $df'$. The convex isotopy 
\[ f_t = (1-t)f  + t f' \]
preserves monomial admissibility and induces the desired isotopy between $L$ and $L'$. 
\end{proof} 

\begin{rem} The lifts of $L$ correspond to equivariant structures on the mirror line bundle as in \cite{FLTZ2}, or equivalently a choice of toric divisor generating the line bundle, as these are in bijection with their support functions. Further, every lift of $L$ is of the form $dH$ for $H$ a twisting Hamiltonian for the divisor corresponding to the support function of the lift, and the Hamiltonian isotopy in the third paragraph of the proof of Proposition \ref{bij} is the flow of an admissible Hamiltonian.
\end{rem}

We have the following as an immediate corollary. 

\begin{cor} \label{action} The group of twisting Hamiltonians acts on the Hamiltonian isotopy classes of monomially admissible Lagrangian sections such that under the bijection of Proposition \ref{bij} a twisting Hamiltonian for a divisor $D$ acts as $(\cdot ) \otimes \mathcal{O}(D)$ when the monomial division is adapted to the fan.  
\end{cor}
\begin{proof} The twisting Hamiltonians act by their time-$1$ flows. Note that the time-$1$ flow of $H_1 + H_2$ is indeed the composition of the commuting time $1$-flows of $H_1$ and $H_2$ as twisting Hamiltonians only depend on the moment map coordinates.

In addition, the lifted time-$1$ flow $\phi^1_H$ of a twisting Hamiltonian $H$ for a divisor $\sum n_\alpha D_\alpha$  changes the lifted argument of $z^\alpha$ on a monomially admissible Lagrangian $L$ over $C_\alpha$ by $- 2 \pi n_\alpha$. Thus, if $\nu_L$ is a support function for a line bundle $V$ then $\nu_{\phi^1_H(L)}$ is a support function for ${V\otimes\mathcal{O}(\sum n_\alpha D_\alpha)}$ as desired.
\end{proof}

\begin{rem} \label{twistandpic} In our setting, every line bundle on $X$ is isomorphic to $\mathcal{O}(D)$ for some toric divisor $D$. Thus, Corollary \ref{action} shows that the group of twisting Hamiltonians surjects onto $\text{Pic}(X)$ with kernel the subgroup generated by admissible Hamiltonians and integral linear maps. This surjection factors as the map taking a twisting Hamiltonian $H$ onto the Lagrangian section that lifts to $dH$ followed by the bijection of Proposition \ref{bij}.
\end{rem}

In Section \ref{catmono}, we show that this action extends to an action by autoequivalences on the Fukaya-Seidel category of Lagrangian sections admissible with respect to a monomial division adapted to $\Sigma$ that corresponds under mirror symmetry to the action of $\text{Pic}(X)$ on the category of line bundles on $X$.

\subsection{Existence of monomial divisions adapted to the fan} \label{existencesect}

As we have now seen, the condition that a monomial division is adapted to the fan plays a crucial role in understanding monomially admissible Lagrangian sections. Thus, we would like to be able to answer the following question.

\begin{qs} \label{whencomb} Under what conditions on $c_\alpha, \Sigma$ and/or $\mu$ is there a monomial division for $W_\Sigma$ and $\mu$ adapted to $\Sigma$?
\end{qs}

\begin{rem} In Question \ref{whencomb}, we are not fixing the toric K\"{a}hler form, but do not list it in the parameters as its influence on the question only appears via its moment map $\mu$.
\end{rem}

In particular, we would like to know that there exist toric K\"{a}hler forms on $(\bb{C}^*)^n$ for which monomial divisions adapted to the fan exist. We will answer that more particular question, but we will leave Question \ref{whencomb} in full generality open. 

For fixed $\Sigma, \mu,$ and $c_\alpha$, the existence of a monomial division that is adapted to the fan is the same as the existence of a combinatorial division that is a monomial division due to the fact that any monomial division adapted to the fan is contained in a combinatorial division. Thus, Question \ref{whencomb} is equivalent to: under what conditions on $c_\alpha, \Sigma$ and/or $\mu$ is a combinatorial division a monomial division? The following proposition allows us to rephrase the question in a more concrete manner.

\begin{prop} \label{equivadapt} Fix $\Sigma, \mu$, and all $c_\alpha$. There exists a monomial division adapted to $\Sigma$ if and only if there exist constants $k_\alpha \in \bb{R}_{>0}$ for $\alpha \in A$ such that the maximum
\[ \max_{\alpha \in A} (|c_\alpha z^\alpha|^{k_\alpha}) \] 
is always achieved by $| c_\alpha z^\alpha |^{k_\alpha}$ for an $\alpha$ such that $\mu(z)$ is in the interior of the star of $\alpha$ outside of a compact subset of $(\bb{C}^*)^n$. 
\end{prop}
\begin{proof} The existence of such $k_\alpha$ in the presence of a monomial division adapted to the fan is automatic from the definitions. For the reverse implication, simply set $C_\beta$ to be the set where
\[ | c_\beta z^\beta|^{k_\beta} = \max_{\alpha \in A} (|c_\alpha z^\alpha|^{k_\alpha}) \]
for each $\beta \in A$. 
\end{proof}

As a first attempt, one may wish to set $k_\alpha = 1$ for all $\alpha \in A$. That is, try to show that the tropical division is adapted to the fan with $\delta = 0$. Thus, we are led to simultaneously consider the following similar question. 

\begin{qs}\label{whentrop} Under what conditions on $c_\alpha, \Sigma, \delta$ and/or $\mu$  is the tropical division adapted to $\Sigma$? 
\end{qs}

As with Question \ref{whencomb}, we will leave the general answer to Question \ref{whentrop} open. As we have already observed, there is never a tropical division adapted to the fan if $X$ is not semi-Fano due to the presence of bounded $C_\alpha$'s. In fact, $X$ must be Fano when $c_\alpha = 1$ for all $\alpha \in A$ as in the semi-Fano case there is always a $C_\alpha$ such that $C_\alpha \setminus \cup_{\beta \neq \alpha} C_\beta$ has empty interior. 

Even in the Fano setting, the examples in Figure \ref{notadapted} show that there is no simple answer to Question \ref{whentrop}. In the first example in Figure \ref{notadapted}, we can obtain a tropical division that is adapted to $\Sigma$ by simply changing the $c_\alpha$. For example, we can take $c_{(1,1)} = e^{-1}$ and all other $c_\alpha = 1$. The second example in Figure \ref{notadapted} cannot be made to be adapted to the fan by a change in the $c_\alpha$ as every line segment with zero slope in $C_{(1,0)}$ has finite length.  

\begin{figure} 
\centering
\begin{tikzpicture}

\draw[yellow, fill=yellow, opacity = 0.5] (-8,0) -- (-6.5, -3) -- (-11, -3) -- (-11, 1.5)  -- (-8,0);
\draw [blue, fill=blue, opacity=0.5] (-8,0) -- (-6.5, -3) -- (-5, -3) -- (-5,0) -- (-8,0);
\draw[red, fill=red, opacity = 0.5] (-8,0) -- (-11, 1.5) -- (-11, 3) -- (-8, 3) -- (-8,0);
\draw[green, fill=green, opacity = 0.5] (-8,0) -- (-8, 3) -- (-5, 3) -- (-5,0) -- (-8,0);
\draw[->, thick] (-8,0) --(-7,0);
\draw[->, thick] (-8,0) -- (-8,1);
\draw[->, thick] (-8,0) -- (-9,-1);
\draw[->, thick] (-8,0) -- (-7, 1);
\draw (-10, -2.5) node {$C_{(-1,-1)}$};
\draw (-5.6, -2) node {$C_{(1,0)}$};
\draw (-10, 2.6) node {$C_{(0,1)}$};
\draw (-6, 2) node {$C_{(1,1)}$};

\draw[yellow, fill=yellow, opacity = 0.5] (0,0) -- (-1.2, 3) -- (-3, 3) -- (-3, -3)  -- (0.75, -3) -- (0,0);
\draw [blue, fill=blue, opacity=0.5] (0,0) -- (3, -3) -- (0.75, -3) -- (0,0);
\draw[red, fill=red, opacity = 0.5] (0,0) -- (3, -3) -- (3, 3) -- (-1.2,3) -- (0,0);
\draw (-1.5, -2) node {$C_{(-3,-1)}$};
\draw (1.2, -2) node {$C_{(1,0)}$};
\draw (1, 2.3) node {$C_{(2,1)}$};
\draw[->, thick] (0,0) --(1,0);
\draw[->, thick] (0,0) -- (2,1);
\draw[->, thick] (0,0) -- (-3,-1);

\end{tikzpicture}
\caption{The tropical divisions for a blowup of $\bb{P}^2$ (left) and a nonstandard fan for $\bb{P}^2$ (right) with $\delta = 0$, $c_\alpha = 1$ for all $\alpha$, and $\mu = \Log$ show that the tropical division need not be adapted to the fan.} \label{notadapted}
\end{figure}
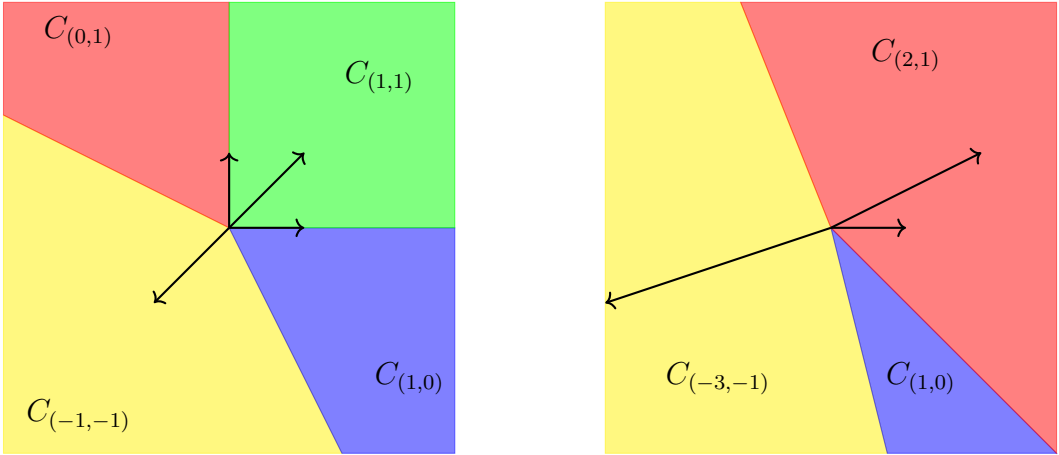 

The only other natural way to attempt to obtain a tropical division that is adapted to $\Sigma$ is via the choice of toric K\"{a}hler form and moment map. Note that a tropical division does not naturally vary with the choice of basis for the lattice $N$ when working with a fixed toric K\"{a}hler form as evidenced by the tropical divisions for $\bb{P}^2$ in Figures \ref{firstexample} and \ref{notadapted}. However, there is a choice of toric K\"{a}hler form which differs from the standard one by an affine linear map (the constant part can be absorbed in the choice of the origin for the moment map) and is invariant with respect to the choice of basis. Following a general principle that reduction should be mirror to restriction, we view our toric variety $X$ as a reduction of $\bb{C}^{N}$. Since a mirror to $\bb{C}^N$ is $( (\bb{C}^*)^N, w_1 + \hdots + w_N)$, we can view the mirror of $X$ as a restriction to $(\bb{C}^*)^n \hookrightarrow (\bb{C}^*)^N$ by the embedding
\[  (z_1, \hdots, z_n) \mapsto (c_{\alpha_1} z^{\alpha_1}, \hdots, c_{\alpha_N} z^{\alpha_N} )\]
where $\alpha_i$ is some ordering of the rays of the fan. Thus, we obtain a toric K\"{a}hler form by pulling back the standard symplectic form on $(\bb{C}^*)^N$. This construction works to produce a tropical admissibility condition which is adapted to $\Sigma$ for all toric Fano surfaces, but will not work in complete generality. Thus, we will eventually be led to more of a contrived construction.

First, let us return to Question \ref{whencomb}. In this case, we have additional freedom as we can still search for $k_\alpha$ as in Proposition \ref{equivadapt} when all other parameters are fixed. In particular, this will allow us to go beyond the Fano case and gives us hope that we can use the standard K\"{a}hler form with $\mu = \Log$. The following proposition does exactly that for surfaces. 

\begin{prop} \label{2dcomb} For toric surfaces ($n = 2$) with $\mu = \Log$ and all $c_\alpha$ nonzero, there is always a monomial division adapted to the fan.
\end{prop} 
\begin{proof}
In this setting, Proposition \ref{equivadapt} reduces the problem to finding positive constants $k_\alpha$ for $\alpha \in A$ such that 
\begin{equation} \label{2dcombeq} k_\alpha \Log(z) \cdot \alpha + \log |c_\alpha| > k_\beta \Log(z) \cdot \beta + \log |c_\beta| \end{equation}
when $\Log(z)$ lies on the ray generated by $\alpha$ for all distinct $\alpha, \beta \in A$ outside of a compact subset of $(\bb{C}^*)^n$. Finding such constants is always possible.  For instance, if we set 
\[ k_\alpha = \frac{ 1}{\| \alpha \| } \]
where $\| \cdot \|$ is the Euclidean norm and $\Log(z) = t\alpha$ for $t > 0$, then \eqref{2dcombeq} becomes
\[ t \left( \| \alpha \| - \frac{\alpha \cdot \beta}{\| \beta\|} \right) > \log| c_\beta | - \log | c_\alpha | \]
which holds for all $\beta \neq \alpha$ for large enough $t$ since the left-hand side is positive. 
\end{proof}

Figure \ref{hirzex} illustrates the importance of the choice of $k_\alpha$ in the proof of Proposition \ref{2dcomb}. It should be noted that the same choice of $k_\alpha$ will not work for higher dimensional fans which have narrow cones. For instance, in a $3$-dimensional cone generated by $\alpha, \beta,$ and $\gamma$, we always have
\[ \frac{v \cdot \gamma}{\| \gamma \|} \geq \max \left\{ \frac{v \cdot \alpha}{\| \alpha \|},  \frac{v \cdot \beta}{\| \beta \|} \right\} \]
for some $v \in \langle \alpha, \beta \rangle$ near infinity if $\alpha \cdot \gamma$ and $\beta \cdot \gamma$ are large enough relative to $\alpha \cdot \beta$. In particular, the argument cannot work when $X$ is not projective as Proposition \ref{adaptedgivesample} together with Proposition \ref{normalizecoeff} show that there is a monomial division adapted to the fan if and only if there is an embedding of the polytope of an ample divisor on $X$ into $\bb{R}^n$ such that the facet corresponding to $\alpha$ is contained in the open star of $\alpha$ for all $\alpha \in A$. The boundary of such a polytope is shown on the right-hand side of Figure \ref{hirzex} for the Hirzebruch surface $\bb{F}_3$. As there is no apparent formula or algorithm, we do not know in general if it possible to find the $k_\alpha$ for $\mu = \Log$. Thus, we will give ourselves additional flexibility by also allowing the toric K\"{a}hler form and moment map to vary.

\begin{figure}
\centering
\begin{tikzpicture}
\draw[->, thick] (0,0) -- (1, 3);
\draw[->, thick] (0,0) -- (0,1);
\draw[->, thick] (0,0) -- (0,-1);
\draw[->, thick] (0,0) -- (-1, 0);
\draw[ultra thick, blue] (4, -1) -- ( -1, 0.67) -- (-1, -1) -- (4, -1) ;

\draw[->, thick] (6,0) -- (7, 3);
\draw[->, thick] (6,0) -- (6,1);
\draw[->, thick] (6,0) -- (6,-1);
\draw[->, thick] (6,0) -- (5, 0);
\draw[ultra thick, blue] (5,1) -- (5, -1) -- ( 11.16, -1) -- (6.16, 1) -- (5,1) ;
\end{tikzpicture}
\caption{Level sets of $\max |z^\alpha|$ and $\max | z^\alpha |^{1/\| \alpha \| }$ under the Log projection on the fan of the corresponding Hirzeburch surface, $\bb{F}_3$.} \label{hirzex}
\end{figure}
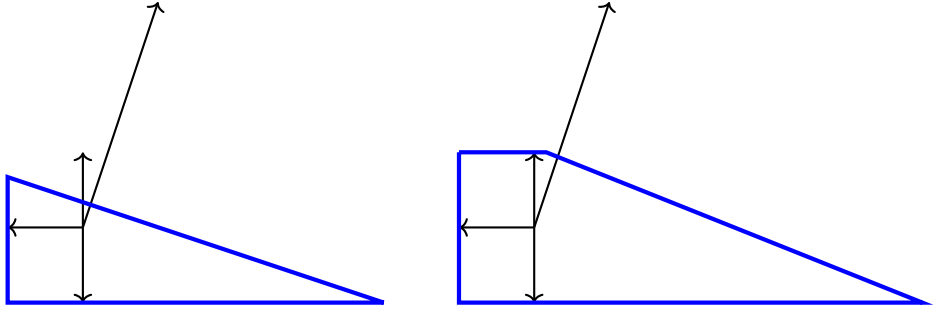

\begin{rem} Before proceeding to a more general setting, perhaps it is useful to briefly comment on why we might have expected to encounter Questions \ref{whencomb} and \ref{whentrop}. As we noted above for $\mu = \Log$, the existence of a monomial division adapted to the fan corresponds to finding an embedding of the polytope of an ample line bundle on $X$ into $\bb{R}^n$ for which the facet corresponding to each $\alpha \in A$ lies in the open star of $\alpha$. However, $\bb{R}^n$ is the base of our SYZ fibration and is naturally identified with $N_\bb{R}$ (see \cite{FLTZ2}). $N_\bb{R}$ contains the fan $\Sigma$ while the polytope naturally lives in a dual vector space. This is related to the fact that when doing SYZ mirror symmetry starting with $X$ as a symplectic manifold, one gets that the mirror is $W$ restricted only to the piece of $(\bb{C}^*)^n$ lying over a moment polytope of $X$. In that case, again some naturality issues arise as one then needs to perform some type of renormalization, which is often done via symplectic inflation along the boundary divisors, in order to make SYZ mirror symmetry involutive. This SYZ perspective also indicates that it is reasonable to expect that a solution involves choosing a nonstandard toric K\"{a}hler form.
\end{rem} 

In our search for monomial divisions adapted to the fan, we will restrict ourselves to a certain class of toric K\"{a}hler forms. Suppose that $\varphi \colon \bb{R}^n \to \bb{R}$ is a smooth strictly convex function with nondegenerate Hessian. Then, we have a toric K\"{a}hler form on $(\bb{C}^*)^n$ given by
\[ \omega_{\varphi} = \sum_{j} d \left( \frac{\partial \varphi}{\partial u_j} \right) \wedge d\theta_j .\]
Moreover, $\Phi \circ \Log$ is a moment map for $\omega_\varphi$ where $\Phi \colon \bb{R}^n \to \bb{R}^n$ is the Legendre transform of $\varphi$, i.e., the gradient of $\varphi$. Note that $\Phi$ is always injective by the strict convexity of $\varphi$ but need not be surjective in general. The lack of surjectivity is not an important geometric feature, and we will restrict our attention to $\varphi$ with bijective $\Phi$. Requirements on the behavior at infinity will be more important for obtaining geometric understanding of being adapted to the fan.

\begin{rem} \label{Kpotential} If $\omega$ is any toric K\"{a}hler form on $(\bb{C}^*)^n$ (with its standard complex structure), then we must have $\omega = \omega_\varphi$ for some smooth strictly convex $\varphi \colon \bb{R}^n \to \bb{R}$ as above.
\end{rem}

\begin{df} \label{radialdef} A toric K\"{a}hler form $\omega$ on $(\bb{C}^*)^n$ is radial if $\omega = \omega_\varphi$ for a smooth strictly convex function $\varphi \colon \bb{R}^n \to \bb{R}$ with nondegenerate Hessian such that the Legendre transform $\Phi$ of $\varphi$ is a bijection that takes rays to rays outside of a compact set. 
\end{df}

The standard symplectic form on $(\bb{C}^*)^n$ is radial and has $\varphi(u) = \frac{1}{2} (u \cdot u)$. The pullback form from $(\bb{C}^*)^N$ discussed earlier in this section is also radial with $\varphi (u) =  \sum_{\alpha \in A} (\alpha \cdot u)^2$. In line with these examples, we will mostly keep in mind the class of radial K\"{a}hler forms where $\varphi$ is a homogeneous function of degree $d> 1$ outside of a compact set. For such forms, $\Phi$ is homogeneous of degree $d-1$ and hence sends rays to rays. In this setting, $\Phi$ is automatically a bijection as shown in Proposition \ref{convsurj} below. Any smooth homogeneous function on $\bb{R}^n$ must be a polynomial as a consequence of Euler's homogeneous function theorem. However, we will need to work with a wider class of functions, which the condition that $\Phi$ takes rays to rays only in the complement of a compact subset allows us. 

\begin{prop} \label{convsurj} Suppose that $\varphi \colon \bb{R}^n \to \bb{R}$ is a strictly convex function that is homogeneous of degree $d > 1$ outside of a compact set.  Then, its Legendre transform $\Phi$ is surjective. 
\end{prop}
\begin{proof} Take a sphere $S_R$ of large radius $R$ that contains the region where $\varphi$ is not homogeneous of degree $d$ in its interior. Consider the function $f \colon S_R \to S_1$ given by
\[ f(u) = \frac{\Phi(u)}{\| \Phi(u) \|} \]
which is well-defined due to the strict convexity of $\varphi$. The function $f$ is an embedding as 
\[\frac{\Phi(u)}{\| \Phi(u) \|} = \frac{\Phi(v)}{\| \Phi(v) \|} \]
for $u \neq v$ contradicts the injectivity of $\Phi$  due to the homogeneity of $\Phi$ outside of $S_R$.
Thus, $f$ must be surjective as there is no embedding of $S_R$ into $\bb{R}^{n-1}$. It follows that the complement of the image of $\Phi$ is compact. However, this implies that the complement must be empty as $\Phi$ is degree $1$ near infinity.
\end{proof}

As a result of the surjectivity assumption in Definition \ref{radialdef}, we will be using the characterization of adapted to the fan in Remark \ref{ontomoment} implicitly. We now proceed to studying monomial divisions adapted to the fan for this class of forms. First, we show that the choice of coefficients $c_\alpha$ is not relevant to the existence question.

\begin{prop} \label{normalizecoeff} Suppose that $\omega_\varphi$ is a radial toric K\"{a}hler form with moment map $ \mu = \Phi \circ \Log$. There exists a monomial division adapted to the fan for 
\[ W = \sum_{\alpha \in A} c_\alpha z^\alpha \]
with all $c_\alpha$ nonzero if and only if there exists a monomial division adapted to the fan for 
 \[ W^1 = \sum_{\alpha \in A} z^\alpha. \]
\end{prop}
\begin{proof} We will use the characterization of Proposition \ref{equivadapt}. 

$(\Rightarrow)$ Consider the function $g(z) = \max_{\alpha \in A}(|c_\alpha z^\alpha|^{k_\alpha})$ that provides a monomial division adapted to the fan for $W$. The set $Q = \{ \Log(z) \colon g(z) \leq e^B \}$ is a convex polytope in $\bb{R}^n$ with boundary facets $F_\alpha$ lying in the planes
\[ u \cdot \alpha = \frac{B}{k_\alpha} - \log(|c_\alpha |)  \]
 for each $\alpha \in A$.
 
 By assumption, we have $\Phi(F_\alpha) \subset \text{int(star}(\alpha))$ for all $\alpha \in A$ if $B$ is large enough. Now, consider the function $h(z) = \max_{\alpha \in A} (|z^\alpha|^{K_\alpha})$ where 
 \[ \frac{1}{K_\alpha} = \frac{B}{k_\alpha} - \log(|c_\alpha |) .\] 
 Since $\Phi$ takes rays to rays, we obtain that outside of a compact subset the image under $\mu$ of the region where $h(z) = | z^\alpha |^{K_\alpha}$ is the cone on $\Phi(F_\alpha)$ and is thus contained in the interior of the star of $\alpha$ for all $\alpha \in A$. Therefore, $h(z)$ guarantees a monomial division adapted to the fan for $W^1$. 
 
$(\Leftarrow)$ Suppose that $g(z) = \max_{\alpha \in A}(| z^\alpha|^{k_\alpha})$ gives a monomial division adapted to the fan for $W^1$. For each $\alpha \in A$, the cone $V_\alpha$ given by $k_\alpha(u \cdot \alpha) \geq k_\beta(u \cdot \beta)$ for all $\beta \in A$ satisfies $\Phi(V_\alpha \setminus K) \subset \text{int(star(}\alpha))$ for some compact subset $K$. As a result, $V_\alpha$ has non-empty interior. In particular, there is a vector $u_\alpha$ such that 
\[ k_\alpha(u_\alpha \cdot \alpha) > k_\beta (u_\alpha \cdot \beta) \]
for all $\beta \in A$. By rescaling if necessary, we can assume that 
\[ k_\alpha(u_\alpha \cdot \alpha) > k_\beta (u_\alpha \cdot \beta) + k_\alpha \log |c_\alpha| - k_\beta \log |c_\beta| \]
for all $\beta \in A$. Now consider the set $V_\alpha'$ defined by $k_\alpha ( u \cdot \alpha ) + k_\alpha \log | c_\alpha | \geq k_\beta(u \cdot \beta) + k_\beta \log|c_\beta| $. We see that if $u \in V_\alpha'$, then 
\[ k_\alpha(u + u_\alpha) \cdot \alpha \geq k_\beta(u + u_\alpha) \cdot \beta \]
for all $\beta \in A$, that is, $u + u_\alpha \in V_\alpha$. As a consequence, any $u \in V_\alpha'$ is within Euclidean distance $\|u_\alpha \|$ of $V_\alpha$. However, we can guarantee that outside of a compact set the distance from $V_\alpha$ to the boundary of the conical set $\Phi^{-1}(\text{star}(\alpha) )$ is larger than $\| u_\alpha \|$. It follows that $\Phi(V_\alpha') \subset \text{int(star(}\alpha))$ as desired. 
\end{proof}

As a result of Proposition \ref{normalizecoeff}, we can focus on studying the existence of a monomial division adapted to the fan when $c_\alpha = 1$ for all $\alpha \in A$. Also, the condition
\[ \mu( \{ z \, | \, |z^\alpha|^{k_\alpha} \geq |z^\beta|^{k_\beta} \text{ for all } \beta \in A \} ) \subset \text{int(star(}\alpha)) \]
is an open condition on the $k_\alpha$. Thus, we can always assume that all $k_\alpha$ are rational when a monomial division adapted to the fan exists. The following proposition gives us the key geometric data for understanding a monomial division adapted to the fan.

\begin{prop} \label{adaptedgivesample} Suppose that $g(z) = \max |z^\alpha|^{k_\alpha}$ with $k_\alpha \in \bb{Q}$ gives a monomial division adapted to the fan for
\[ W^1_\Sigma = \sum_{\alpha \in A} z^\alpha \]
and $\mu = \Phi \circ \Log$ for a radial toric K\"{a}hler form. If $N/k_\alpha \in \bb{N}$ for all $\alpha \in A$ and $N \gg 0$, then 
\[ Q = \{ \Log(z) \, | \, g(z) \leq e^N \} \]
is the polytope of an ample line bundle on $X$. 
\end{prop}
\begin{proof} $Q$ is the polytope of the line bundle $\mathcal{O} \left( \sum_{\alpha \in A} \frac{N}{k_\alpha}D_\alpha \right)$. To show that it is ample, it is enough to show that the normal fan to $Q$ is $\Sigma$. First, there is a possible facet $F_\alpha$ of $Q$ for each $\alpha \in A$ where $u \cdot \alpha = N/k_\alpha$. Thus, the rays of the normal fan are generated by a subset of $A$. 

Since $\Phi(F_\alpha) \subset \text{int(star(} \alpha))$ for all $\alpha \in A$, we have that $F_{\alpha_1} \cap \hdots \cap F_{\alpha_m} \neq \emptyset$ implies that $\langle \alpha_1, \hdots, \alpha_m \rangle$ is a cone in $\Sigma$. Therefore, the set of cones of the normal fan is a subset of the cones in $\Sigma$. However, both are complete fans and therefore must coincide.  
\end{proof}

The polytope $Q$ for a monomial division adapted to the fan for a Hirzebruch surface is shown on the right-hand side of Figure \ref{hirzex}. We immediately obtain the following corollary, which we mentioned earlier.

\begin{cor} If there is a monomial division adapted to the fan for $W_\Sigma$ with a radial toric K\"{a}hler form, then $X$ is projective.
\end{cor}

Proposition \ref{adaptedgivesample} will allow us to relate the existence of a monomial division adapted to the fan to properties of a polytope. Eventually, we will use the construction of toric K\"{a}hler forms from \cite{Zhou}, which are introduced to solve the related issue of showing that the Lagrangian skeleton from \cite{FLTZ1, FLTZ2, FLTZ3, FLTZS} is indeed a relative Lagrangian skeleton for $( (\bb{C}^*)^n, W_\Sigma)$. Let us now proceed in that direction. The following definition is a slight modification of Definition 2.8 from \cite{Zhou}.

\begin{df} \label{adaptedpot} Let $P \subset \bb{R}^n$ be a convex polytope containing the origin. A strictly convex function $\varphi$ is adapted to $P$ if $\varphi$ is homogeneous outside of a compact set in the interior of $P$ and has a unique minimum in the interior of each face of $P$ with positive dimension and codimension.
\end{df}

In fact, Definition \ref{adaptedpot} can be rephrased in terms of $\Phi$.

\begin{prop} A strictly convex function $\varphi$ is adapted to $P$ if and only if $\varphi$ is homogeneous outside of a compact set in the interior of $P$ and $\Phi$ takes some point in the interior of each face of positive codimension to the normal cone to that face. 
\end{prop}
\begin{proof} A critical point of $\varphi$ on the interior of a face must be a critical point of $\varphi$ restricted to the affine subspace generated by the vectors in the face. Thus, the critical point must be a unique minimum due to the strict convexity of $\varphi$. Let $x$ be such a minimum. Suppose that the face is given by $F_{\alpha_1} \cap \hdots \cap F_{\alpha_m}$ where $F_{\alpha_i}$ is a facet with outward pointing normal vector $\alpha_i$. By Lagrange multipliers, it follows that 
\[ \Phi(x) = \lambda_1 \alpha_1 + \hdots + \lambda_m \alpha_m \]
for $\lambda_i \in \bb{R}$.  It remains to show that $\lambda_i \geq 0$ for all $i \in \{1, \hdots, m \}$. Without loss of generality, we will show that $\lambda_1 \geq 0$. Consider the restriction $\wt \varphi$ of $\varphi$ to the affine subspace generated by the vectors in the face $F_{\alpha_2} \cap \hdots \cap F_{\alpha_m}$ or just $\wt \varphi = \varphi$ on $\bb{R}^n$ if $m = 1$. When $m = 1$, $\varphi$ must have a minimum on the interior of $P$ since it is strictly convex and homogeneous outside of $P$. Otherwise, $\wt \varphi$ has a minimum in the interior of $F_{\alpha_2} \cap \hdots \cap F_{\alpha_m}$ by assumption. Thus, the sublevel sets of $\wt \varphi$ are bounded convex sets that first meet the affine subspace generated by vectors in $F_{\alpha_1}$ at $x$. The gradient of $\wt \varphi$ must point outwards from these level sets and hence from $F_{\alpha_1}$. That is, $\lambda_1 \geq 0$. 
\end{proof}

We would like to characterize being adapted to the fan in a similar manner. 

\begin{prop} \label{amplegivesadapted} Suppose that $\omega = \omega_\varphi$ is radial and $P$ is a convex polytope with facets $F_\alpha$ that satisfy $u \cdot \alpha = h_\alpha$. Suppose further that $P$ contains the compact subset outside of which $\Phi$ sends rays to rays in its interior. If each cone $\langle \alpha_1, \hdots, \alpha_m \rangle$ of the normal fan intersects $\Phi(P)$ only in $\Phi(F_{\alpha_1}) \cup \hdots \cup \Phi(F_{\alpha_m})$ then there is a monomial division adapted to the fan for $W= \sum c_\alpha z^\alpha$ and $\mu = \Phi \circ \Log$ with $k_\alpha = 1/h_\alpha$. 
\end{prop}
\begin{proof} By the proof of Proposition \ref{normalizecoeff}, we can assume that $W = W^1$.  Now, consider $g(z) = \max | z^\alpha |^{1/h_\alpha}$. Then, $\{ \Log(z) \, | \, g(z) \leq e \}$ is equal to $P$. Thus, the moment map image of the region where $g(z) = | z^\alpha |^{1/h_\alpha}$ is the cone on $\Phi(F_\alpha)$ outside of a compact set. By assumption, a cone $\sigma$ of the normal fan, which is complete, can only intersect $\Phi(F_\alpha)$ if $\alpha \in \sigma$. Therefore, $\Phi(F_\alpha) \subset \text{int(star(}\alpha))$ giving the desired result.
\end{proof}

We also have a converse to Proposition \ref{amplegivesadapted}. 

\begin{prop} \label{adaptedgivesample2} The polytope $Q$ from Proposition \ref{adaptedgivesample} satisfies the property that a cone $\langle \alpha_1, \hdots, \alpha_m \rangle$ of $\Sigma$ intersects $\Phi(Q)$ only in $\Phi(F_{\alpha_1}) \cup \hdots \cup \Phi(F_{\alpha_m})$. 
\end{prop}
\begin{proof} We have that $\text{int(star(}\beta)) \cap \langle \alpha_1, \hdots, \alpha_m \rangle \neq \emptyset$ if and only if $\beta = \alpha_i$ for some $i \in \{ 1, \hdots, m \}$. But, $\Phi(F_\beta) \subset \text{int(star(}\beta))$ for all $\beta \in A$ by assumption. Therefore, $\Phi(F_\beta) \cap \langle \alpha_1, \hdots, \alpha_m \rangle \neq \emptyset$ implies that $\beta = \alpha_i$ for some $i \in \{1, \hdots, m\}$. 
\end{proof}

Thus, we see that $W_\Sigma$ admitting a monomial division adapted to the fan with $\omega_\varphi$ and $\mu = \Phi \circ \Log$ is characterized by a similar condition to $\varphi$ being adapted to the polytope of some ample line bundle on $X$ when $\varphi$ is homogeneous outside of a compact set. In fact, the two notions are equivalent in dimension $2$ and when $\Phi$ is linear, but the existence of a monomial division adapted to the fan is stronger in general. The fact that the two coincide when $\Phi$ is linear shows in particular that they are equivalent for $\mu = \Log$. Thus, the condition that there exists an ample line bundle on $X$ to which $\phi(u) = u \cdot u$ is adapted required for the construction in \cite{GS}  and left as an open problem in \cite{Zhou} is equivalent to the existence of a monomial division adapted to the fan with $\mu = \Log$, which we also leave open.

Now, we will move on to the construction of radial toric K\"{a}hler forms that admit a monomial division adapted to the fan. We follow the ideas in \cite{Zhou}, but somewhat modify and simplify the construction given there.

\begin{prop} \label{rhoexists} For any convex polytope $P$ containing the origin, there exists a smooth strictly convex function $\varphi$ with nondegenerate Hessian such that $\varphi$ is homogeneous of degree $2$ outside of a compact set in the interior of $P$ and its Legendre transform $\Phi$ satisfies that each cone $\langle \alpha_1, \hdots, \alpha_m \rangle$ of the normal fan to $P$ intersects $\Phi(P)$ only in $\Phi(F_{\alpha_1}) \cup \hdots \cup \Phi(F_{\alpha_m})$.
\end{prop}
\begin{proof} By the proof of Proposition \ref{adaptedgivesample2}, it is enough to produce a strictly convex $\varphi$ with nondegenerate Hessian such that $\varphi$ is homogeneous of degree $2$ outside of a compact set in the interior of $P$ and has Legendre transform satisfying $\Phi(F_\alpha) \subset \text{int(star(}\alpha))$ for each ray $\langle \alpha \rangle$ of the normal fan. 

Suppose that each face $F_\alpha$ of $P$ is given by the equation $u \cdot \alpha = h_\alpha$. Note that all $h_\alpha$ are positive since the origin is in the interior of $P$. Let $\varphi_P$ be the continuous piecewise linear function given by $\varphi_P (u) = \frac{1}{h_\alpha} (\alpha \cdot u)$ when $u \in \text{cone}(F_\alpha)$ where $\text{cone}(F_\alpha)$ is the cone on $F_\alpha$. It follows from Lemma 6.1.5(d) of \cite{CLS} that $\varphi_P$ is convex.

If $\eta$ is a nonnegative and smooth mollifier function on $\bb{R}^n$ supported on the ball around the origin of radius $1$ and $\eta_\eps (x) = \eta (x/\eps)/\eps^n$, then $\varphi_\eps = \eta_\eps * \varphi_P$ is a smooth convex function. Suppose that $u \in \text{cone}(F_{\alpha})\setminus B_\eps (0)$ where $\text{cone}(F_{\alpha})$ is the cone on $F_\alpha$ and $B_\eps(0)$ is the ball of radius $\eps$ around $0$. If $u$ is in addition contained in the complement of the $\eps$-neighborhood of the boundary of $\text{cone}(F_{\alpha})$, then
\[ \nabla \varphi_\eps (u) = \int_{\bb{R}^n} \nabla \varphi_P( u- x) \eta_\eps(x) \, dx = \frac{1}{h_\alpha} \alpha \]
is in the interior of the star of $\alpha$. Now, assume instead that $u$ is in the $\eps$-neighborhood of the boundary of $\text{cone}(F_{\alpha})$. There is a $\rho$ such that $\| u \| \geq \rho > \eps$ implies $B_\eps (u)$ can only meet faces $F_{\beta_1}, \hdots, F_{\beta_m}$ such that $\langle \alpha, \beta_1, \hdots, \beta_m \rangle$ is a cone of the normal fan to $P$. Then, $\nabla \varphi_\eps (u)$ is a convex combination of $\alpha, \beta_1, \hdots, \beta_m$ and the coefficient on $\alpha$ is nonzero. Since $\langle \alpha, \beta_1, \hdots, \beta_m \rangle$ is a cone of the normal fan, any vector of this form lies in the interior of the star of $\alpha$.

Therefore, we have seen that if $u \in \text{cone}(F_{\alpha})\setminus B_\rho (0)$, then $\nabla \varphi_\eps (u) \in \text{int(star(}\alpha))$. Note that we can take $\rho \to 0$ as $\eps \to 0$ so we can assume in particular that $\nabla \varphi_\eps (u) \in \text{int(star(}\alpha))$ if $\varphi_\eps(u) = 1$.

Let $\wt \varphi_\eps$ be the homogeneous degree $2$ function whose $1$-level set coincides with that of $\varphi_{\eps}$. The function $\wt \varphi_\eps$ exists by the convexity of the $1$-sublevel set, which contains the origin. Since any nonnegative homogenous function of degree $d > 1$ with smooth level sets and convex sublevel sets is convex and smooth away from the origin, $\wt \varphi_\eps$ is a convex function that is smooth away from the origin. Suppose that $u \in \text{cone}(F_{\alpha})\setminus B_\eps (0)$ and $\varphi_\eps(u) =1$. Then, $\wt \Phi_\eps (u) = \nabla \wt \varphi_\eps (u) \in \text{int(star(}\alpha))$ for $\eps$ small enough since $\nabla \wt \varphi_\eps (u) = \lambda \nabla \phi_{\eps} (u)$ for $\lambda > 0$. Since $\nabla \wt \varphi_\eps$ is homogeneous of degree $1$, it follows that $\wt \Phi_\eps(u) \in \text{int(star(}\alpha))$ for all $u \in \text{cone}(F_\alpha)$ for small $\eps$. In particular, $\wt \Phi_\eps(F_\alpha) \subset \text{int(star(}\alpha))$ for all $\alpha \in A$ and small $\eps$. 

Now, consider the function $\varphi_{\eps_1, \eps_2} (u) = \wt \varphi_{\eps_1}(u) + \eps_2 \| u \|^2$ with $\eps_1, \eps_2 > 0$, which is homogeneous of degree $2$. Further, $\varphi_{\eps_1, \eps_2}$ is strictly convex and has nondegenerate Hessian. Since $\varphi_{\eps_1, \eps_2}$ is $C^\infty$ close to $\wt \varphi_{\eps_1}$ for $\eps_2$ small away from the origin, we have $\Phi_{\eps_1, \eps_2} (F_\alpha) \subset \text{int(star(}\alpha))$ for all $\alpha \in A$ and small $\eps_1$ and $\eps_2$.

Smoothing $\varphi_{\eps_1, \eps_2}$ in a small neighborhood of the origin, as can be done preserving strict convexity and nondegeneracy of the Hessian by \cite{Ghomi}, gives us the desired function $\varphi$.
\end{proof}

\begin{rem} In understanding the above proof, it may be helpful to observe that $\varphi_\eps$ is still linear and equal $\varphi_P$ away from the $\eps$-neighborhood of the boundaries of the domains of linearity if $\eta$ is in addition chosen to be symmetric. In particular, the $1$-level set of $\phi_\eps$ is a smoothing of $\partial P$ near the edges and corners and the $1$-level set of $\phi_{\eps_1, \eps_2}$ is a perturbation to make the sublevel set strictly convex.
\end{rem}

\begin{rem} The degree of homogeneity in Proposition \ref{rhoexists} is not important in the argument. We could have taken $\varphi$ to have any degree $d > 1$. 
\end{rem}

Combining Proposition \ref{rhoexists} with Proposition \ref{amplegivesadapted} immediately gives the following two corollaries which were the goal of this section.

\begin{cor} \label{adaptedexists} For any projective toric variety with fan $\Sigma$, there is a radial toric K\"{a}hler form on $(\bb{C}^*)^n$ for which $W_\Sigma$, with all $c_\alpha$ nonzero, admits a monomial division adapted to the fan with $\mu = \Phi \circ \Log$. 
\end{cor}

\begin{cor} \label{tropicalfanoadapts} For any Fano toric variety with fan $\Sigma$, there is a radial toric K\"{a}hler form on $(\bb{C}^*)^n$ for which $W_\Sigma$, with all $c_\alpha$ nonzero, admits a monomial division adapted to the fan with $\mu = \Phi \circ \Log$ and $k_\alpha = 1$ for all $\alpha \in A$. That is, any tropical division for $W_\Sigma$ with small $\delta$ is adapted to the fan with $\mu = \Phi \circ \Log$. 
\end{cor} 

\section{The Fukaya-Seidel category of a monomial division} \label{categories} \label{3}

The goal of this section is to define and study an $A_\infty$-category $\mathcal{F}_\Delta(W)$ with respect to a monomial division $\Delta$ for a Laurent polynomial $W$.  The setup is similar to that of Seidel \cite{SeBook, Lef1} except using that our Lagrangians project to $\bb{R}_{>0}$ by $c_\alpha z^\alpha$ in $C_\alpha$ rather than projecting everywhere to $\bb{R}_{>0}$ under $W$. Roughly, the objects of $\mathcal{F}_\Delta( W)$ are monomially admissible exact Lagrangians with extra data for gradings and orientations. Morphisms are defined by Floer cochain complexes after increasing the arguments of each monomial $c_\alpha z^\alpha$ of $W$ in its corresponding subset in the division, $C_\alpha$, near infinity on the source Lagrangian. Since it is technically difficult to implement the entire $A_\infty$ structure directly in this setup, we will use the localization approach of Abouzaid-Seidel in \cite{AbSloc}, which has been used to various extents in other recent works \cite{AbSKhov, AuSpec, GPS1, Ke}. 

\subsection{Bounding pseudohomolorphic discs} \label{control}

The first step in defining $\mathcal{F}_\Delta(W)$ is to show that the Floer theory of monomially admissible Lagrangians is well-defined. For that, the only special aspect of our setup is to show that pseudoholomorphic discs with boundary on monomially admissible Lagrangians are contained in a compact subset of $(\bb{C}^*)^n$.

Let $J$ be an $\omega$-compatible almost complex structure on $(\bb{C}^*)^n$ that agrees with the standard complex structure outside of a compact subset of $(\bb{C}^*)^n$ or more generally a monomially admissible complex structure defined as follows. 
\begin{df} 
An almost complex structure $J$ on $(\bb{C}^*)^n$ is admissible with respect to $\Delta$, or \textit{monomially admissible}, if for each $\alpha \in A$,  $z^\alpha$ is a $J$-holomorphic function in the complement of a compact subset of $\mu^{-1}(C_\alpha)$ .
\end{df}

All statements in this section also hold if $J$ is domain-dependent. Now, let $L_0, \hdots, L_k$ be exact, pairwise transverse Lagrangians such that $L_j$ is monomially admissible with respect to the monomial division $\Delta$ but for the function $W_j= e^{-i\theta_j} W$ with $ \theta_0 \neq \theta_1 \neq \hdots \neq \theta_k $. Let $p_i \in L_{i-1} \cap L_i$ for $i = 1, \hdots, k$ and $q \in L_0 \cap L_k$. Let $S$ be the closed unit disc with $k+1$ boundary punctures $z_0, z_1, \hdots, z_{k}$ ordered counterclockwise. 

\begin{prop} \label{comp1} There exists a compact subset of $(\bb{C}^*)^n$ that contains the image of any $J$-holomorphic map $u\colon S \to (\bb{C}^*)^n$ such that the image of the boundary region between $z_i$ and $z_{i+1}$ lies on $L_i$ for $i = 0, \hdots, k$ with $z_{k+1} = z_0$ and such that $u$ extends to a continuous map on the closed unit disc $D$ with $u(z_i) = p_i$ for $i = 1, \hdots, k$ and $u(z_0) = q$.
\end{prop}
\begin{proof} Let  $u$ be such a map. Consider a function 
\[ g(z) = \max_\alpha ( | c_\alpha z^\alpha |^{k_\alpha} ) \] 
satisfying the second condition of Definition \ref{div}. First, note that any local maximum of $g \circ u$ must be a local maximum of one of the $| c_\alpha z^\alpha| \circ u$. Hence, there are no local maxima of $g \circ u$ in the interior of the disc by the maximum modulus principle outside of the union of the compact subsets of each $C_\alpha$ where the functions $c_\alpha z^\alpha$ are not guaranteed to be holomorphic.

Further, the second condition of Definition \ref{div} implies that outside of a compact subset of each $C_\alpha$ the boundary points of the disc are mapped by $c_\alpha z^\alpha \circ u$ to a disjoint union of rays from the origin. But then, we can apply the Schwarz reflection principle to any of these boundary points to holomorphically extend $c_\alpha z^\alpha \circ u$ and conclude as above that $|c_\alpha z^\alpha| \circ u$ has no local maxima at these boundary points.  

We have obtained that after extending $u$ continuously to the boundary punctures the maximum of $g \circ u$ is achieved in a fixed compact subset $K$ of $(\bb{C}^*)^n$. Thus, the image of $u$ is contained in the compact subset $g^{-1} ([0, M])$ where $M$ is the maximum value of $g$ on $K$. 
\end{proof}

\begin{rem} \label{energy1} It will be useful to note for later that any map $u$ as in Proposition \ref{comp1} has an a priori energy bound by Stokes' theorem due to the exactness of the Lagrangians $L_0, \hdots, L_k$ when the almost complex structure is compatible with $\omega$. 
\end{rem}

\begin{rem} \label{conical} For unbounded Lagrangians inside of a Liouville domain, it is more common to argue that the images of holomorphic discs with boundary on these Lagrangians are contained in a compact set via a maximum principle for the radial coordinate in the convex ends as in \cite{Vi}. Although the proof of Proposition \ref{comp1} is similar in principle, the standard approach will not work as it requires the Lagrangians to be conical and Legendrian at infinity. Recall that we construct Lagrangian sections mirror to line bundles as the time-$1$ image of the real positive locus under the flow of a twisting Hamiltonian $H$ from Section \ref{twistham}. These Lagrangians are certainly conical but the condition that they are Legendrian with respect to $\lambda = \sum \mu_i \, d\theta_i$ is 
\[ \sum_{i=1}^n \mu_i \frac{\partial^2 H}{\partial \mu_i \partial \mu_j} = 0 \]
for all $j = 1, \hdots, n$ or more concretely, $Z(\nabla H) = 0$ where $Z$ is the Liouville vector field. Since $Z$ is simply the radial vector field on $\bb{R}^n$, the condition implies that $\nabla H$ is radially invariant. However, this would imply that $H$ is smoothed in conical regions, which cannot be done preserving monomial admissibility.
\end{rem}

In order to be able to push Lagrangians into the correct position and define all morphisms in $\mathcal{F}_\Delta (W)$ via localization, we will need to work in the more general setting of perturbed holomorphic maps. Let $H(\mu_1, \hdots, \mu_n)$ be a Hamiltonian on $(\bb{C}^*)^n$ that depends only on the moment map coordinates. Let $f\colon (\bb{C}^*)^n \to \bb{C}$ be a Laurent monomial $f(z) = z^\alpha = z_1^{\alpha_1} \hdots z_n^{\alpha_n}$ and $u\colon S \to (\bb{C}^*)^n$ be a map satisfying 
\begin{equation} \label{perturb} (du - X_H \otimes \beta)^{0,1}_J = 0 \end{equation}
where $\beta$ is a real one-form on $S$. Assume that $\nabla H \cdot \alpha$ is constant. Let $v = f \circ u \colon S \to \bb{C}$ and write $v = v_1 + iv_2$. The approach below is again a variation on older techniques (see for example Lemma 4.2 of \cite{Lef1}).

\begin{prop} \label{monmax} The function $|v|$ has no local maxima in the interior of $S$ if $f$ is holomorphic and ${(\nabla H \cdot \alpha) d\beta \leq 0}$. Further, $|v|$ has no local maxima on the boundary under the additional assumptions that $\beta|_{\partial S} = 0$ and $v$ maps each boundary component to a ray from the origin.
\end{prop}
\begin{proof}
In order to establish an interior maximum principle for $|v|$, we will compute $\Delta |v|^2$. Let $z = s + it$ be a local holomorphic coordinate on $S$. Note that
\[ (dv)^{0,1} = dv + i dv i = \frac{\partial v}{\partial s} ds  + \frac{\partial v}{\partial t} dt + i \left( \frac{\partial v}{\partial s} ds \, i + \frac{\partial v}{\partial t} dt \, i \right) \]
\[ = \left( \frac{\partial v}{\partial s} + i \frac{\partial v}{\partial t} \right) ds + \left( \frac{\partial v}{\partial t} - i \frac{\partial v}{\partial s} \right) dt = \bar \partial v \, ds - i \bar \partial v \, dt \]
since $ds \, i = - dt$ and $dt \, i = ds$. Thus, 
\begin{equation} \label{delbar} \bar \partial v = (dv)^{0,1} (\partial_s ) .\end{equation}
Now, observe that 
\[ 0 = (du - X_H \otimes \beta)^{0,1}  = du - X_H \otimes \beta + J du i - J X_H \otimes \beta i \] 
implies that 
\[ du i = X_H \otimes \beta i + J du - JX_H \otimes \beta .\]
Thus, 
\[ (dv)^{0,1} = df \, du + i df \, du i = df \, du + i df ( X_H \otimes \beta i + J du - JX_H \otimes \beta) \]
\[ df(JX_H) \otimes \beta i + df(X_H) \otimes \beta \]
\[ = \Big( df(X_H) \beta(\partial_s) + df(JX_H) \beta(\partial_t) \Big)ds + \Big(df(X_H) \beta(\partial_t) - df(JX_H) \beta(\partial_s) \Big) dt \]
which combined with \eqref{delbar} gives
\begin{equation} \label{delbar2} \bar \partial v = df(X_H) \beta(\partial_s) + i df(X_H) \beta(\partial_t). \end{equation}
Note that in coordinates $(\mu_1, \theta_1, \hdots, \mu_n, \theta_n)$ on $(\bb{C}^*)^n$ and coordinates $( r, \theta)$ on $\bb{C}$, we have 
\[ df = \left( \begin{matrix} * \ 0 \ \hdots \ * \ 0 \\ 0 \ \alpha_1 \ \hdots \ 0 \ \alpha_n \end{matrix} \right) .\]
Thus, $df(X_H) = (\nabla H \cdot \alpha) \partial_\theta = i (\nabla H \cdot \alpha) v $ and $idf(X_H) = -(\nabla H \cdot \alpha) v$. From now on, we will denote $\nabla H \cdot \alpha$ by $c$. We have computed that \eqref{delbar2} simplifies to
\begin{equation} \label{delbar3} \bar \partial v = cv (-\beta(\partial_t) + i \beta(\partial_s) ) \end{equation}
which is equivalent to
\begin{equation} \label{CR1} \frac{\partial v_1}{\partial s} - \frac{\partial v_2}{\partial t} = - cv_1 \beta (\partial_t) - cv_2 \beta(\partial_s) \end{equation}
and
\begin{equation} \label{CR2} \frac{\partial v_1}{\partial t} + \frac{\partial v_2}{\partial s} = -cv_2 \beta(\partial_t) + cv_1 \beta(\partial_s). \end{equation}

Using \eqref{delbar3}, we have
\begin{equation} \label{lap} \Delta v = \partial \bar \partial v = c \partial v (-\beta(\partial_t) + i \beta(\partial_s) ) + cv\big( -d\beta(\partial_s, \partial_t) +i (\partial_t \beta(\partial_t) + \partial_s \beta(\partial_s))\big). \end{equation}
Note that by applying \eqref{delbar3} again, we see that
\[c \partial v ( -\beta(\partial_t) + i \beta(\partial_s) ) = c \left(2 \frac{\partial v}{\partial s} - \bar \partial v \right)(-\beta(\partial_t) + i \beta(\partial_s) ) \]
\[ = 2c \frac{\partial v}{ \partial s}( -\beta(\partial_t) + i \beta(\partial_s) ) - c^2 v ( -\beta(\partial_t) + i \beta(\partial_s) )^2 \]
\[ = 2c \frac{\partial v}{ \partial s}(- \beta(\partial_t) + i \beta(\partial_s) ) - c^2 v ( \beta(\partial_t)^2 - \beta(\partial_s)^2 - 2i \beta(\partial_t) \beta(\partial_s) ) \]
and \eqref{lap} expands to
\begin{equation} \begin{split} \label{lap2} \Delta v =  &2c \frac{\partial v}{ \partial s}(- \beta(\partial_t) + i \beta(\partial_s) ) - c^2 v \big( \beta(\partial_t)^2 - \beta(\partial_s)^2 - 2i \beta(\partial_t) \beta(\partial_s) \big) \\
 &+  cv\big( -d\beta(\partial_s, \partial_t) +i (\partial_t \beta(\partial_t) + \partial_s \beta(\partial_s))\big) \end{split} \end{equation}

Also, we observe that
\[ |\bar \partial v|^2 = \left| \frac{\partial v_1}{\partial s} - \frac{\partial v_2}{\partial t} + i \left( \frac{\partial v_1}{\partial t} + \frac{\partial v_2}{\partial s} \right) \right|^2 = | \nabla v_1 |^2 + |\nabla v_2 |^2 + 2 \left( \frac{\partial v_1}{\partial t} \frac{\partial v_2}{\partial s} - \frac{\partial v_1}{\partial s} \frac{\partial v_2}{\partial t} \right)  \]
but also by \eqref{delbar3}
\[ |\bar \partial v|^2 = c^2 |v|^2 | - \beta(\partial_t) + i \beta(\partial s) |^2 = c^2 |v|^2 (\beta(\partial_t)^2 + \beta(\partial_s)^2 ).\]
Therefore, 
\begin{equation} \label{normterm} | \nabla v_1 |^2 + |\nabla v_2 |^2 =  c^2 |v|^2 (\beta(\partial_t)^2 + \beta(\partial_s)^2 ) + 2 \left(  \frac{\partial v_1}{\partial s} \frac{\partial v_2}{\partial t} - \frac{\partial v_1}{\partial t} \frac{\partial v_2}{\partial s} \right) \end{equation}

Since $|v|^2 = v_1^2 + v_2^2$, we get that
\begin{equation} \begin{split} \label{nlap} &\frac{1}{2} \Delta |v|^2 = v \cdot \Delta v + | \nabla v_1 |^2 + |\nabla v_2|^2 \\
&= 2c \left[\frac{\partial v}{\partial s} (- \beta(\partial_t) + i \beta(\partial_s) ) \right] \cdot v - c|v|^2 d\beta(\partial_s, \partial_t)  + 2 c^2 |v|^2 \beta(\partial_s)^2 +  2 \left(  \frac{\partial v_1}{\partial s} \frac{\partial v_2}{\partial t} - \frac{\partial v_1}{\partial t} \frac{\partial v_2}{\partial s} \right) \end{split} \end{equation}
using \eqref{lap2}, \eqref{normterm}, and that $v \cdot iv = 0$. We now note that
\[ \frac{\partial v}{\partial s} (-\beta(\partial_t) + i \beta(\partial_s) ) = \left( - \frac{\partial v_1}{\partial s} \beta(\partial_t) - \frac{\partial v_2}{\partial s} \beta(\partial s) \right) + i \left( -\frac{\partial v_2}{\partial s} \beta(\partial_t) + \frac{\partial v_1}{\partial s} \beta(\partial_s) \right).  \]
Hence, 
\[ c \left[ \frac{\partial v}{\partial s} (-\beta(\partial_t) + i \beta(\partial_s) ) \right]  \cdot v +   \left(  \frac{\partial v_1}{\partial s} \frac{\partial v_2}{\partial t} - \frac{\partial v_1}{\partial t} \frac{\partial v_2}{\partial s} \right) \]
\[ = \frac{\partial v_1}{\partial s} \left( -c v_1 \beta(\partial_t) + cv_2 \beta(\partial_s) + \frac{\partial v_2}{\partial t} \right) + \frac{\partial v_2}{\partial s} \left( - cv_1 \beta(\partial_s) - cv_2 \beta(\partial_t) - \frac{\partial v_1}{\partial t} \right)\]
\[ = \left( \frac{\partial v_1}{\partial s} \right)^2  +2 \frac{\partial v_1}{\partial s} cv_2 \beta(\partial_s) + \left( \frac{\partial v_2}{\partial s} \right)^2 - 2 \frac{\partial v_2}{\partial s} cv_1 \beta(\partial_s)  \]
using \eqref{CR1} and \eqref{CR2}. Combining this with \eqref{nlap}, we finally obtain
\[ \frac{1}{2} \Delta |v|^2 = - c|v|^2 d\beta(\partial_s, \partial_t) + 2 \left( \frac{\partial v_1}{\partial s} + cv_2 \beta(\partial_s) \right)^2 + 2 \left( \frac{\partial v_2}{\partial s} - cv_1 \beta(\partial_s) \right)^2 \]

Thus, we see that $|v|^2$ is subharmonic given that $c d\beta \leq 0$ and hence cannot have any interior local maxima.

Now, suppose that we have that the image under $v$ of the boundary components of $S$ lie on arcs from the origin so that $\text{arg}(v)$ is a constant on the boundary. Near any point on a boundary component of $S$, we can assume that the coordinates $z = s + it$ are such that the boundary component is given by $t = 0$. We then have
\begin{equation} \label{constarg} 0 = \frac{\partial \text{arg}(v)}{\partial s} = \frac{1}{|v|^2} \left( -v_2 \frac{\partial v_1}{\partial s} + v_1 \frac{\partial v_2}{\partial s} \right) \end{equation}
on the boundary. Thus, we have 
\[ \frac{1}{2} \frac{\partial |v|^2}{\partial t} = v_1 \frac{\partial v_1}{\partial t} + v_2 \frac{\partial v_2}{\partial t}  = v_1 \left( - \frac{\partial v_2}{\partial s} - cv_2 \beta(\partial_t) + cv_1 \beta(\partial_s) \right) + v_2 \left( \frac{\partial v_1}{\partial s} + cv_1 \beta (\partial_t) + cv_2 \beta(\partial_s) \right) \]
\[ = c|v|^2 \beta(\partial_s) \]
on the boundary using \eqref{CR1}, \eqref{CR2}, and \eqref{constarg}. We have that $\beta(\partial_s) = 0$ since $\beta|_{\partial S}  = 0$. Thus, we see that there can be no local maxima of $|v|^2$ along the boundary by Hopf's lemma since we have already seen $\Delta |v|^2 \geq 0$.
\end{proof}

Now, suppose that we have strip-like ends near each puncture of $S$. More precisely, we have proper holomorphic embeddings $\epsilon_{j} \colon \bb{R}_{>0} \times [0,1] \to S$ such that $\epsilon_{j}^{-1}(\partial S) = \bb{R}_{>0} \times \{0,1\}$ and $\lim_{s \to \infty} \epsilon_{j} (s,t) = z_j$ for $j = 1, \hdots, k$ and the same with $\bb{R}_{<0}$ in the place of $\bb{R}_{>0}$ for $j = 0$. Fix weights $w_j \in \bb{R}$ for $j = 0,  \hdots, k $ and time-$w_j$ flow lines $p_j(t)$ of $X_H$ from $L_{j-1}$ to $L_j$ for $j = 1, \hdots, k$ and $p_0(t)$ a time $w_{0}$ flow line of $X_H$ from $L_0$ to $L_k$. Using Proposition \ref{monmax} and repeating the proof of Proposition \ref{comp1}, we obtain the following.

\begin{prop} \label{comp2} There exists a compact subset of $(\bb{C}^*)^n$ that contains the image of any solution of \eqref{perturb} such that $\nabla H \cdot \alpha$ is constant and $(\nabla H \cdot \alpha) d\beta \leq 0$ in $C_\alpha$ for all $\alpha \in A$, $\beta|_{\partial S} = 0$, $\epsilon_j^*\beta = w_j dt$, the image of the boundary region between $z_j$ and $z_{j+1}$ lies on $L_j$ for $j = 0, \hdots, k$ with $z_{k+1} = z_0$, and $\lim_{s \to \infty} u \circ \eps_j (s,t) = p_j (t)$ for $i = 0, \hdots, k$. 
\end{prop}

\begin{rem} \label{perturbenergy} It will again be important to have an a priori energy bound for the maps in Proposition \ref{comp2}. The conditions on $\beta$ in the strip-like ends and the limit points are really only needed for this energy bound and not in the proof of Proposition \ref{comp2}. For solutions of \eqref{perturb}, there are two notions of energy, geometric and topological, defined respectively by
\[ E^{\text{geom}}(u) = \int_S \frac{1}{2} \| du - X_H \otimes \beta \|^2 \, d\text{vol}_S = \int_S u^* \omega - u^*dH \wedge \beta \]
and
\[ E^{\text{top}} (u) = \int_S u^* \omega - d( u^*H \, \beta) = E^\text{geom}(u) - \int_S u^*H d\beta. \]
Similar to the unperturbed case, topological energy has an a priori energy bound by Stokes' theorem due to exactness of our Lagrangians. When $H d\beta \leq 0$, we have $E^\text{geom}(u) \leq E^\text{top}(u)$ and a bound of geometric energy. Unfortunately, this is not always the case in our desired applications. However, Proposition \ref{comp2} guarantees that $u$ cannot leave a compact subset of $(\bb{C}^*)^n$ so there exists $0 > K \in \bb{R}$ such that $u^*H \geq K$. With $d\beta \leq 0$, we then have
\[ E^\text{geom}(u) \leq E^\text{top}(u) + K \int_S d\beta  = E^\text{top}(u) + K ( - w_0 + w_1 + \hdots + w_n) \]
giving a bound independent of $u$ and $\beta$. 
\end{rem}

\begin{rem} \label{moving} We will also need to consider Floer solutions with moving boundary conditions. In particular, we will consider solutions of \eqref{perturb} with a boundary condition $L_z = \phi^{\psi(z)} (L)$ for $z \in \partial S$ where $\phi$ is the flow of a Hamiltonian $H$ with $\nabla H \cdot \alpha$ constant for all $\alpha \in A$ and $\psi\colon \partial S \to \bb{R}$ is a smooth  function which is locally constant in the strip-like ends (and monotonic in all applications). Further, we will require that $\beta|_{\partial S} = d \psi$ and $\epsilon_j^*\beta = 0$. 

In this setting, the interior maximum principle is exactly the same as in Proposition \ref{monmax} and requires only that $(\nabla H \cdot \alpha) d\beta \leq 0$. On the boundary $t = 0$ in coordinates $z = s + it$, we have 
\[ \psi'(s)(\nabla H \cdot \alpha) = \frac{ \partial \text{arg}(v)}{\partial s} \]
in place of \eqref{constarg} which leads to $\partial |v|^2/ \partial t = 0$ again giving the desired maximum principle on the boundary. The condition that $\beta|_{\partial S} = d \psi$ implies that $E^{top}(u)$ is indeed a topological quantity, and a similar argument to that of Remark \ref{perturbenergy} will give a uniform bound on geometric energy.
\end{rem}

\subsection{Floer theory of monomially admissible Lagrangians} \label{Floer}

We now define the Floer theory of exact monomially admissible Lagrangians with respect to a monomial division $\Delta$ for a Laurent polynomial $W$ in an essentially standard way which largely follows the setup of \cite{SeBook} where more details are given. 

Suppose that $L_0, L_1$ are transverse exact Lagrangians monomially admissible with respect to $\Delta$ and $W_j = e^{-i\theta_j}W$ for $j = 0,1$, respectively. Further, suppose that $\theta_0 \neq  \theta_1$. Let $\mathcal{J}(\Delta)$ be the space of $\omega$-compatible monomially admissible almost complex structures. Given a $J \in \mathcal{J}(\Delta)$ and $p, q \in L_0 \cap L_1$, define $\widehat{\mathcal{M}}(p, q; J)$ to be the moduli space of $J$-holomorphic strips $u\colon \bb{R} \times [0,1] \to (\bb{C}^*)^n$ satisfying  $u(\bb{R} \times \{ 0 \}) \subset L_0$,  $u(\bb{R} \times \{ 1 \}) \subset L_1$, $\lim_{s \to \infty} u(s,t) = p$, and $\lim_{s \to -\infty} u(s,t) = q$. Let $\mathcal{M}(p, q; J)$ be the quotient of $\widehat{\mathcal{M}}(p, q;  J)$ by the reparameterization action of $\bb{R}$. Let $\overline{\mathcal{M}}(p, q; J)$ be the Gromov compactification of $\mathcal{M}(p,q;J)$ obtained by adding broken holomorphic strips (there is no sphere or disc bubbling due to exactness) at the boundary. 

\begin{prop} \label{comp3} $\overline{\mathcal{M}}(p, q; J)$ is compact. 
\end{prop}
\begin{proof} This is a standard case of Gromov compactness after using Proposition \ref{comp1} to see that all discs in $\overline{\mathcal{M}}(p, q; J)$ lie in a compact subset of $(\bb{C}^*)^n$. 
\end{proof}

A generic $t$-dependent almost-complex structure $J = \{J_t\}_{t\in [0,1]}$ chosen in $\mathcal{J}(\Delta)$ will be regular in the sense that $\overline{\mathcal{M}}(p, q; J)$ is a smooth manifold (see for instance Section 9k in \cite{SeBook} and note that the restrictions on $J$ at infinity pose no problem since the curves are contained in a compact set; the transversality was first established in \cite{FHS}). We choose a regular $J$, which we denote $J_{L_0, L_1}$. With the a priori energy bound noted in Remark \ref{energy1}, Gromov compactness guarantees that the number of points in the zero-dimensional part of $\overline{\mathcal{M}}(p, q; J_{L_0, L_1})$, which consists of Maslov index zero strips and is the same as the zero-dimensional part of $\mathcal{M}(p, q; J_{L_0, L_1}),$ is finite.
 
Then, we define $CF(L_0, L_1)$ to be the vector space freely generated by the intersection points of $L_0$ and $L_1$, and we define a differential on $CF(L_0, L_1)$ by linearly extending
\[ m^1(p) = \sum_{q \in L_0 \cap L_1 } \# \mathcal{M}(p, q; J_{L_0, L_1}) q \]
where $p$ is a basis element and $\#$ is the count of points in the zero-dimensional part. The fact that $m^1$ is indeed a differential follows from analyzing the boundary of the $1$-dimensional part of $\overline{\mathcal{M}}(p, q; J_{L_0, L_1})$. 

The construction above works when over a field of characteristic two and produces ungraded chain complexes. To obtain graded complexes, we choose a holomorphic volume form on $(\bb{C}^*)^n$ which gives a phase map $L \to S^1$ for each Lagrangian. A grading on a Lagrangian is a choice of lift of the phase map to a map $L \to \bb{R}$. The obstruction to the existence of such a lift is the Maslov class in $H^1(L)$. Hence, simply connected Lagrangians always admit a grading. Note that the Lagrangian sections, which we are interested in as the mirrors to line bundles, all carry an essentially canonical grading given by $n$ minus the Morse index of their twisting Hamiltonian as explained in Example 2.10 in \cite{SeGraded}. The issue of defining Floer complexes over a field over characteristic not equal to $2$, such as $\bb{C}$ in our case, amounts to choosing an orientation for the moduli spaces above which can be done via a choice of Pin structure for each Lagrangian as detailed in Section 11j of \cite{SeBook}.

\begin{df} A Lagrangian brane is a Lagrangian submanifold with a choice of grading and Pin structure. Further, any adjective applied to a Lagrangian brane, such as monomially admissible, will mean to apply that adjective to the underlying Lagrangian submanifold.
\end{df}

We will often abuse notation and denote a Lagrangian brane the same as its underlying Lagrangian submanifold. We can also define higher structure maps on the chain complexes of Lagrangian branes in a similar way to how we defined $m^1$. For $d \geq 2$, let $\overline{\mathcal{R}}_{d+1}$ be the Deligne-Mumford-Stasheff compactification of the space of discs with $d +1$ marked points labeled $z_0, \hdots, z_d$ in counterclockwise order. Let $\overline{\mathcal{S}}_{d+1} \to \overline{\mathcal{R}}_{d+1}$ be the universal family over this space. We choose a family of strip-like ends
\[ \epsilon_{j}^{d+1} \colon \bb{R}_{>0} \times [0,1] \times \overline{\mathcal{R}}_{d+1} \to \overline{\mathcal{S}}_{d+1} \]
for $j = 1, \hdots, d$ and 
\[ \epsilon_{0}^{d+1}\colon \bb{R}_{<0} \times [0,1] \times \overline{\mathcal{R}}_{d+1} \to \overline{\mathcal{S}}_{d+1} \]
which are compatible in the sense that they determine boundary collars
\[ \overline{\mathcal{R}}_{d+1} \times \overline{\mathcal{R}}_{\ell+1} \times [R, \infty) \to \overline{\mathcal{R}}_{d+\ell} \]
for sufficiently large $R$ by gluing some $z_j$ with $j \geq 1$ in a disc in $\overline{\mathcal{R}}_{\ell+1}$ to $z_0$ in a disc in  $\overline{\mathcal{R}}_{d+1}$ along the positive strip-like end at $z_j$ and negative strip-like end at $z_0$ and that the induced strip-like ends on $\overline{\mathcal{R}}_{d+\ell}$ match those already chosen for that stratum. 

Now, suppose that $L_0, \hdots, L_d$ are mutually transverse exact monomially admissible Lagrangian branes with respect to $\Delta$ and $W_j = e^{-i \theta_j}W$ for $j = 0, \hdots, d$. Further, suppose that $\theta_j \neq \theta_\ell$ for $j \neq \ell$ and that we have chosen regular almost complex structures $J_{L_j, L_{j+1}}$ for $j = 0, \hdots, d-1$ and for the pair $L_0, L_d$. We can then choose inductively for every sequence $0 \leq j_0 < \hdots < j_\ell \leq d$ with $\ell \geq 2$ a family of almost complex structures
\[ J_{L_{j_0}, \hdots, L_{j_\ell}}\colon \overline{\mathcal{S}}_{\ell + 1} \to \mathcal{J}(\Delta) \]
which are compatible in the sense that they agree with the already chosen complex structures for pairs when labeling the boundary components by Lagrangians counterclockwise from $z_0$ and the restriction of $J_{L_{j_0}, \hdots, L_{j_{\ell + m -1}}}$ to any boundary stratum $\overline{\mathcal{R}}_{\ell+1} \times \overline{\mathcal{R}}_{m +1} \subset \overline{\mathcal{R}}_{\ell +m }$ is the already chosen almost complex structure on $\overline{\mathcal{S}}_{\ell+1}$ and $\overline{\mathcal{S}}_{m +1}$ using the appropriate Lagrangian labels for the broken disc.

With those choices in hand, let $p_j \in L_j \cap L_{j+1}$ for $j = 0, \hdots, d-1$ and $q \in L_0 \cap L_d$. We define $\overline{\mathcal{M}}(p_1, \hdots, p_d, q; J_{L_0, \hdots, L_d} )$ to be the set of pairs of $(r, u)$ where $r \in \overline{\mathcal{R}}_{d +1}$ and $u$ is a $J_{L_0, \hdots, L_d}$-holomorphic map on the fiber of $\overline{\mathcal{S}}_{d+1} \to \overline{\mathcal{R}}_{d+1}$ over $r$ which maps boundary components to $L_0, \hdots, L_d$ counterclockwise from $z_0$ (on the glued domain) and limit to the $p_j$ in the positive strip-like ends and any negative gluing strip-like end and to $q$ in the negative strip-like end at $z_0$. Applying the same reasoning in the proof of Proposition \ref{comp3} to $\overline{\mathcal{M}}(p_1, \hdots, p_d, q; J_{L_0, \hdots, L_d} )$ shows that this is a compact space. Moreover, it will also be a smooth manifold for a generic domain-dependent choice of $J_{L_0, \hdots, L_d}$. Thus, we can define maps
\[ m^d\colon CF^\bullet(L_{d-1}, L_d) \otimes \hdots \otimes CF^\bullet(L_0, L_1) \to CF^\bullet(L_0, L_d) [2-d] \]
by 
\[ m^d(p_d, \hdots, p_1) =  \sum_{q \in \mathcal{H}_{L_0, L_k}} \# \overline{\mathcal{M}}(p_1, \hdots, p_d,  q; J_{L_0, \hdots, L_d}) q. \]
These operations satisfy the $A_\infty$ relations due to the geometry of the boundary of the $1$-dimensional part of $\overline{\mathcal{M}}(p_1, \hdots, p_d, q; J_{L_0, \hdots, L_d} )$.

In this setting, the Floer cohomology groups are invariant in the ways that one might expect. 
\begin{prop} \label{invarianthf} For Lagrangian branes $L_0, L_1$ monomially admissible with respect to $\Delta$ and $W_{j}  = e^{-i\theta_{j}}W$ for $j \in \{0, 1\}$ and $\theta_0 \neq \theta_1$, the Floer cohomology $HF^\bullet(L_0, L_1)$ is invariant of choice of $J \in \mathcal{J}(\Delta)$ and admissible Hamiltonian isotopies of $L_0$ and $L_1$. 
\end{prop} 
\begin{proof} The statement follows from the standard continuation map argument given that the set of monomially admissible Hamiltonians is contractible, the set of monomially admissible almost complex structures adapted to a toric K\"{a}hler form on $(\bb{C}^*)^n$ is contractible, the continuation maps are well-defined for monomially admissible Hamiltonians, and the pushforward of $J \in \mathcal{J}(\Delta)$ by the flow monomially admissible Hamiltonian remains in $\mathcal{J}(\Delta)$ . The first statement follows from the observation that the set of monomially admissible Hamiltonians is convex, and the latter three follow from Lemma \ref{contractible}, Proposition \ref{comp2}, and Lemma \ref{preservesJ}, respectively. 
\end{proof}

\begin{rem} \label{continuedef} In the proof above and in what follows, a continuation map is a map from $CF^\bullet(L_0, L_1)$ to the perturbed complex $CF^\bullet(L_0, L_1; H)$ generated by flow lines of the Hamiltonian $H$ from $L_0$ to $L_1$. A continuation map is defined by counting solutions to \eqref{perturb} satisfying the conditions of Proposition \ref{comp2} on the strip $S = \bb{R}_s \times [0,1]_t$ with the additional requirements that $\beta = 0$  for $s \gg 0$ (at the input) and $\beta = dt$ for $s \ll 0$. The perturbed complex $CF^\bullet(L_0, L_1; H)$ can be identified with $CF^\bullet(\phi_H(L_0), L_1)$ after a change of almost complex structure. 
\end{rem}

We end this subsection by showing the set of monomially admissible almost complex structures adapted to a K\"{a}hler form on $(\bb{C}^*)^n$ is contractible. This is an important feature for the almost complex structures used in any Floer theory and was needed in the preceding proposition. 

\begin{lem} \label{contractible} Suppose $\omega$ is a K\"{a}hler form on $(\bb{C}^*)^n$. Then, the set 
\[ \mathcal{J}_\omega(\Delta) = \{ J \in \mathcal{J}(\Delta) | J \text{ is adapted to } \omega \} \]
is contractible. 
\end{lem}
\begin{proof}[Sketch of proof] If $z \in C_\alpha$ and $f(z) = c_\alpha z^\alpha$, then one can show that $f$ is $J$-holomorphic at $z$ if and only if $J_z = (J_0)_z$ on $(\ker df)^\omega$ where $J_0$ is the standard almost complex structure on $(\bb{C}^*)^n$.  

Thus, we can consider the set of metrics 
\[ \mathcal{G}(\Delta) = \{ g \in \{ \text{Riemannian metrics on } (\bb{C}^*)^n \} \, | \, g_z = (g_{J_0})_z \text{ on } (\ker d(c_\alpha z^\alpha) )_z^\omega \text{ if } \mu(z) \in C_\alpha \} \]
where $g_{J_0} = \omega( \cdot, J_0 \cdot)$. The set $\mathcal{G}(\Delta)$ is convex and hence contractible. In addition, there are maps $ \mathcal{J}_\omega(\Delta) \to \mathcal{G}(\Delta)$ and $\mathcal{G}(\Delta) \to \mathcal{J}_\omega(\Delta)$ such that the composition $\mathcal{J}_\omega(\Delta) \to \mathcal{G}(\Delta) \to \mathcal{J}_\omega(\Delta)$ is the identity. It follows that $\mathcal{J}_\omega(\Delta)$ is contractible.
\end{proof}

\subsection{The Fukaya-Seidel category via localization} \label{flocalize}

We will implement the localization procedure of \cite{AbSloc} in order to define $\mathcal{F}_\Delta(W)$ as a partially wrapped Fukaya category. The reader may also find the exposition in Chapter 10 of \cite{SeNotes} useful in understanding how to obtain Fukaya categories via localization. In order to proceed, we need to be able to flow Lagrangians that are monomially admissible with respect to $W$ to those monomially admissible with respect to $e^{-i\theta}W$. 

\begin{assump} \label{canassump} Assume that there exists a Hamiltonian $K_\Delta = K_\Delta (\mu_1, \hdots, \mu_n)$ on $(\bb{C}^*)^n$ such that $\nabla K_\Delta \cdot \alpha = 1$ outside of a compact set in each $C_\alpha$.
\end{assump}

Assumption \ref{canassump} always holds when $W = W_\Sigma$ and $\Delta$ is adapted to $\Sigma$ by Corollary \ref{twistexist} as in that case $K_\Delta$ is a twisting Hamiltonian for the canonical bundle written as $\sum -D_\alpha$ up to a factor of $2\pi$. In general, the time-$\theta$ flow of the Hamiltonian $K_\Delta$ will take Lagrangians monomially admissible with respect to $W$ to Lagrangians monomially admissible with respect to $e^{-i \theta}W$.

Let $O_\Delta$ be a countable set of monomially admissible Lagrangian branes with respect to $W$ and $\Delta$. Further, choose a Hamiltonian $K_\Delta$ as in Assumption \ref{canassump} and inductively choose a sequence $0 = \theta^0 < \theta^1 < \hdots < \theta^j < \hdots < \pi$ such that $\lim_{j \to \infty} \theta^j = \pi$ in a sufficiently generic way such that for all $L \in O_\Delta$, $\psi^{\theta^j}(L)$ intersects $\psi^{\theta^k}(L')$ transversely for all $\theta^k < \theta^j$ and all $L' \in O_\Delta$ where $\psi$ is the flow of $K_\Delta$. 

With that data, we define an $A_\infty$-category $\mathcal{F}_\Delta^\circ$ whose objects are pairs $(L, \theta)$ where $L$ is a Lagrangian brane in $O_\Delta$ and $\theta$ belongs to our chosen sequence of angles. Morphisms are defined by
\[ \Hom_{\mathcal{F}_\Delta^\circ} \big((L_0, \theta_0), (L_1, \theta_1) \big) = \begin{cases} CF^\bullet(\psi^{\theta_0}(L_0) , \psi^{\theta_1}(L_1)) & \theta_0 > \theta_1 \\ \mathbb{K} \cdot e_{(L_0, \theta_0)} & (L_0, \theta_0) = ( L_1, \theta_1) \\ 0 & \text{otherwise} \end{cases} \] 
where $e_{(L_0, \theta_0)}$ is a formal element of degree zero and $\mathbb{K}$ is a field ($\mathbb{K} = \mathbb{C}$ in our intended applications). Note that $(L_0, \theta_0) = (L_1, \theta_1)$ means not only that the angles and underlying Lagrangians agree but also the brane structures. The brane structures used to compute Floer cochain complexes are obtained via pushforward by the $\psi^{\theta_j}$ from those on the $L_j$. The structure of a strictly unital $A_\infty$-category on $\mathcal{F}_\Delta^\circ$ is given by defining the $e_{(L, \theta)}$ to be strict units and defining the only nontrivial structure maps 
\[ m^d\colon CF^\bullet( \psi^{\theta_{d-1}} (L_{d-1}), \psi^{\theta_d}(L_d) ) \otimes \hdots \otimes CF^\bullet( \psi^{\theta_0}(L_0), \psi^{\theta_1}(L_1) ) \to CF^\bullet( \psi^{\theta_{0}} (L_{0}), \psi^{\theta_d}(L_d) ) [2-d]\]
as in Section \ref{Floer} when $\theta_0 > \theta_1 > \hdots > \theta_k$ after making a consistent choice of strip-like ends and regular monomially admissible almost complex structures. 

Now, we observe that $\mathcal{F}_\Delta^\circ$ comes with a class of quasi-units $c_{(L, \theta \to \theta')} \in \Hom( (L, \theta'), (L, \theta))$ when $\theta' > \theta$ defined by counting solutions to \eqref{perturb} with $S$ the upper half plane equipped with a strip-like end, $H = K_\Delta$, moving boundary condition given by $\psi^{f(s)}(L)$ where $f\colon \partial S \to \bb{R}$ is $\theta'$ for $s \ll 0$ and $\theta$ for $s \gg 0$, and $\beta$ is a one-form on $S$ that vanishes in the strip-like end and satisfies  $d\beta \leq 0$ and $\beta|_{\partial S} = df$ (such elements appear for instance in Section 8k of \cite{SeBook} in a more standard setting). As usual, when we say we count such solutions, we mean that we count the number of elements in the zero-dimensional part of the moduli space of such maps for a generic domain-dependent almost complex structure which agrees with the one chosen for $\psi^{\theta}(L)$ and $\psi^{\theta'}(L)$ in the strip-like end. Compactness holds by Remark \ref{moving}. Compactifying the $1$-dimensional part of this moduli space and counting boundary points shows that the quasi-units are closed.  The multiplication map
\begin{equation} \label{multmap} m^2(\cdot, c_{(L, \theta \to \theta')})\colon CF^\bullet( \psi^{\theta}(L) , L') \to CF^\bullet( \psi^{\theta'}(L), L') \end{equation}
can be seen as a continuation map (well-defined by Proposition \ref{comp2} when the monomially admissible Lagrangians are transverse and disjoint at infinity)  by gluing and applying a change of coordinates to produce a constant boundary condition as in Section 8k of \cite{SeBook} for any $L'$ where both Floer cochain complexes are well-defined. In fact, with any sensible definition of $HF(L, L)$, we should have that the quasi-units are images of the (cohomological) identity morphism under continuation. Moreover, the quasi-units can be seen as the first order term of a natural transformation as in Section \ref{defsect}. The quasi-units have several additional properties. 

\begin{prop} \label{invariantqu} The following hold for the quasi-units defined in the paragraph above.
\begin{enumerate}[label=(\alph*)]
\item If $\theta'' > \theta' > \theta$, then $m^2( c_{(L, \theta \to \theta')}, c_{(L, \theta' \to \theta'')}) = c_{(L, \theta \to \theta'')}$ in $HF^0( \psi^{\theta''}(L), \psi^{\theta}(L) )$. 
\item The quasi-units are sent to quasi-units on the level of cohomology under the continuation map isomorphism in the proof of Proposition \ref{invarianthf}.
\item The multiplication map \eqref{multmap} is a quasi-isomorphism when $\theta$ and $\theta'$ are both larger than the admissibility angle of $L'$, i.e., when $\theta, \theta' > \tau$ where $L'$ is monomially admissible for $e^{-i\tau}W$. 
\end{enumerate}
\end{prop}
\begin{proof} Both (a) and (b) are gluing arguments. For (a), one can see that gluing the domain of the multiplication map with that of $c_{(L, \theta \to \theta')}$ and  $c_{(L, \theta' \to \theta'')}$ at the two inputs gives the upper half-plane with a one-form and boundary condition that can be isotoped to those of $c_{(L, \theta \to \theta'')}$. See Figure \ref{gluequ}. Similarly, (b) involves gluing the domain of a quasi-unit to the input of the domain of a continuation map. Again, this will give a domain that can be isotoped to that of a quasi-unit. 

To see (c), first recall that the multiplication map \eqref{multmap} can be identified with a continuation map. The fact that a continuation map is an isomorphism in this case has analogues in more standard setups (see for instance Section 6.1 of \cite{Lef1}). Indeed, the same proof applies. The Hamiltonian isotopy from $\psi^\theta(L)$ to $\psi^{\theta'} (L)$ induced by $K_\Delta$ can be factored as a compactly supported Hamiltonian isotopy and a Hamiltonian isotopy which is the identity in a compact set containing all intersection points with $L'$. For the former, we can construct an inverse continuation map. For the latter, we can assume that the Hamiltonian $H = K_\Delta$ is nonnegative so that $H d \beta \leq 0$ so that $0 \leq E^{\text{geom}}(u) \leq E^{\text{top}}(u)$ (see Remark \ref{perturbenergy}) for any perturbed holomorphic map $u$ counted as part of the continuation map. From $E^{\text{top}}(u) \geq 0$, we obtain that the continuation map increases action, which is in this case simply the difference in values for primitives of the one-form $\lambda$ on the boundary Lagrangians at an intersection point. Moreover, a  continuation map solution $u$ can only preserve the action if it has topological and hence geometric energy zero, i.e., it is constant. It follows that the continuation map is an isomorphism on cohomology since it can be represented as an upper-triangular matrix with ones along the diagonal. 
\end{proof}

\begin{figure}
\centering 
\begin{tikzpicture}
	\draw (0,0.5) -- (3,0.5);
	\draw (0,2.5) -- (3,2.5);
	\draw (1.5, 3) node {$\psi^{\theta}(L) $};
	\draw (1.5, 0) node {$\psi^{\theta''}(L)$};
	\draw[dashed] (3, 0.5) arc (-90:90: 1);
	
	\draw (11,0) -- (13,0);
	\draw (11,1) -- (13,1);
	\draw (12, 1.5) node {$\psi^{\theta'}(L) $};
	\draw (12, -0.5) node {$\psi^{\theta''}(L)$};
	\draw[dashed] (13, 0) arc (-90:90: 0.5);
	
	\draw (11,2) -- (13,2);
	\draw (11,3) -- (13,3);
	\draw (12, 3.5) node {$\psi^{\theta}(L) $};
	\draw[dashed] (13, 2) arc (-90:90: 0.5);
	
	\draw plot [smooth, tension = 0.5] coordinates {(8, 2) (8.5 , 2.1) (10, 2.9) (10.5 ,3)};
	\draw plot [smooth, tension = 0.5] coordinates {(8, 1) (8.5 , 0.9) (10, 0.1) (10.5 ,0)};
	\draw (10.5 ,2) arc (90: 270 : 0.5);
	
	\draw[<->, thick] (5 , 1.5) -- (7, 1.5);
\end{tikzpicture}
\caption{Gluing construction in the proof of Proposition \ref{invariantqu}(a). The dashed lines denote the moving boundary, and the support of $\beta$ is in a small neighborhood of those lines.} 
\label{gluequ}
\end{figure}
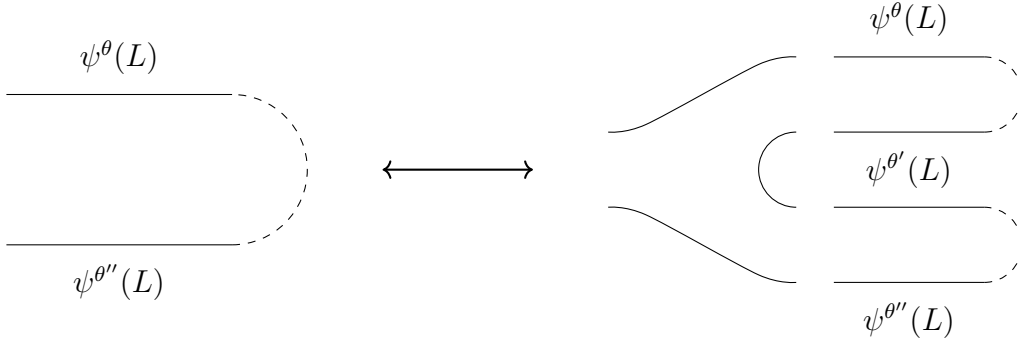 

Now, we define the monomially admissible Fukaya-Seidel category $\mathcal{F}_\Delta (W)$ on the objects of $O_\Delta$ to be the localization of $\mathcal{F}_\Delta^\circ$ at all the quasi-units. More precisely, if $Q$ is the set of quasi-units and $\text{Tw} \, \mathcal{F}_\Delta^\circ$ is the $A_\infty$-category of twisted complexes of objects of $\mathcal{F}_\Delta^\circ$, $\mathcal{F}_\Delta (W)$ has the same objects as $\mathcal{F}_\Delta^\circ$ and has morphisms as in the image of
\[ \mathcal{F}_\Delta^\circ \to \text{Tw} \, \mathcal{F}_\Delta^\circ \to \text{Tw} \, \mathcal{F}_\Delta^\circ / \text{Cones}(Q) \]
where $\text{Cones}(Q)$ is the set of all cones of the morphisms in $Q$ and the quotient is taken in the sense of Lyubashenko-Ovsienko \cite{LO}. In particular, the quotient morphism spaces for $A_\infty$-categories $\mathcal{B} \subset \mathcal{A}$ are defined by
\[ \Hom_{\mathcal{A}/\mathcal{B}} (Z_0,Z_1)  = \bigoplus_{\substack{ p \geq 0 \\ B_1, \hdots, B_p \in \mathcal{B}}} \Hom_\mathcal{A} (B_p, Z_1)[1] \otimes \hdots \otimes \Hom_\mathcal{A} (B_1, B_2)[1] \otimes \Hom_\mathcal{A} (Z_0, B_1) \]
with the bar differential and higher products. The construction takes a strict unit to a strict unit. Therefore, $\mathcal{F}_\Delta (W)$ is strictly unital. The quotient functor $Q\colon \mathcal{A} \to \mathcal{A}/\mathcal{B}$ is the inclusion of $\Hom_\mathcal{A} (Z_0,Z_1)$ as the $p = 0$ term. The quotient enjoys the universal property that for any other $A_\infty$-category $\mathcal{C}$, the $A_\infty$ functor
\[ \circ Q\colon \text{Fun}(\mathcal{A}/\mathcal{B}, \mathcal{C}) \to \text{Fun}(\mathcal{A}, \mathcal{C}) \]
is full and faithful on homology and with image equal up to quasi-isomorphism to the subcategory of functors which are zero in $\text{Fun}(\mathcal{B}, \mathcal{C})$. 

It is useful to understand the morphisms in $\mathcal{F}_\Delta (W)$. First, we note that $(L, \theta)$ and $(L, \theta')$ are quasi-isomorphic in $\mathcal{F}_\Delta(W)$ as we have inverted the quasi-units. Thus, there is map from the elements of $O_\Delta$ onto the quasi-isomorphism classes of objects of $\mathcal{F}_\Delta(W)$. 
\begin{prop} \label{localhoms} The two maps in
\[ H^\bullet \Hom_{\mathcal{F}_\Delta(W)} (L_0, L_1) \to \lim_{\theta \to \pi} HF^\bullet ( \psi^\theta (L_0), L_1)  \leftarrow HF^\bullet ( \psi^\eps (L_0), L_1) \]
are isomorphisms for all $L_0, L_1 \in O_\Delta$ and any $\eps > 0$ for which $\psi^\eps(L_0)$ and $L_1$ intersect transversely.
\end{prop}
\begin{proof} For the map on the left, see Lemma 7.18 in \cite{SeNotes}. In fact, a stronger chain level statement that the morphism chain complex itself is a homotopy colimit of the Floer cochain complexes is proved there. The condition that 
\[ \lim_{\theta \to \pi} HF^\bullet((L, \theta), C) = 0 \]
whenever $C$ is the cone of a quasi-unit needed in the cited proof follows from Proposition \ref{invariantqu}(c). 

The fact that the map on the right is an isomorphism also follows from Proposition \ref{invariantqu}(c).
\end{proof}

Although many choices were made in constructing $\mathcal{F}_\Delta(W)$, the following proposition along with standard arguments (such as in Section 10a of \cite{SeBook}) on the independence of $\mathcal{F}_\Delta^\circ$ from choices made in its definition show that the only choices of consequence for the quasi-equivalence type of $\mathcal{F}_\Delta(W)$ are the elements of $O_\Delta$ and the monomial division $\Delta$.

\begin{prop} \label{choices} The quasi-equivalence type of $\mathcal{F}_\Delta (W)$ is independent of the following choices made in defining the quasi-units: strip-like ends, monomially admissible almost complex structures, $K_\Delta$, generic sequence of angles, and one-form $\beta$. 
\end{prop}
\begin{proof} The quasi-equivalence type of the quotient depends only on the cohomology classes of the quasi-units (a proof can be found in Section 3.1.3 of \cite{GPS1}). The cohomology classes of quasi-units are independent of the choice of strip-like ends, monomially admissible almost complex structure, and one-form $\beta$ because the space of choices is contractible. 

Assume that $K^1_\Delta$ and $K^2_\Delta$ are two choices of global twisting functor. Then, we can take a generic sequence of angles such that for all $L \in O_\Delta$, $\psi_{1}^{\theta^j}(L)$ and $\psi_{2}^{\theta^j}(L)$ intersect $\psi_{1}^{\theta^k}(L')$ and $\psi_{2}^{\theta^k}(L')$ transversely for all $\theta^k < \theta^j$ and all $L' \in O_\Delta$. We can then define an  $A_\infty$-category $\mathcal{F}_\Delta^\circ (K^1_\Delta \sqcup K^2_\Delta)$ as above except equipping objects with the additional data of whether to flow by $\psi_1$ or $\psi_2$. Since $K^1_\Delta$ and $K^2_\Delta$ differ by an admissible Hamiltonian, we can define quasi-units increasing angles as above that also possibly the twisting flow between $\psi_1$ and $\psi_2$. Then, $\mathcal{F}_\Delta (W; K_\Delta^1)$ and $\mathcal{F}_\Delta (W; K_\Delta^2)$ fully and faithfully map to the localization of $\mathcal{F}_\Delta^\circ (K^1_\Delta \sqcup K^2_\Delta)$ at such quasi-units. Moreover, these two inclusions are quasi-equivalences by combining Proposition \ref{invarianthf} and Proposition \ref{localhoms}. Thus, it remains only to deal with the choice of generic sequence of angles.

Suppose that are $\{ \theta^j \}$ and $\{ \wt\theta^j \}$ are two choices of sequences of angles satisfying the conditions outlined above needed to define monomially admissible Fukaya-Seidel categories $\mathcal{F}_\Delta (W; \{ \theta^j \})$ and $\mathcal{F}_\Delta (W; \{ \wt\theta^j \})$. Then, one can find another sequence $\{ \widehat\theta^j \}$ satisfying the conditions to define another such category such that a sequence consisting of the elements of $\{ \theta^j \} \cup \{ \widehat\theta^j \}$ and a sequence consisting of the elements of $\{ \wt\theta^j \} \cup \{ \widehat\theta^j \}$ both also satisfy the conditions for defining monomially admissible Fukaya-Seidel categories. Then, we have inclusions
\[ \mathcal{F}_\Delta (W; \{ \theta^j \}) \to \mathcal{F}_\Delta (W; \{ \theta^j \} \cup \{ \widehat\theta^j \}) \leftarrow  \mathcal{F}_\Delta (W; \{ \widehat\theta^j \} ) \to  \mathcal{F}_\Delta (W; \{ \wt \theta^j \} \cup \{ \widehat\theta^j \}) \leftarrow \mathcal{F}_\Delta (W; \{ \wt\theta^j \} )\]
that are all quasi-equivalences by Proposition \ref{localhoms}. 
\end{proof}

In addition, we could upgrade Proposition \ref{invarianthf} to show that the quasi-isomorphism type of $\mathcal{F}_\Delta(W)$ depends only on the admissible Hamiltonian isotopy classes of the Lagrangian branes in $O_\Delta$. With that in mind, there is a particular choice of objects of interest to us. 

\begin{df} The category of monomially admissible Lagrangian sections $\mathcal{F}_\Delta^s (W_\Sigma)$ is the monomially admissible Fukaya-Seidel category $\mathcal{F}_\Delta (W_\Sigma)$ with $O_\Delta$ consisting only of Lagrangian sections equipped with the canonical grading of a section and containing a Lagrangian section in each admissible Hamiltonian isotopy class. 
\end{df}

In the case that $W = W_\Sigma$ and the division is adapted to the fan, we can construct $\mathcal{F}^s_\Delta(W_\Sigma)$ by taking $O_\Delta$ to be the set of $\phi_D(\mathcal{L})$ for every toric divisor $D$ on  $X$ where $\phi_D$ is the time-$1$ flow of a twisting Hamiltonian for $D$ and $\mathcal{L} = (\bb{R}_{>0})^n$ is the zero section.  In that case,  Proposition \ref{bij} shows that $\mathcal{F}_\Delta^s (W_\Sigma)$ indeed contains all admissible Hamiltonian isotopy classes of sections with some repeated when several divisors correspond to the same line bundle on $X$. In Section \ref{catmono}, we will make restrictions on the $\phi_D$, but the choice of $\phi_D$ is up to admissible Hamiltonian so does not affect the quasi-isomorphism type of $\mathcal{F}_\Delta^s (W_\Sigma)$. Even more is true.

\begin{prop} \label{adaptedinvariant} For any $\Delta$ and $\Delta'$ adapted to $\Sigma$, $\mathcal{F}^s_\Delta(W_\Sigma)$ and $\mathcal{F}_{\Delta'}^s(W_\Sigma)$ are quasi-equivalent.
\end{prop}
\begin{proof} Both $\Delta$ and $\Delta'$ are contained in a combinatorial division $\Delta^{c}$ (by taking the cones in $\Delta^{c}$ as close to filling the whole stars as necessary). Then, the objects of $\mathcal{F}^s_{\Delta^c}(W_\Sigma)$ represent every admissible Hamiltonian isotopy class of Lagrangian sections with respect to $\Delta$ and $\Delta'$. Hence, the categories are quasi-equivalent.
\end{proof}

We expect that $\mathcal{F}^s_\Delta(W_\Sigma)$ generates $\mathcal{F}_\Delta(W_\Sigma)$ whenever the former is quasi-equivalent to the subcategory of Lagrangian sections of the latter, i.e., whenever the set $O_\Delta$ used to define $\mathcal{F}_\Delta(W_\Sigma)$ contains an object in each admissible Hamiltonian isotopy class of Lagrangian sections. Thus, we also expect $\mathcal{F}_\Delta(W_\Sigma)$ to be independent of  $\Delta$ among monomial divisions adapted to $\Sigma$. 

\subsection{A preliminary Floer cohomology computation} \label{floercomp}

In this subsection, we perform a computation to illustrate how some morphism spaces in $\mathcal{F}^s_\Delta (W_\Sigma)$ can be explicitly matched with morphisms of line bundles in $D^b\Coh(X)$ when $\Delta$ is adapted to $\Sigma$. We will see that this matching extends to an equivalence of categories using a less direct approach in Section \ref{hms}.

Suppose that $D$ is an ample divisor on $X$. Since $\Delta$ is adapted to $\Sigma$, we can find a twisting Hamiltonian $H_D$ for $D$ with flow $\phi_D^t$ by Corollary \ref{twistexist}. Our goal is to compute
\[ H^\bullet \Hom_{\mathcal{F}^s_\Delta (W_\Sigma)} (\mathcal{L} , \mathcal{L}(k D)) \]
for $k \in \bb{N}$ where $\mathcal{L} = (\bb{R}_+)^n$ and $\mathcal{L}(kD) = \phi_D^k(\mathcal{L})$ as before. By Proposition \ref{localhoms}, we are attempting to compute $HF^\bullet (\psi^{\eps}( \mathcal{L}), \mathcal{L}(kD))$ for a small $\eps > 0$. The generators of the underlying chain complex correspond to the intersection points of any lift of $\psi^{-\eps}(\mathcal{L}(kD))$ with $\bb{R}^n \times 2 \pi \bb{Z}^n \subset T^*\bb{R}^n$. In particular, the generators correspond to $\mu \in \bb{R}^n$ such that $\nabla (kH_D - \eps K_\Delta) (\mu) \in 2\pi\bb{Z}^n$. We first analyze the behavior of $\nabla H_D$ itself. 

\begin{lem} \label{polytope} Suppose that $\Delta$ is adapted to $\Sigma$ and $D = \sum_{\alpha} n_\alpha D_\alpha$ is an ample divisor. There is a twisting Hamiltonian $H_D$ such that the image of $\nabla H_D$ is  $2\pi P$ where $P$ is the polytope associated to the divisor $D$ given by $u \cdot \alpha \geq -n_\alpha$ for all $\alpha \in A$. Moreover, we can ensure that $\nabla H_D$ maps injectively to the interior of $2\pi P$ and has nondegenerate Hessian on the preimage of the interior. 
\end{lem} 
\begin{proof} First, recall that $H_D$ can be constructed as $F_D * \eta_\eps$ where $F_D$ is the $2\pi$ times the piecewise-linear support function of $D$ and $\eta_\eps$ is a smooth, nonnegative, and symmetric mollifier supported on $B_\eps(0)$ for a small $\eps > 0$ as in the proof of Proposition \ref{smooth}. On each maximal cone $\sigma$ in $\Sigma$, we have that $F_D(\mu) = 2 \pi m_\sigma \cdot \mu$ where $m_\sigma \cdot \alpha = - n_\alpha$ for all $\alpha \in \sigma$. The function $H_D$ is concave as noted in Remark \ref{concave}.

Since $\nabla H_D = (\nabla F_D) * \eta_\eps$ and $\eta_\eps$ is nonnegative, the image of $\nabla H_D$ lies in the convex hull of the vectors $2 \pi m_\sigma$. The convex hull of the $m_\sigma$ is in fact $P$ because $D$ is ample (see for instance Theorem 6.1.7 of \cite{CLS}). The vertices $2 \pi m_\sigma$ of $2 \pi P$ lie in the image of $\nabla H_D$ as $H  _D = F_D$ away from the $\eps$-neighborhood of the union of cones of $\Sigma$ with positive codimension. Further, Corollary 2 of \cite{Minty} implies that the image of $\nabla H_D$ is nearly convex (contains the interior of its convex hull). Thus, the interior of $2\pi P$ is in the image of $\nabla H_D$.

If $\nabla H_D (\mu)$ lies in the interior of $2 \pi P$ then the ball $B_\eps(\mu)$ is not contained in the star of $\alpha$ for any $\alpha \in A$. It follows that the set of $\mu \in \bb{R}^n$ mapping to the interior of $2 \pi P$ is bounded. Combining that with the fact that $\nabla H_D$ maps onto the interior of $2 \pi P$, we have that $2 \pi P$ is contained in and hence equal to the image of $\nabla H_D$. 

Now, we will show that $\nabla H_D$ is injective and has nondegenerate Hessian on $U = \nabla H_D^{-1}(\text{int}(2 \pi P))$. For this, we will require that $\eta_\eps$ is strictly positive on the open ball of radius $\eps$ around the origin, i.e., the interior of its support. Under this condition, $\mu \in U$ if and only if the ball $B_\eps(\mu)$ is not contained in $\text{star}(\alpha)$ for any $\alpha \in A$. 

Since $H_D$ is concave, injectivity on $U$ follows from nondegeneracy of the Hessian. Suppose that $B_\eps(\mu)$ intersects $k$ walls $\sigma_1^1 \cap \sigma_1^2, \hdots, \sigma_k^1 \cap \sigma_k^2$ where $\sigma_{j}^1, \sigma_{j}^2$ are maximal cones of $\Sigma$. Let $\nu_j = m_{\sigma_j^1} - m_{\sigma_j^2}$ be a normal vector to $\sigma_j^1 \cap \sigma_j^2$ for $j =1, \hdots, k$. Note that $\nu_j \neq 0$ for all $j$ since $D$ is ample. Using integration by parts and that $\nabla H_D = (\nabla F_D) * \eta_\eps$, we have that the Hessian is given by
\[ \text{Hess}(H_D) (\mu) = \sum_{j=1}^k \frac{1}{\| \nu_j \|} \left( \int_{\sigma_j^1 \cap \sigma_j^2 \cap B_\eps(\mu)} \eta_\eps \, d\text{vol} \right) A_j \]
where $A_j$ is the matrix with entries $(A_j)_{\ell s} = - (\nu_j)_\ell (\nu_j)_s$ for $j, s = 1, \hdots, n$. That is, $\text{Hess}(H_D) (\mu) = \sum_j x_j A_j$ where the $x_j$ are positive constants and the $A_j$ are negative semidefinite matrices with $\ker(A_j) = \text{span} (\nu_j )^\perp = \text{span}( \sigma^1_j \cap \sigma^2_j)$. 

To show that $\text{Hess}(H_D)$ is nondegenerate in $U$, it is enough to show that some subset of the $\nu_j$ span $\bb{R}^n$ when $\mu \in U$ by our computation of the Hessian. Suppose that $\mu \in U$ and $\mu \in \sigma$ where $\sigma = \langle \alpha_1, \hdots, \alpha_n \rangle$ is a maximal cone of $\Sigma$ (we have implicitly used that $X$ is smooth). In order for $B_\eps(\mu)$ not to be contained in $\text{star}(\alpha)$ for any $\alpha \in A$, we must have in particular that $B_\eps(\mu)$ is not contained in $\text{star}(\alpha_j)$ for $j = 1, \hdots, n$. For $j = 1, \hdots, n$, let $d_j = \inf \{ r | B_r(\mu) \cap \partial \text{star}(\alpha_j) \neq \emptyset \}$ and let $\sigma^1_j \cap \sigma^2_j$ be a wall in $\partial \text{star}(\alpha_j)$ that is tangent to the sphere of radius $d_j$  around $\mu$ and has normal $\nu_j$ as above. In our notation, $\sigma^1_j$ is a maximal cone of $\Sigma$ with $\alpha_j$ as one of its generators, and $\sigma^2_j$ is a maximal cone of $\Sigma$ that does not have $\alpha_j$ as a generator but contains all other generators of $\sigma^1_j$. By reindexing, we can assume that $d_\ell \geq d_j$ for $\ell \geq j$. As a result, the open ball $B_{d_j}(\mu)$ is contained in $\text{star}(\alpha_\ell)$ for $\ell>j$. In particular, $\alpha_\ell$ is a generator of $\sigma^1_j$ implying that $\alpha_\ell \in \sigma^2_j$ as well. Therefore, $\alpha_\ell \in \sigma^1_j \cap \sigma^2_j$ for $\ell > j$. Combining this with the fact that $\nu_j \cdot \alpha_j \neq 0$, we see that the vectors $\nu_j$ for $j =1, \hdots, n$ are linearly independent. 
\end{proof}

Note that the function $k H_D - \eps K_\Delta$ with $k \neq 0$ is still a smoothing of a stricly concave ($k>0$) or strictly convex ($k < 0$) piecewise linear function for small enough $\eps$, where strict convexity/concavity are defined as in Remark \ref{concave}. Thus, the arguments of Lemma \ref{polytope} apply to show that we can take the image of $k \nabla H_D - \eps \nabla K_\Delta$ to be the polytope given by $u \cdot \alpha \geq 2\pi(-kn_\alpha - \eps)$ for all $\alpha \in A$ if $k > 0$ and by $u \cdot \alpha \leq 2\pi( -kn_\alpha - \eps)$ for all $\alpha \in A$ if $k < 0$ where $D = \sum n_\alpha D_\alpha$. For $k > 0$, this amounts to moving each facet of $2\pi kP$ by affine length $2\pi \eps$ outward. For $k < 0$, this amounts to moving each facet of $2\pi k P$ by affine length $2\pi \eps$ inward. See Figure \ref{floercomps}. Since we can guarantee that the Hessian is nondegenerate in the interior of these polytopes, all the intersection points corresponding to interior points in $2 \pi \bb{Z}^n$ are transverse. Moreover, they all lie in degree $0$ if $k > 0$ or in degree $n$ if $k < 0$. As a result, the differential must vanish in either case.

\begin{figure}
\centering
\begin{tikzpicture}
\draw[->, thick] (0,0) -- (1, 0);
\draw[->, thick] (0,0) -- (0,3);
\draw[->, thick] (0,0) -- (0,-1);
\draw[->, thick] (0,0) -- (-3, 0);
\draw[ultra thick, blue] (-2.2, 2.4) -- ( 0.4, -0.2) -- (-2.2, -0.2) -- (-2.2, 2.4) ;
\draw[blue, fill = blue, opacity = 0.3] (-2.2, 2.4) -- ( 0.4, -0.2) -- (-2.2, -0.2) -- (-2.2, 2.4) ;
\draw[red] (0,0) node {$\times$};
\draw[red] (-2,0) node {$\times$};
\draw[red] (-2,2) node {$\times$};
\draw[red] (0,2) node {$\times$};

\draw[->, thick] (6,2) -- (9, 2);
\draw[->, thick] (6,2) -- (6,3);
\draw[->, thick] (6,2) -- (6,-1);
\draw[->, thick] (6,2) -- (5, 2);
\draw[ultra thick, blue] (6.4 ,1.8 ) -- ( 7.8 ,0.4 ) -- (7.8, 1.8) -- (6.4 , 1.8) ;
\draw[blue, fill = blue, opacity = 0.3] (6.4 ,1.8 ) -- ( 7.8 ,0.4 ) -- (7.8, 1.8) -- (6.4 , 1.8);
\draw[red] (6,2) node {$\times$};
\draw[red] (6,0) node {$\times$};
\draw[red] (8, 2) node {$\times$};
\draw[red] (8,0) node {$\times$};
\end{tikzpicture}
\caption{Argument/fiber projections (images of the gradients) of the lifts of the form $d(H_D - \eps K_\Delta)$ of $\psi^{-\eps}(\mathcal{L}(D_{(1,0)}))$ and $\psi^{-\eps}(\mathcal{L}(-D_{(1,0)}))$ with $X = \mathbb{P}^2$. Elements of $2 \pi \bb{Z}^2$ are shown in red. } \label{floercomps}
\end{figure}
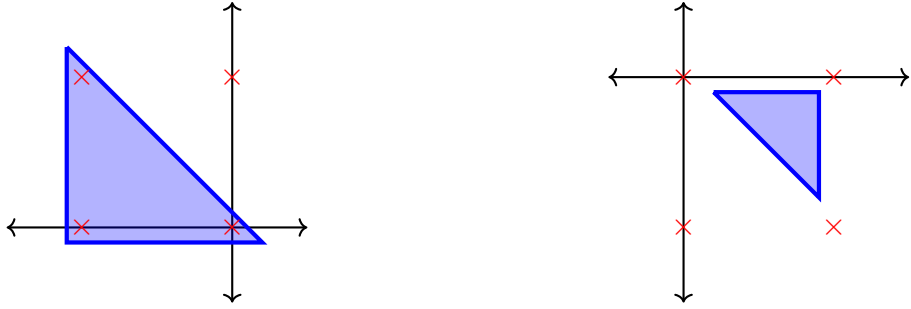

Therefore, we have obtained
\[ H^\bullet \Hom_{\mathcal{F}^s_\Delta (W_\Sigma)} (\mathcal{L} , \mathcal{L}(k D)) = \begin{cases} \bb{C}\langle kP \cap \bb{Z}^n \rangle & \bullet = 0 \\ 0 & \bullet \neq 0 \end{cases} \]
if $k > 0$ and
\[ H^\bullet \Hom_{\mathcal{F}^s_\Delta (W_\Sigma)} (\mathcal{L} , \mathcal{L}(k D)) = \begin{cases} \bb{C}\langle (kP \setminus \partial (kP) ) \cap \bb{Z}^n \rangle & \bullet = n \\ 0 & \bullet \neq 0 \end{cases} \]
if $k < 0$ where $P$ is the polytope associated to $D$. That is, we have obtained, outside of the case $k = 0$, the following proposition.

\begin{prop} \label{ampleexample} Suppose that $\Delta$ is adapted to $\Sigma$ and $D$ is an ample divisor on $X$. We have an isomorphism of graded vector spaces
\[ H^\bullet \Hom_{\mathcal{F}^s_\Delta (W_\Sigma)} (\mathcal{L} , \mathcal{L}(k D)) \cong \Ext^\bullet_X (\mathcal{O}, \mathcal{O}(kD)). \]
\end{prop}

Let us now discuss the case $k = 0$ that has thus far been neglected. If $X$ is Fano, we can of course again apply Lemma \ref{polytope}. In general, the only possible generators of $CF(\psi^{\eps}(\mathcal{L}), \mathcal{L})$ correspond to intersection points of $\eps dK_\Delta$ and the zero section $\bb{R}^n \subset T^*\bb{R}^n$. Moreover, holomorphic discs lift to $T^*\bb{R}^n$. Thus, one can see that
\[ H^\bullet \Hom_{\mathcal{F}^s_\Delta (W_\Sigma)} (\mathcal{L} , \mathcal{L}) \cong HF^\bullet_{T^*\bb{R}^n} ( \eps d K_\Delta, \bb{R}^n) \cong H^\bullet (\bb{R}^n; \bb{C}) = \begin{cases} \bb{C} & \bullet = 0 \\ 0 & \bullet \neq 0 \end{cases} \]
 from the fact that $K_\Delta$ is everywhere outward-pointing at infinity by its definition in Assumption \ref{canassump} (see Section \ref{hms} for more details). In fact, we know that $\mathcal{F}^s_\Delta (W_\Sigma)$ is strictly unital by construction. Hence, $H^0 \Hom_{\mathcal{F}^s_\Delta (W_\Sigma)} (\mathcal{L} , \mathcal{L})$ is generated by the strict unit. 

Note that there is a natural identification 
\[ \Hom_{\mathcal{F}^s_\Delta (W_\Sigma)} (\mathcal{L}(\ell D) , \mathcal{L}(k D)) = \Hom_{\mathcal{F}^s_\Delta (W_\Sigma)} (\mathcal{L} , \mathcal{L}( (k - \ell) D)) \]
up to choice of almost-complex structure by applying the flow of the twisting Hamiltonian $-\ell H_D$ to all data. Thus, we see that Proposition \ref{ampleexample} applies equally well to the morphism spaces $H^\bullet \Hom_{\mathcal{F}^s_\Delta (W_\Sigma)} (\mathcal{L}(\ell D) , \mathcal{L}(k D))$ for all $\ell, k \in \bb{Z}$. 

We can obtain restrictions on the product by a topological argument. For this, we first note that there is an additional $\bb{Z}^n$-grading on $CF^\bullet (L_0 , L_1)$ for two exact Lagrangian sections $L_0$ and $L_1$ in $(\bb{C}^*)^n$ with fixed lifts $\wt L_0$ and $\wt L_1$ to the universal cover $T^*\bb{R}^n$ (cf. Section 3.2 of \cite{Ab2}). The grading $|p|$ of an intersection point $p \in L_0 \cap L_1$ is the unique element of $\bb{Z}^n$ satisfying $\wt p_0 + 2\pi |p| = \wt p_1$ where $\wt p_0 \in \wt L_0$ and $\wt p_1 \in \wt L_1$ are lifts of $p$. By lifting holomorphic discs, we immediately see that the $A_\infty$ structure maps preserve the $\bb{Z}^n$ grading. After identifying generators with integral elements of a polytope as above, it follows that the multiplication map 
\[ m^2\colon CF^\bullet (\psi^{\eps'} \mathcal{L}(kD), \mathcal{L}(\ell D) ) \otimes CF^\bullet ( \psi^{\eps} \mathcal{L}, \psi^{\eps'} \mathcal{L}(kD)) \to CF^\bullet (\psi^{\eps}\mathcal{L}, \mathcal{L}(\ell D) ) \]
for $ \ell > k > 0$ and small $\eps > \eps'$ satisfies $m^2 (p,q) = B_{p,q} (q + p) \in \ell P$ where $q \in kP$, $p \in (\ell - k)P$, and $B_{p,q} \in \bb{Z}$ is the count $\# \mathcal{M}(p, q, q+p; J)$. In fact, we can compute the structure constants $B_{p,q}$ by forcing all these intersection points to lie at the same point in $(\bb{C}^*)^n$ as in the argument below.

\begin{prop} \label{structureconstants} For any $q \in k P$ and $p \in (\ell -k) P$ with $\ell > k > 0$, $B_{p,q}$ as defined above is equal to $\pm 1$.
\end{prop}
\begin{proof} For clarity, we first apply the flow of $-kH_D - \eps'K_\Delta$ to all data so that we are looking at
\[ CF^\bullet( \mathcal{L}, \psi^{-\eps'} \mathcal{L} ( (\ell -k) D)) \otimes CF^\bullet (\psi^{\eps - \eps'} \mathcal{L}(-kD), \mathcal{L} ) \to CF^\bullet( \psi^{\eps - \eps'} \mathcal{L}(-k D),  \psi^{- \eps'} \mathcal{L}( (\ell -k) D)). \]
Lifting to $T^*\bb{R}^n$, we have the Lagrangians $df_1$, the zero section, and $df_2$ where
\[ f_1(\mu) = -k H_D(\mu) + (\eps - \eps')K_\Delta(\mu) + q \cdot \mu \]
and 
\[ f_2(\mu) = (\ell -k)H_D(\mu) - \eps' K_\Delta(\mu) - p \cdot \mu .\]
From this viewpoint, the two intersection points $q$ and $p$ correspond to $(\mu_q, 0)$ and $(\mu_p, 0)$, respectively, such that $\nabla f_1(\mu_q) = \nabla f_2(\mu_p) = 0$, that is, 
\[ k \nabla H_D(\mu_q) + (\eps' - \eps) \nabla K_\Delta ( \mu_q) = q \]
and
\[ (\ell - k) \nabla H_D (\mu_p) - \eps' \nabla K_\Delta(\mu_p) = p. \]

We now consider the functions $\wt{f}_1(\mu) = f_1(\mu - \mu_q)$ and $\wt{f}_1(\mu) = f_1(\mu - \mu_p)$. The gradient images of $\wt{f}_1, \wt{f}_2$, and $\wt{f}_2 - \wt{f}_1$ coincide with those of $f_1$, $f_2$, and $f_2 - f_1$. As a result, the points $p$, $q$, and $q+p$ all now correspond to $(0, 0) \in T^*\bb{R}^n$. Let $L_1$ and $L_2$ be the Lagrangians obtained from projecting $d\wt{f}_1$ and $d\wt{f}_2$ to $(\bb{C}^*)^n$.

We can assume without loss of generality that our monomial division $\Delta$ was a combinatorial division. On any other combinatorial division $\Delta'$ whose cones are all slightly smaller than those of $\Delta$, we then have that $\text{arg}(c_\alpha z^\alpha|_{L_1}) = \eps - \eps'$ and $\text{arg}(c_\alpha z^\alpha|_{L_2}) = -\eps'$ in the complement of a compact subset of $C_\alpha \in \Delta'$ for all $\alpha \in A$. As a result, their (monomially admissible) Floer theory is well-defined. The multiplication
\[ CF^\bullet (\mathcal{L}, L_2) \otimes CF^\bullet (L_1, \mathcal{L}) \to CF^\bullet (L_1, L_2) \]
takes $(p,q)$ to $\pm(q + p)$ because $p$, $q$, and $q + p$ are all the same point forcing all disks to be constant by energy considerations. Now, the Lagrangians $L_1$ and $L_2$ are isotopic through the flow of an admissible Hamiltonian with respect to $\Delta'$ to our original Lagrangians that came from projecting $df_1$ and $df_2$, respectively. Consequently, the Floer cochain complexes are related by continuation maps that are isomorphisms on cohomology, preserve $m^2$ on cohomology, and respect the $\bb{Z}^n$ grading. Since all the complexes we consider have no differentials, it follows that the original count of disks $B_{p,q}$ is also $\pm 1$.
\end{proof}

As observed in Lemma 3.21 of \cite{Ab1}, the sign of $B_{p,q} = \pm 1$ is actually independent of $p$ and $q$, and we can take all the signs to be positive by simply negating the generators if necessary. Therefore, we have computed the homogeneous coordinate ring of $X$ as in \cite{Ab1} using the Lagrangians $\mathcal{L}(kD)$ with $k \in \bb{N}$. Further, we expect that $\mathcal{L}, \mathcal{L}(D), \hdots, \mathcal{L}(nD)$ split-generate $\mathcal{F}^s_\Delta (W_\Sigma)$ and $\mathcal{F}_\Delta (W_\Sigma)$ as the mirror objects $\mathcal{O} , \hdots, \mathcal{O}(nD)$ split-generate $D^b \Coh(X)$ by Theorem 4 of \cite{OrGen}. A proof of this generation result would give a proof of the expected HMS derived equivalence between $\mathcal{F}_\Delta (W_\Sigma)$ and $D^b \Coh(X)$. Instead of pursuing that direction further, we take a different approach in the following subsection by  identifying $\mathcal{F}^s_\Delta (W_\Sigma)$ with the category of line bundles on $X$ following \cite{Ab2}. This has the disadvantage of not explicitly computing the morphism spaces, but the advantage of geometrically identifying the mirrors of more objects.

\subsection{HMS for a monomial division adapted to the fan} \label{hms}

While the preceding subsection was the analogue of part of \cite{Ab1} in our setting, this subsection will be the analogue to \cite{Ab2}. Our goal will to be to prove that $\mathcal{F}^s_\Delta (W_\Sigma)$ is quasi-equivalent to a dg-enhancement of the category of line bundles on $X$ when $\Delta$ is adapted to the fan. The main result of \cite{Ab2} is the same statement with the different but similar category $\mathscr{T}\text{Fuk}((\bb{C}^*)^n, M_{t,1})$ in place of $\mathcal{F}^s_\Delta (W_\Sigma)$.

In fact, we can deduce HMS for $\mathcal{F}^s_\Delta (W_\Sigma)$ when $\Delta$ is adapted to the fan for the case $\mu = \Log$ directly from \cite{Ab2}. Recall that Abouzaid defines a deformation of $W_\Sigma$ to a symplectic fibration $W_{t,1}$ whose fiber $M_{t.1} = W_{t,1}^{-1}(1)$ has a particularly nice combinatorial structure from the data of an ample divsior $\sum \nu(\alpha) D_\alpha$ on $X$ given by a polytope $P$. In particular, the complement of the amoeba $\mathcal{A}_{t,1} = \Log (M_{t,1})$ has a distinguished component $\mathscr{P}_{t,1}$ whose rescaling $\mathscr{P} = \frac{1}{\log(t)} \mathscr{P}_{t,1}$ is contained in and $C^0$-close to $P$ (for large enough $t$). The symplectic fibration (over a neighborhood of $P$) is given by
\[ W_{t,1} = \sum_{\alpha \in A\cup\{0\}} t^{-\nu(\alpha)}(1- \psi_\alpha(z))z^\alpha \]
where the $\psi_\alpha$ are smooth functions depending only on $\Log(z)$ which satisfy (among other properties) $\psi_\alpha(z) = 1$ when $\Log(z)$ is outside of an arbitrarily small neighborhood of the facet dual to $\alpha$ and $\phi_\alpha(z) = 0$ when $\Log(z)$ is in a smaller neighborhood of the dual facet. Abouzaid then defines an object of $\mathscr{T}\text{Fuk}((\bb{C}^*)^n, M_{t,1})$, referred to as a tropical Lagrangian sections, as a Lagrangian section $L$ over $\mathscr{P}$ with boundary in $M_{t,1}$ which is admissible in the sense that there is a neighborhood of $\partial L$ where $L$ agrees with the parallel transport of $\partial L$ along a path from $1$ in $\bb{C}$. Abouzaid defines $\mathscr{T}\text{Fuk}((\bb{C}^*)^n, M_{t,1})$ as an $A_\infty$-pre-category where transverse sequences are given by a positive sequence of tropical Lagrangian sections where a sequence $(L_0, \hdots, L_d)$ of tropical Lagrangian sections is said to be positive if the curves near $1 \in \bb{C}$ given by $\gamma_j = W_{t,1}(L_j)$ have tangent vectors at $1$ that are oriented counterclockwise (see Figure 3 in \cite{Ab2}). Morphisms and higher multiplication maps for transverse sequences are given by counts of holomorphic discs with respect to almost complex structures satisfying the condition that $W_{t,1}$ is $J$-holomorphic in a neighborhood of $M_{t,1}$. 

\begin{prop} Suppose that $\Delta$ is a monomial division adapted to $\Sigma$ for $\mu = \Log$. Then, there is a quasi-equivalence between $\mathcal{F}^s_{\Delta}(W_\Sigma)$ and $\mathscr{T}\text{Fuk}((\bb{C}^*)^n, M_{t,1})$.
\end{prop}
\begin{proof}[Sketch of proof.] In light of Propositions \ref{normalizecoeff} and \ref{adaptedinvariant}, we can assume the coefficients $c_\alpha$ are equal to $1$ and $\Delta$ is a combinatorial division. We can also assume that the $k_\alpha$ in the first condition of Definition \ref{div} are rational as noted before Proposition \ref{adaptedgivesample}. 

Let $Q$ be as in Proposition \ref{adaptedgivesample}. Applying Abouzaid's construction for the ample divisor corresponding to $Q$ allows us to relate monomially admissible Lagrangian sections with respect to $\Delta$ to tropical Lagrangian sections. We choose a collection of monomially admissible Lagrangian sections in each Hamiltonian isotopy class such that each monomially admissible Lagrangian $L$ satisfies $z^\alpha|_{L} \in \bb{R}_{>0}$ in $C_\alpha$ outside of a fixed compact subset inside of $Q$ and far from the boundary of $Q$. By adjusting the cones of the combinatorial division if necessary, we can also assume that each $C_\alpha$ contains the neighborhood of the facet dual to $\alpha$ where $\psi_\alpha(z) \neq 1$. Then, each of the chosen monomially admissible Lagrangian sections restricts to a section over $\mathscr{Q}$, and the values of $W_{t,1}$ on any of these monomially admissible Lagrangians are real and positive in a neighborhood of $\partial \mathscr{Q}$ since we have that $z^\alpha$ is real and positive on the Lagrangian whenever $\psi_\alpha(z) \neq 1$. This gives an embedding of monomially admissible Lagrangian sections into tropical Lagrangian sections that hits every Hamiltonian isotopy class of tropical Lagrangian section by Proposition 3.20 of \cite{Ab2} and our Corollary \ref{twistexist}. 

Upgrading this embedding to a quasi-equivalence between $\mathcal{F}^s_\Delta(W_\Sigma)$ and $\mathscr{T}\text{Fuk}((\bb{C}^*)^n, M_{t,1})$ requires several steps. The first step is to replace $\mathscr{T}\text{Fuk}((\bb{C}^*)^n, M_{t,1})$ by an honest $A_\infty$-category $\mathcal{A}$ defined exactly as $\mathcal{F}^s_\Delta(W_\Sigma)$ except restricting objects to sections over $\mathscr{Q}$ and using almost complex structures as in $\mathscr{T}\text{Fuk}((\bb{C}^*)^n, M_{t,1})$ and with push-off angles near $0$ rather than going all the way to $\pi$. The quasi-equivalence between $\mathcal{A}$ and $\mathscr{T}\text{Fuk}((\bb{C}^*)^n, M_{t,1})$ is then simply a combination of Proposition \ref{prelocal} and a procedure that replaces the radial arcs that objects of $\mathcal{A}$ project to in a neighborhood of $1 \in \bb{C}$ with paths going to $1$ while preserving positivity of the sequence and not creating intersections. 

The most important technical difficulty lies in the next step. We want a quasi-equivalence between $\mathcal{A}$ and an $A_\infty$-category $\mathcal{B}$ whose objects are the objects of $\mathcal{A}$ re-extended to monomially admissible Lagrangian sections and with morphisms computed using monomially admissible almost complex structures. The result follows from this quasi-equivalence as it is easy to see that $\mathcal{B}$ is quasi-equivalent to $\mathcal{F}_\Delta^s(W_\Sigma)$ as the only difference is the sequence of angles chosen. We can first replace $\mathcal{A}$ by the quasi-equivalent category $\mathcal{A}'$ that re-extends the objects to sections over $\bb{R}^n$ but using almost complex structures that make $W_{t,1}$ holomorphic outside of a compact set. We will assume without loss of generality that each almost complex structure for $\mathcal{B}$ is equal to the standard almost complex structure on $(\bb{C}^*)^n$ outside of a compact subset. For each sequence $(\psi^{\theta_0}(L_0), \hdots, \psi^{\theta_d}(L_d))$ of objects in $\mathcal{A}'$ and $\mathcal{B}$ with $\theta_0 > \theta_1 > \hdots > \theta_d$, we have two almost complex structures $J_a$ and $J_b$ used to compute $m^d$ in $\mathcal{A}'$ and $\mathcal{B}$, respectively. We can assume that $J_a$ agrees with $J_b$ in a ball $B$ of radius $R$ (possibly depending on the chosen sequence) in the standard K\"{a}hler metric that contains all of the disks contributing to $m_{\mathcal{B}}^d$. 

We claim that for large enough $R$, all $J_a$-holomorphic discs contributing to $m^d_{\mathcal{A}'}$ will be contained in $B$ and hence are the same discs that contribute to $m^d_{\mathcal{B}}$. To see this, first note that for large enough $R$ Proposition \ref{comp1} implies that there is a compact subset $Z \subset B$ that contains all the intersection points of the Lagrangians in the transverse sequence and all relevant $J_b$-holomorphic discs and is such that the argument of $c_\alpha z^\alpha$ restricted to $\psi^{\theta_j} (L_j)$ is equal to $\theta_j$ in $\Log^{-1}(C_\alpha)$ outside of $Z$. Also, the maximum principle argument of Proposition \ref{comp1} shows that any $J_a$-holomorphic disc not contained in $Z$ must intersect the boundary of $B$. Next, it is possible to check that the union $L = \psi^{\theta_0}(L_0) \cup \hdots \cup \psi^{\theta_d}(L_d)$ is geometrically bounded in $B \setminus Z$ in the sense of Definition 4.7.1 of \cite{Sik} because the distance between the Lagrangians in the sequence is bounded below and the smoothing regions in Proposition \ref{smooth} have fixed normal widths with respect to the Euclidean metric in $\Log$ coordinates so the Hamiltonian whose flow gives $\psi^{\theta_j}(L_j)$ can be taken to have bounded above gradient for each $j$. It then follows from \cite{Sik} that any holomorphic disc satisfies monotonicity estimates bounding its diameter by its energy. On the other hand, we have already noted in Remark \ref{energy1} that we have an a priori energy bound due to exactness of our Lagrangians. Taking $R$ large enough, we can force the diameter of $B \setminus Z$ to be arbitrarily large and hence preclude the existence of any $J_a$-holomorphic discs not contained in $Z$. Thus, we have the desired quasi-equivalence between $\mathcal{A}'$ and $\mathcal{B}$ from an identification of objects and morphisms.
\end{proof}

To prove HMS for adapted monomial divisions in general, we need to adapt the argument in \cite{Ab2} to our setting. The proof in \cite{Ab2} involves the following major steps.

\begin{enumerate}
\item Show that $\mathscr{T}\text{Fuk}((\bb{C}^*)^n, M_{t,1})$ is quasi-equivalent to an $A_\infty$-category of Morse functions.
\item Show that the $A_\infty$-category of Morse functions can be replaced with an $A_\infty$-category where morphisms are computed with simplicial cochains in place of Morse theory.
\item Show that the simplicial cochains in step 2 are equivalent to a \v{C}ech model for the category of line bundles on $X$.
\end{enumerate}

In addition to being imprecise, the outline above mentions $A_\infty$-categories while Abouzaid works with $A_\infty$-pre-categories in \cite{Ab2}. Appendix \ref{precat} shows that this distinction is simply a matter of language for Fukaya-Seidel categories and includes a reminder of the definition of an $A_\infty$-pre-category. We will address the first step in our setting and deduce the other two almost directly from \cite{Ab2}. 

The $A_\infty$ structure on a Morse category was first discovered by Fukaya \cite{FMorse}. As with defining the Fukaya-Seidel category $\mathcal{F}_\Delta (W_\Sigma)$, we will need to modify the standard construction to appropriately account for non-compactness. Let $\Delta$ be a monomial division for a Laurent polynomial $W$ on $(\bb{C}^*)^n$ and moment map $\mu$ for a toric K\"{a}hler form $\omega$ on $(\bb{C}^*)^n$. For simplicity, we assume that $\mu\colon (\bb{C}^*)^n \to \bb{R}^n$ is onto. By Remark \ref{Kpotential}, we can always write $\omega = \omega_\varphi$ and take $\mu = \Phi \circ \Log$ where $\Phi = \nabla \varphi$ is the Legendre transform of $\varphi$. In this context, we are simply assuming $\Phi$ is onto. We will also assume that Assumption \ref{canassump} holds.

\begin{df} A smooth function $f$ on $\bb{R}^n$ is admissible with respect to $\Delta$, or \textit{monomially admissible}, if $\nabla f \cdot \alpha = n_\alpha$ for some $n_\alpha \in 2\pi \bb{Z}^n$ in the complement of a compact subset of $C_\alpha$ for all $\alpha \in A$.
\end{df}

\begin{rem} When $W = W_\Sigma$, monomially admissible functions are simply twisting Hamiltonians.
\end{rem}

Recall that $(\bb{C}^*)^n$ is $T^*\bb{R}^n$ mod a fiberwise lattice and that $\omega = \omega_\varphi$ lifts to an exact symplectic form on $T^* \bb{R}^n$. Given a Riemannian metric $g$ on $\bb{R}^n$, we get an induced almost complex structure $J_g$ on $T^*\bb{R}^n$ (see \cite{FO}). For example, the metric given by
\[ (g_\varphi)_{j\ell} = ( \text{Hess}(\varphi)^{-1} )_{j \ell} \]
in coordinates $(\mu_1, \hdots, \mu_n) = \Phi(\log|z_1|, \hdots, \log|z_n|)$ induces the lift of the standard almost complex structure on $(\bb{C}^*)^n$. Note that 
\begin{equation} \label{gphi} g_\varphi (U, V) = U \cdot d\Phi^{-1}(V) \end{equation}
implying that 
\begin{equation} \label{gradientcomp} \nabla_\varphi f = d\Phi (\nabla f) 
\end{equation}
where $\nabla f$ is the Euclidean gradient in coordinates $(\mu_1, \hdots, \mu_n)$. We will more generally consider the following class of metrics.

\begin{df} A Riemannian metric $g$ on $\bb{R}^n$ is admissible with respect to $\Delta$, or \textit{monomially admissible}, if $J_g$ induces an almost complex structure on $(\bb{C}^*)^n$ that is admissible with respect to $\Delta$.
\end{df}

\begin{rem} The fact that $J_g$ induces an almost complex structure on $(\bb{C}^*)^n$ is automatic in our setting. The condition that the induced almost complex structure is monomially admissible is equivalent to the lift of $c_\alpha z^\alpha$ to $T^*\bb{R}^n$ is $J_g$-holomorphic over the complement of a compact subset of $C_\alpha$ for each $\alpha \in A$. 
\end{rem}

When the lift of $z^\alpha$ to $T^*\bb{R}^n$ is $J_{g}$-holomorphic, we have that
\begin{equation} \label{mong}
g (\alpha, V) = \alpha \cdot d\Phi^{-1}(V)
\end{equation}
and thus 
\begin{equation} \label{gradientcomp2}
d \Phi^{-1}(\nabla_g f) \cdot \alpha = \nabla f \cdot \alpha,
\end{equation}
for any smooth function $f$ generalizing \eqref{gphi} and \eqref{gradientcomp}. Thus, we have that \eqref{gradientcomp2} holds in the complement of a compact subset of $C_\alpha$ for all $\alpha \in A$ if the metric $g$ is monomially admissible.

Now, let $MO_\Delta$ be a countable set of functions admissible with respect to $\Delta$. We can then choose a generic sequence $0 = \theta^0 < \theta^1 < \hdots < \theta^j < \hdots < \pi$ with $\lim_{j \to \infty} \theta^j = \pi$ so that $f_1 - f_0 + (\theta^{j_1} - \theta^{j_0})K_\Delta + m$ is Morse for any $f_0, f_1 \in MO_\Delta$, $j_1 > j_2$, and $m \in 2\pi \bb{Z}^n$ which we identify with the function $m \cdot \mu$ on $\bb{R}^n$. We define an $A_\infty$-category $\Morse_\Delta^\circ$ whose objects are pairs $(f, \theta)$ where $f \in MO_\Delta$ and $\theta \in \{ \theta^j \}$. Morphisms are given by
\[ \Hom_{\Morse^\circ_\Delta} \big((f_0, \theta_0), (f_1, \theta_1) \big) = \begin{cases} \underset{m \in 2\pi\bb{Z}^n}{\bigoplus} CM^\bullet( f_1 - f_0 + (\theta_{1} - \theta_0)K_\Delta + m) & \theta_0 > \theta_1 \\ \bb{K} \cdot e_{(f_0, \theta_0)} & (f_0, \theta_0) = ( f_1, \theta_1) \\ 0 & \text{otherwise} \end{cases} \] 
where $e_{(f_0, \theta_0)}$ is a formal element, $\bb{K}$ is a field (again, $\bb{K} = \bb{C}$ in our application), and $CM^\bullet(\cdot)$ is the Morse co-chain complex of a Morse function. The $A_\infty$ structure is given by setting the $e_{(f,\theta)}$ to be strict units and defining all other nontrivial structure maps by counts of gradient trees as defined by Fukaya \cite{FMorse} (see also Section 4.1 of \cite{Ab2}) with a monomially admissible metric making the sequence of functions involved Morse-Smale in the sense of Definition 4.21 of \cite{Ab2}. Such metrics can be chosen and the moduli spaces of gradient trees can be compactified in the usual way to obtain the $A_\infty$ equations as a consequence of the following proposition.

\begin{prop} \label{Mcomp} Suppose that $f$ is a monomially admissible function, $g$ is a monomially admissible Riemannian metric, and $\eps \in (0, \pi)$. Then there is a compact subset of $\bb{R}^n$ containing the image of any gradient trajectory $\gamma$ of $\wt{f} = f - \eps K_\Delta$ with respect to $g$ that limits to critical points of $\wt{f}$ in any direction in which the domain of $\gamma$ is non-compact.
\end{prop}
\begin{proof} Consider the function 
\begin{equation} \label{mondivfunct} G(\mu) = \max_{\alpha \in A} \Big( k_\alpha (\Phi^{-1}(\mu) \cdot \alpha + \log| c_\alpha |) \Big) \end{equation}
which is the induced function on the base $\bb{R}^n$ from the function in Definition \ref{div}. Note that outside of a compact subset of $\bb{R}^n$ any local max of $G \circ \gamma$ is a local max of $ G_\alpha (t) = k_\alpha (\Phi^{-1}(\gamma(t)) \cdot \alpha + \log| c_\alpha |)$ for some $\alpha \in A$ at a point in the subset of $C_\alpha$ where $\nabla \wt{f}$ is a nonzero constant.  However, we then have that
\[ \frac{d}{dt} G_\alpha(t) = k_\alpha d\Phi^{-1}(\dot{\gamma}(t)) \cdot \alpha = k_\alpha d\Phi^{-1}(\nabla_g \wt f) \cdot \alpha = k_\alpha \nabla \wt f \cdot \alpha \neq 0 \]
using \eqref{gradientcomp2}. Therefore, $G \circ \gamma$ has no local maxima outside of a compact subset $K$ of $\bb{R}^n$ that does not depend on $\gamma$. 

After completing $\gamma$ to include its limit points if its domain is noncompact, we see that the maximum of $G \circ \gamma$ occurs in $K$. Therefore, the image of $\gamma$ is contained in the compact subset $G^{-1}([0, M])$ where $M$ is the maximum of $G$ on $K$. 
\end{proof}

Proposition \ref{Mcomp} is the manifestation of Proposition \ref{comp1} in the Morse theory of monomially admissible functions and metrics. We now have the language to relate the Floer theory of monomially admissible Lagrangian sections to the Morse theory of monomially admissible functions.

\begin{prop} \label{FOprop} Let $\Morse^\circ_\Delta$ be the Morse $A_\infty$-category defined above for some set $MO_\Delta$ of monomially admissible functions. Suppose that $O_\Delta$ is the set of Lagrangian branes obtained from projection to $(\bb{C}^*)^n$ of the graph of $df$ for $f \in MO_\Delta$ and that $\mathcal{F}_\Delta^\circ$ is the associated $A_\infty$-category defined in Section \ref{flocalize}. Then, the natural maps identifying objects and morphisms between $\mathcal{F}_\Delta^\circ$ and $\Morse^\circ_\Delta$ are quasi-isomorphisms. 
\end{prop}
\begin{proof} Let $L_0$ and $L_1$ be Lagrangian sections that are monomially admissible with respect to $e^{-i\theta_0}W$ and $e^{-i\theta_1}W$, respectively, with $\theta_0 \neq \theta_1$. For our application, $L_0$ and $L_1$ will be the images of monomially admissible Lagrangian sections with respect to $W$ by $\psi^{\theta_0}$ and $\psi^{\theta_1}$ with $0 \leq \theta_1 < \theta_0 < \pi$. Given lifts $\wt{L_0}$ and $\wt{L_1}$ to $T^*\bb{R}^n$ of $L_0$ and $L_1$, we have
\[ CF^\bullet (L_0, L_1) = \bigoplus_{m \in 2\pi \bb{Z}^n} CF^\bullet (\wt{L}_0, \wt{L}_1 +m ) \]
when the left-hand side is well-defined and where the addition is fiberwise. Moreover, the $A_\infty$ structures can be computed on the right-hand side and respect the $\bb{Z}^n$-grading as holomorphic discs can be lifted to $T^* \bb{R}^n$. This is the same $\bb{Z}^n$-grading already discussed in Section \ref{floercomp} in a special case. In addition, we can always find $H_0$ and $H_1$ such that $\wt{L_0}$ and $\wt{L_1}$ are the graphs of $dH_0$ and $dH_1$, respectively.

After that observation, the statement is essentially Theorem 2.3 of \cite{FO} in our setting. The main ingredient is another theorem of Fukaya and Oh in \cite{FO}, which was elaborated on and expanded in \cite{ekmorse, ruanmorse}. The theorem states that in the cotangent bundle of a compact manifold for small enough $\eps > 0$, there is an orientation-preserving homeomorphism between the moduli space of $J_{\eps g}$-holomorphic discs with fixed input and output points and boundaries on a collection of Lagrangians that are the graphs of differentials of functions and the moduli space of $\eps g$ gradient trees with the same input and output points and edges labeled by differences of the same functions. 

To adapt the theorem to our setting, it is enough to have a compact subset of the base $\bb{R}^n$ containing all $g$-gradient trees and all projections to $\bb{R}^n$ of $J_{\eps g}$-holomorphic discs for a fixed set of inputs, ouputs, and boundary conditions when $g$ is monomially admissible. For $\eps =1$ and $g$, this is achieved by Proposition \ref{comp1} and Proposition \ref{Mcomp}, respectively. However, $J_{\eps g}$ will not induce a monomially admissible almost complex structure on $(\bb{C}^*)^n$ for $\eps \neq 1$ as 
\[ J_{\eps g} \left( \frac{\partial}{\partial \log |z_j|} \right) = \frac{1}{\eps} J_g  \left( \frac{\partial}{\partial \log |z_j|} \right) \]
for $j = 1, \hdots, n$. However, the function $w^\alpha$ will be $J_{\eps g}$-holomorphic with $w_j = e^{\log|z_j|/\eps + i \theta_j}$ for $j = 1, \hdots, n$ if $z^\alpha$ is $J_g$-holomorphic. As a result, we can apply the same arguments in Proposition \ref{comp1} for $J_{\eps g}$ by replacing each $z^\alpha$ with $w^\alpha$. 

The final step, which accounts for the inability to uniformly choose $\eps$ for all relevant moduli spaces, is to apply the filtration argument in the proof of Corollary 4.36 of \cite{Ab2}. The argument in our setting is exactly the same except we do not work with $A_\infty$-pre-categories and thus one simply sets the multiplication maps to $0$ for the sequences that would be declared not to be transverse in the construction of the filtration.
\end{proof}

\begin{rem} Proposition \ref{FOprop} is the analogue of Theorem 5.8 in \cite{Ab2} in our setting. However, Proposition \ref{FOprop} has a more straightforward proof than Theorem 5.8 in \cite{Ab2} because we can directly apply the Fukaya-Oh theorem while Abouzaid cannot. 
\end{rem}

We immediately have the following.

\begin{cor} \label{FOcor} Let $W= W_\Sigma$ and $\Delta$ be adapted to $\Sigma$. Suppose that $MO_\Delta$ contains one twisting Hamiltonian for each divisor on $X$. Then, $\mathcal{F}^s_\Delta (W_\Sigma)$ is quasi-equivalent to the localization $\Morse_\Delta$ of $\Morse^\circ_\Delta$ at the image of the quasi-units under the quasi-isomorphism of Proposition \ref{FOprop}.
\end{cor}

\begin{rem} As a consequence of the discussion that follows, we will see that the Morse cohomology $HM^\bullet( - K_\Delta)$ is concentrated in degree zero where it has rank one. As a consequence, the localization of $\Morse_\Delta^\circ$ to define $\Morse_\Delta$ in Corollary \ref{FOcor} can be done at chain-level representatives of any non-zero class in $H^\bullet \Hom_{\Morse^\circ_\Delta} ( (f, \theta'), (f,\theta) )$ for all $f \in MO_\Delta$ and $\theta' > \theta$ in $\{ \theta^j \}$. 
\end{rem}

Now that we can work entirely in the setting of Morse theory, we will change to a manifold with corners rather than the noncompact manifold $\bb{R}^n$ in order to more easily describe the Morse cohomology in terms of the behavior of $f$ at infinity and to be able to directly use results in \cite{Ab2}. For this, we assume without loss of generality that there is a compact set $K \subset \bb{R}^n$ such that for every $f \in MO_\Delta$ and for $f = K_\Delta$, $\nabla f \cdot \alpha$ is constant and in $2 \pi \bb{Z}^n$ in $C_\alpha \setminus K$ for all $\alpha \in A$ and that every monomially admissible metric computing morphisms in $\Morse_\Delta^\circ$ satisfies \eqref{gradientcomp2} in $C_\alpha \setminus K$ for all $\alpha \in A$. Note that these assumptions imply that the critical points of every Morse function whose Morse cochain complex appears as a morphism space in $\Morse_\Delta^\circ$ lie in $K$. It also follows by the proof of Proposition \ref{Mcomp} that every gradient tree needed to define $\Morse_\Delta^\circ$ is contained in the compact set $G^{-1}([0, M])$ where $G$ is defined in \eqref{mondivfunct} and $M$ is the maximum value of $G$ on $K$. Given that, we choose $y > M$ and set $Y = G^{-1}([0, y])$. We define $\Morse^\circ_\Delta(Y)$ in the same way as $\Morse^\circ_\Delta$ except restricting all data to $Y$. 

\begin{rem} It would be interesting to show that one could proceed without the restriction to $Y$ and compute the cohomology of line bundles as the cohomology on $\bb{R}^n$ relative to a conical set as it is more traditionally presented (see for instance Section 9.1 of \cite{CLS}). If possible, such a proof would still work if one bounded holomorphic disks in a way other than using the function from the second condition of Definition \ref{div}.
\end{rem}

By construction, it is clear that the restriction map $\Morse^\circ_\Delta \to \Morse^\circ_\Delta(Y)$ is a quasi-isomorphism.  Note that $Y = \Phi(Q)$ where $Q$ is the convex polytope in $\bb{R}^n$ given by 
\[ k_\alpha ( u \cdot \alpha + \log|c_\alpha|) \leq y \]
for all $\alpha \in A$. In particular, $Y$ has the structure of a smooth manifold with corners. Further, we have by construction that no critical points or Morse flow trees involved in the morphism space and structure maps on $\Morse^\circ_\Delta(Y)$ intersect the boundary of $Y$. In fact, the proof of Proposition \ref{Mcomp} shows that no gradient trajectories between interior points of $Y$ intersect $\partial Y$, i.e., the Morse functions that we consider are boundary convex in the sense of Definition 4.2 of \cite{Ab2}. 

\begin{rem} In what follows, we will use results in \cite{Ab2} that are stated for manifolds with boundary instead of corners. This discrepancy is never an issue due to the observation above that gradient trees and critical points involved in our Morse theory do not intersect $\partial Y$. In fact, this makes our setting somewhat simpler than in \cite{Ab2} where critical points on the boundary are allowed. If one still prefers to work with a manifold with boundary rather than corners, then it should be possible to carefully smooth $\partial Y$ by smoothing $\partial Q$ by using a positive symmetric mollifier as in the first half of the proof of Proposition \ref{rhoexists} and carry out all the steps below. 
\end{rem} 

Before proceeding to the categorical statements, it is perhaps useful to first understand the Morse cohomology of the functions we consider. 

\begin{df}[Definition 4.7 of \cite{Ab2}] Suppose that $f$ is a boundary convex Morse function on $Y$. We define $\partial_f^+ Y$ to be the set of points in $\partial Y$ that are limit points of gradient trajectories of $f$ with initial points in the interior of $Y$.
\end{df}

Note that in our intended application, one can replace ``limit point" in the above definition by ``endpoint" since there are no critical points in $\partial Y$. With that notation, we can express the Morse cohomology.

\begin{prop}[Lemma 4.24 of \cite{Ab2}] \label{morsecohom} If $(f,g)$ is a Morse-Smale pair on $Y$ and $f$ is boundary convex, 
\[ HM^\bullet(f) \cong H^\bullet (Y, \partial_f^+ Y). \]
\end{prop}

Note that on the interiors of the maximum dimensional strata of $\partial Y$, a point belongs to $\partial_f^+ Y$ if and only if the gradient $\nabla_g f$ is outward-pointing at that point. In addition, $Y$ has (possibly empty) faces $Y_\alpha = \Phi(F_\alpha)$ for each $\alpha \in A$ where $F_\alpha$ is the face of $Q$ where $k_\alpha ( u \cdot \alpha + \log|c_\alpha|) \leq \log(y)$, and $\cup_{\alpha \in A} Y_\alpha = \partial Y$. As a result of \eqref{mong}, $\alpha$ is an outward-pointing normal to $Y_\alpha$ with respect to a monomially admissible $g$. If in addition, $\nabla f \cdot \alpha$ is constant on $Y_\alpha$ for all $\alpha \in A$, as is the case for the Morse functions defining morphisms in $\Morse_\Delta^\circ (Y)$, then $\text{int}(Y_\alpha) \cap \partial_f^+ Y$ is either empty or equal to $\text{int}(Y_\alpha)$ for each $\alpha \in A$. In particular, the closure of $\partial_f^+ Y$ is the union of all $Y_\alpha$ such that $\nabla f \cdot \alpha$ is positive outside of a compact subset of $C_\alpha$.

Now, we define a simplicial model in order to categorify Proposition \ref{morsecohom}. The analogous $A_\infty$-pre-category is defined in Appendix E of \cite{Ab2} in more generality. We will focus only on the case needed to prove HMS. Let $Y_b$ be the triangulation of $Y$ that is the image of the barycentric subdivision $Q_b$ of $Q$. Given a function $h\colon \partial Y \to \bb{R}$, we define $\partial_h^+ Y_b$ to be the closure of the maximal cells of $Y_b$ in $\partial Y$ where $h$ takes some positive value in the interior. Given a countable set $MO_\Delta$ of monomially admissible functions, we define a (possibly discontinuous) function $h(f)\colon \partial Y \to 2\pi \bb{Z}$ for each $f \in MO_\Delta$ by 
\begin{equation} \label{boundaryfunction}
h(f)(\mu) = \max_{ \{ \alpha \, | \, \mu \in Y_\alpha \} } \nabla f \cdot \alpha
\end{equation}
so that in particular
\begin{equation} \label{boundaryfunction2} h(f) = g(\nabla_g f, \alpha) = \nabla f \cdot \alpha \end{equation}
on the interior of $Y_\alpha$ for each $\alpha \in A$. In the definition of $h(f)$, we use \eqref{boundaryfunction} only to be explicit. The only necessary property is \eqref{boundaryfunction2}. We set $SO_\Delta$ to be the set of $h(f)$ for $f \in MO_\Delta$.  We also define functions $h _\Delta = h(K_\Delta)$ for $K_\Delta$ and $h_m = h(m \cdot \mu)$ for all $m \in 2 \pi \bb{Z}^n$ by \eqref{boundaryfunction}. As before, we fix a generic sequence $0 = \theta^0 < \theta^1 < \hdots < \theta^j < \hdots < \pi$ with $\lim_{j \to \infty} \theta^j = \pi$. 

\begin{rem} Since $\nabla f \cdot \alpha$ is constant on $Y_\alpha$ for all $\alpha \in A$ and $f \in MO_\Delta$, we have that the inclusion $\partial_f^+ Y \subset \partial_{h(f)}^+ Y_b$ is a deformation retract for all $f \in MO_\Delta$. In fact, $\partial_{h(f)}^+ Y_b$ is the closure of $\partial_f^+ Y$. The same statements also hold for the pairs $(K_\Delta, h_\Delta)$ and $(m, h_m)$ and for any linear combination of the aforementioned functions. 
\end{rem}

We now define a dg-category $\Simp^\circ_\Delta (Y)$ whose objects are pairs $(h, \theta)$ where $h \in SO_\Delta$ and $\theta \in \{ \theta^j \}$ with morphisms
\[ \Hom_{\Simp^\circ_\Delta (Y)} ( (h_0, \theta_0), (h_1, \theta_1) ) = \begin{cases} \underset{m \in 2\pi\bb{Z}^n}{\bigoplus} C^\bullet( Y_b, \partial_{h_1 - h_0 + (\theta_{1} - \theta_0)h_\Delta + m}^+ Y_b) & \theta_0 > \theta_1 \\ \bb{K} \cdot e_{(h_0, \theta_0)} & (h_0, \theta_0) = ( h_1, \theta_1) \\ 0 & \text{otherwise} \end{cases} \]
where $e_{(h_0, \theta_0)}$ is a formal element and $\bb{K}$ is a field. The differential is the simplicial cochain differential on each factor of the $\bb{Z}^n$ grading and multiplication is given by the following diagram induced by the inclusion $ \partial_{h_2 - h_0 + (\theta_{2} - \theta_0)h_\Delta + m + \ell}^+ Y_b \subset  \partial_{h_1 - h_0 + (\theta_{1} - \theta_0)h_\Delta + m}^+ Y_b \, \cup \, \partial_{h_2 - h_1 + (\theta_{2} - \theta_1)h_\Delta + \ell}^+ Y_b$.
\[ \begin{tikzcd}[column sep= 1.5 mm]
C^\bullet( Y_b, \partial_{h_2 - h_1 + (\theta_{2} - \theta_1)h_\Delta + \ell}^+ Y_b) \otimes C^\bullet( Y_b, \partial_{h_1 - h_0 + (\theta_{1} - \theta_0)h_\Delta + m}^+ Y_b) \arrow{d}  \\ C^\bullet( Y_b, \partial_{h_1 - h_0 + (\theta_{1} - \theta_0)h_\Delta + m}^+ Y_b \, \cup \, \partial_{h_2 - h_1 + (\theta_{2} - \theta_1)h_\Delta + \ell}^+ Y_b) \arrow{r} & C^\bullet( Y_b, \partial_{h_2 - h_0 + (\theta_{2} - \theta_0)h_\Delta + m + \ell}^+ Y_b)
\end{tikzcd} \]

The following proposition is essentially a consequence of Corollary 4.33 of \cite{Ab2} along with the discussion at the end of Section 6.2 in \cite{Ab2}. 
\begin{prop} \label{MorseSimp} The $A_\infty$-categories $\Morse^\circ_\Delta(Y)$ and $\Simp_\Delta^\circ (Y)$ are quasi-isomorphic when defined using the same set $MO_\Delta$ of monomially admissible functions.
\end{prop}

The quasi-isomorphism takes a pair $(f, \theta)$ to $(h(f), \theta)$. The statement is proved in \cite{Ab2} by first passing through a cellular model. At this point, we are working entirely with topological data. This allows us to remove the formality of localization and work with $Q$ in place of $Y$. 

We define $\Simp_\Delta(Y)$ to be the localization of $\Simp_\Delta^\circ (Y)$ at chain-level representatives with a $\bb{Z}^n$ grading of $0$ (in order to have a $\bb{Z}^n$ grading on the localization) of any non-zero class in 
\[ H^\bullet \Hom_{\Simp_\Delta^\circ (Y)} ( (h, \theta'), (h, \theta)) = H^\bullet (Y ) = \begin{cases} \bb{K} & \bullet = 0 \\ 0 & \bullet \neq 0 \end{cases} \]
 for all $h \in SO_\Delta$ and $\theta' > \theta$ in $\{ \theta^j \}$. For instance, we can simply choose to localize at the strict unit in $C^\bullet(Y_b)$ for all $h \in SO_\Delta$ and $\theta' > \theta$ in $\{ \theta^j \}$.

In another direction, we define a dg-category $\Simp_\Delta(Q)$ whose objects are elements of $SO_\Delta$ with morphisms defined by 
\[ \Hom_{\Simp_\Delta(Q)} (h_0, h_1) = \bigoplus_{m \in 2\pi \bb{Z}^n} C^\bullet ( Q_b, \partial^+_{h_1 - h_0 + m} Q_b) \]
where $\partial_{h}^+ Q_b$ is the closure of the maximal cells of $Q_b$ where $h \circ \Phi$ takes some positive value. The differential and multiplication are defined similarly to $\Simp^\circ_\Delta(Y)$. 

\begin{prop} \label{noloc} The $\dg$-categories $\Simp_\Delta (Y)$ and $\Simp_\Delta(Q)$ are quasi-equivalent when defined using the same set $SO_\Delta$. 
\end{prop}
\begin{proof} We first note that since $Y_b$ is the image of $Q_b$ under $\Phi$, we have $\partial_h^+ Y_b = \Phi(\partial_h^+ Q_b)$ for any function $h\colon \partial Y \to \bb{R}$. Thus, the cochain complexes $C^\bullet(Y_b, \partial_h^+ Y_b)$ and $C^\bullet(Q_b, \partial_h^+ Q_b)$ are identified for any function $h\colon \partial Y \to \bb{R}$. We further observe that if $h$ is constant on the interior of $Y_\alpha$ for all $\alpha \in A$ and $h\colon \partial Y \to 2\pi \bb{Z}$, then 
\[ \partial^+_{h} Y_b = \partial^+_{h + (\theta_1 - \theta_0) h_\Delta} Y_b \]
for any $\theta_0 > \theta_1$ in $\{ \theta^j \}$. As a result, we can define a dg-functor
\[ \Simp_\Delta^\circ (Y) \to \Simp_\Delta (Q) \]
which takes $(h, \theta)$ to $h$ and is defined on morphisms via the identification outlined above between $\partial_{h_1 - h_0 + (\theta_1 - \theta_0)h_\Delta +m}^+ Y_b$ and  $\partial_{h_1- h_0 + m}^+ Q_b$ for each $m \in \bb{Z}^n$ and by taking the $e_{(h,\theta)}$ to the strict unit in $C^\bullet (Y_b)$. 

Further, the set of morphisms which are inverted to construct $\Simp_\Delta(Y)$ are all sent to invertible morphisms in $\Simp_\Delta (Q)$. As a result, we have an induced functor 
\[ \Simp_\Delta(Y) \to \Simp_\Delta(Q) \]
which is a quasi-equivalence by the analogue of Proposition \ref{localhoms} for $\Simp_\Delta(Y)$.
\end{proof}

The final step is to show that $\Simp_\Delta(Q)$ is quasi-equivalent to a dg-enhancement of the category of line bundles on $X$ when $\Delta$ is adapted to the fan and $MO_\Delta$ contains a twisting Hamiltonian for each divisor on $X$. In particular, we will use the dg \v{C}ech category of line bundles on $X$, $\Cech(X)$, defined in Definition 6.8 of \cite{Ab2}. The proof of Proposition \ref{adaptedgivesample} shows that $Q$ has the combinatorial type of the polytope of an ample line bundle on $X$ when $\Delta$ is adapted to $\Sigma$. With that in mind, we note that $\Simp_\Delta(Q)$ could equivalently be defined in this case as having objects all functions  $h\colon \partial Q \to 2\pi\bb{Z}$ that are constant on each face and equal to the max of its values on the intersecting faces elsewhere, cf. \eqref{boundaryfunction}, and morphisms defined in the same way as in $\Simp_\Delta (Q)$ without composing with $\Phi$. 

In addition, we see that $\Simp_\Delta (Q)$ is almost the same as the dg-category $\Simp^{\bb{Z}^n}(Q_b)$ defined in \cite{Ab2} when $\Delta$ is adapted to $\Sigma$. The only differences are that $\Simp^{\bb{Z}^n}(Q_b)$ allows more objects, which are all quasi-isomorphic to corresponding objects of $\Simp_\Delta(Q)$, and that $\Simp_\Delta(Q)$ is defined directly with $Q$ rather than a tropical approximation. In fact, the latter point makes $\Simp_\Delta(Q)$ even closer to the category $\Cech^{\bb{Z}^n}(Q)$ defined in \cite{Ab2} than $\Simp^{\bb{Z}^n}(Q)$. As a result of these observations, the claim in the proof of Proposition 6.7 in \cite{Ab2} and Proposition 6.9 in \cite{Ab2} combine to give the following.

\begin{prop} \label{cechtime} If $\Delta$ is adapted to $\Sigma$ and $MO_\Delta$ contains a twisting Hamiltonian for each divisor on $X$, the dg-categories $\Simp_\Delta(Q)$ and $\Cech(X)$ are quasi-equivalent. 
\end{prop}

Proposition \ref{cechtime} combined with all the previous results of this subsection gives the desired homological mirror symmetry.

\begin{cor} \label{hmscor} If $\Delta$ is adapted to $\Sigma$, $\mathcal{F}^s_\Delta(W_\Sigma)$ is quasi-equivalent to $\Cech(X)$. As a result, the homotopy category of $Tw^\pi \mathcal{F}^s_\Delta (W_\Sigma)$ is equivalent to $D^b \Coh(X)$. 
\end{cor}

Here, $\text{Tw}^{\pi} (\mathcal{A})$ is the idempotent closure of the category of twisted complexes on an $A_\infty$-category $\mathcal{A}$. On objects, the quasi-equivalence of Corollary \ref{hmscor} takes $\phi_D(\mathcal{L})$ to the line bundle $\mathcal{O}(D)$. 

\subsection{Tropical divisions versus combinatorial divisions} \label{tropvcomb}

In the previous subsection, we saw that when $\Delta$ is a combinatorial division or any division adapted to $\Sigma$, $\mathcal{F}^s_\Delta(W_\Sigma)$ is quasi-equivalent to a dg-enhancement of line bundles on $X$. As a result, we expect that whenever $\mathcal{F}_\Delta(W_\Sigma)$ has enough objects, it is derived equivalent to $D^b \Coh(X)$ when $\Delta$ is a combinatorial division. In this subsection, we wish to investigate tropical divisions that are not adapted to the fan such as when $X$ is not Fano.

We will start by looking at the Hirzebruch surfaces $\bb{F}_m$ with $m \geq 2$ and speculate about a more general picture. We work with the fan $\Sigma_m$ of $\bb{F}_m$ whose primitive generators are $(1, m), (0,1), (-1,0)$ and $(0,-1)$. For simplicity, we assume that $\mu = \Log$ and $c_\alpha = 1$ for all $\alpha$. Let $\Delta^t_m$ be the tropical division for $W_{\Sigma_m}$ with $\delta = 0$ or $\delta$ small. The division $\Delta^t_m$ is shown in Figure \ref{hirztrop} for $m = 3$. Temporarily, we will focus on $m \geq 3$ and come back to the somewhat degenerate case of $m = 2$. For $m \geq 3$, $C_{(0,1)}$ is empty. Thus, $\Delta^t_m$ is certainly not adapted to the fan $\Sigma_m$, and we should not expect monomially admissible Lagrangian sections with respect to $\Delta^t_m$ to correspond to line bundles on $\bb{F}_m$. 

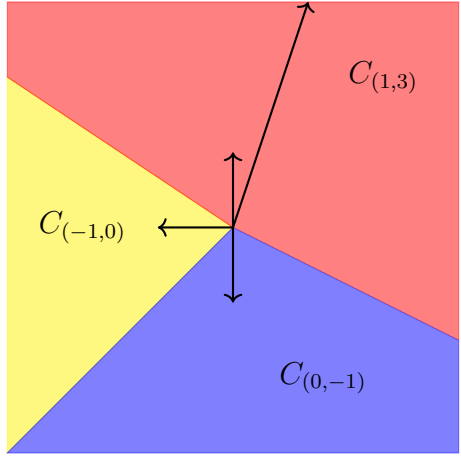
\begin{figure} 
\centering
\begin{tikzpicture}

\draw[yellow, fill=yellow, opacity = 0.5] (0,0) -- (-3, -3) -- (-3, 2) -- (0, 0);
\draw [blue, fill=blue, opacity=0.5] (0,0) -- (3, -1.5) -- (3, -3) -- (-3,-3) -- (0,0);
\draw[red, fill=red, opacity = 0.5] (0,0) -- (-3, 2) -- (-3, 3) -- (3,3) -- (3,-1.5);
\draw (-2, 0) node {$C_{(-1,0)}$};
\draw (1.2, -2) node {$C_{(0,-1)}$};
\draw (2, 2) node {$C_{(1,3)}$};
\draw[->, thick] (0,0) --(1,3);
\draw[->, thick] (0,0) -- (0,1);
\draw[->, thick] (0,0) -- (-1,0);
\draw[->, thick] (0,0) -- (0,-1);

\end{tikzpicture}
\caption{The tropical division for $W_{\Sigma_3} = z_1z_2^3 + z_2 + \frac{1}{z_1} + \frac{1}{z_2}$ overlaid on the fan $\Sigma_3$ with $\mu = \Log$ and $\delta = 0$.} \label{hirztrop}
\end{figure}

However, $\Delta^t_m$ is adapted to the fan $\Sigma_m'$ whose primitive generators are $(1,m), (-1,0)$ and $(0,-1)$, which is a fan for the singular toric variety $\bb{P}^2(1,1,m)$. By the proof of Proposition \ref{normalizecoeff}, the preceding fact is even true for arbitrary $c_\alpha$. Although $\bb{P}^2(1,1,m)$ is not smooth, the fan $\Sigma_m$ is still simplicial. Line bundles on the variety $\bb{P}^2(1,1,m)$ correspond to integral piecewise linear functions on the fan up to an integral linear function. On the other hand, the Hamiltonian isotopy class of a monomially admissible Lagrangian section $L$ is determined by the function $\nu_L$ in \eqref{isotopyfunction} on the generators of $\Sigma_m'$, which can take arbitrary integral values, up to an integral linear function by the arguments in Section \ref{sectionclasses}. Although $\nu_L$ is a piecewise linear function, its linear parts are not necessarily integral. For example, the piecewise linear function corresponding to $\nu_L(-1,0) = 1$ and $\nu_L(0,-1) = \nu_L(1,m) = 0$ is equal to $- \mu_1 + \mu_2/m$ on the cone $\langle (-1,0), (1,m) \rangle$. Therefore, we conclude that there is no natural bijection between line bundles on $\bb{P}^2(1,1,m)$ and monomially admissible Lagrangian sections with respect to $\Delta^t_m$.

The situation can be remedied by replacing $\bb{P}^2(1,1,m)$ with its associated toric stack $\bold{P}^2(1,1,m)$. See Section 2 of \cite{AKO1} for a discussion of this particular stack and \cite{BCS} for a general discussion and definition of toric Deligne-Mumford stacks. Line bundles on $\bold{P}^2(1,1,m)$ are in one-to-one correspondence with integral functions on the primitive generators of $\Sigma_m'$ up to integral linear functions as in the case of a smooth variety. As a result, we have the following.

\begin{prop} \label{hirzbij} (cf. Proposition \ref{bij}) For $m \geq 3$, the Hamiltonian isotopy classes of admissible Lagrangian sections with respect to $\Delta^t_m$ are in bijection with line bundles on the stack $\mathbf{P}^2(1,1,m)$. 
\end{prop}

In the case $m = 2$ with all $c_\alpha = 1$, $C_{(0,1)}$ is no longer empty and is contained in the intersection of $C_{(-1,0)}$ and $C_{(1,2)}$. Thus, we get the condition that $\nu_L(0,1) = \frac{1}{2} (\nu_L(1,2) + \nu_L(-1,0) )$ must be an integer. This is the same condition for $L$ to come from a piecewise integral linear function, i.e., correspond to a line bundle on $\bb{P}^2(1,1,2)$.  Since this does not agree with the general picture above and the cone $C_{(0,1)}$ is degenerate, it seems logical to remove this issue by changing coefficients. For instance, we can set $c_{(0,1)} = e^{-1}$ so that $C_{(0,1)}$ is empty as in the case $m \geq 3$. We will keep this convention from here on out and hence have that Proposition \ref{hirzbij} also holds for $m = 2$. 

We would like to upgrade Proposition \ref{hirzbij} to an equivalence of categories. In fact, we already have the tools to do so. In this particularly simple case, the desired quasi-equivalence can be deduced from the computation in Section \ref{floercomp} as all line bundles on $\mathbf{P}^2(1,1,m)$ are isomorphic to $\mathcal{O}(kD_{(-1,0)})$ for some $k \in \bb{Z}$ and have sections matching with integer points of the corresponding polytope. Alternatively, the methods of Section \ref{hms} also apply in this setting as the moment polytope of $\mathbf{P}^2(1,1,m)$ is simplicial and the cohomology of line bundles on a smooth toric Deligne-Mumford stack can be described analogously to the case of smooth toric varieties as shown in Proposition 4.1 of \cite{BH}. 

\begin{prop} \label{hirztrophms} The monomially admissible Fukaya-Seidel category of Lagrangian sections $\mathcal{F}^s_{\Delta^t_m} (W_{\Sigma_m})$ is quasi-equivalent to a $\dg$-enhancement of the category of line bundles on $\mathbf{P}^2(1,1,m)$. 
\end{prop}

As before, we expect that we should also have a quasi-equivalence between $Tw^\pi\mathcal{F}_{\Delta^t_m}(W_{\Sigma_m})$ and a dg-enhancement of $D^b\Coh(\mathbf{P}^2(1,1,m))$ when $\mathcal{F}_{\Delta^t_m} (W_{\Sigma_m})$ has enough objects. 

This example has been explored before for the usual construction of Fukaya-Seidel categories. In \cite{AKO1}, it is shown using a basis of thimbles that the Fukaya-Seidel category $\mathcal{FS}(W_{\Sigma_m})$ is derived equivalent to the category of coherent sheaves on $\mathbf{P}^2(1,1,m)$. Further, they show that there is a subset of thimbles that generate a subcategory derived equivalent to the category of coherent sheaves on $\bb{F}_m$, that is, mirror to an inclusion $D^b\Coh(\bb{F}_m) \to D^b\Coh(\mathbf{P}^2(1,1,m))$ (which is an equivalence when $m =2$). Thus, Proposition \ref{hirztrophms} suggests that $\mathcal{F}_{\Delta^t_m} (W_{\Sigma_m})$ recovers the Fukaya-Seidel category of $W_{\Sigma_m}$ while $\mathcal{F}_{\Delta} (W_{\Sigma_m})$ for $\Delta$ adapted to $\Sigma_m$ corresponds to a subcategory of the Fukaya-Seidel category mirror to $\bb{F}_m$. It should be noted, however, that the inclusion $D^b\Coh(\bb{F}_m) \to D^b\Coh(\mathbf{P}^2(1,1,m))$ is not induced by a geometric map between $\bb{F}_m$ and $\mathbf{P}^2(1,1,m)$. Moreover, we cannot hope to see the mirror to this inclusion by analyzing only monomially admissible Lagrangian sections with respect to the tropical division and a division adapted to $\Sigma_m$ because the rank of the Picard group of $\bb{F}_m$ is larger than that of $\mathbf{P}^2(1,1,m)$. On the other hand, it is possible to construct an inclusion of $\mathcal{F}^s_{\Delta}(W_{\Sigma_m})$ into the Fukaya-Seidel category of $W_{\Sigma_m}$ by viewing the objects of $\mathcal{F}^s_{\Delta}(W_{\Sigma_m})$ as objects of Abouzaid's category of tropical Lagrangian sections from \cite{Ab2} when $\Delta$ is adapted to the fan. Further explanation of this example in the thimble framework can be found in \cite{BDFKK}. 

The example of monomial admissibility in the mirror to $\bb{F}_m$ suggests a more general picture for the meaning of the tropical division for any $W_\Sigma$ and its relation to the adapted divisions. Our expectation for this picture is summarized in the following two conjectures.

\begin{conj} \label{fsconj} For an appropriate choice of coefficients $c_\alpha$ and/or toric K\"{a}hler form, $\mathcal{F}_{\Delta^t}(W_\Sigma)$ is derived equivalent to the Fukaya-Seidel category of $W_\Sigma$ where $\Delta^t$ is the tropical division for small $\delta$. 
\end{conj}

\begin{conj} \label{tropconj} Suppose the normal fan $\Sigma_{ac}$ to the polytope associated to the anticanonical bundle on $X$ is simplicial. Let $\mathbf{Y}$ be the smooth Deligne-Mumford stack corresponding to $\Sigma_{ac}$. For an appropriate choice of toric K\"{a}hler form and any small $\delta$, the tropical division $\Delta^t$ for $W_\Sigma$ is adapted to $\Sigma_{ac}$ and $\mathcal{F}_{\Delta^t}^s (W_\Sigma)$ is quasi-equivalent to a dg-enhancement of the category of line bundles on $\mathbf{Y}$. 
\end{conj}

\begin{rem} It is not reasonable to hope for Conjecture \ref{fsconj} to hold without some choice of coefficients and toric K\"{a}hler form. The case of $\bb{F}_2$ investigated above shows that the choice of coefficients can affect the quasi-equivalence type of $\mathcal{F}^s_{\Delta^t} (W_\Sigma)$. Moreover, we saw in Section \ref{existencesect} that the toric K\"{a}hler form plays an essential role in the behavior of a monomial division through the moment map. In fact, Proposition \ref{rhoexists} shows that the first part of Conjecture \ref{tropconj} is true in the sense that there is always a radial toric K\"{a}hler form $\omega_\varphi$ on $(\bb{C}^*)^n$ for which $\Delta^t$ is adapted to $\Sigma_{ac}$ for $\mu = \Phi \circ \Log$ and small $\delta$.
\end{rem}

From Conjectures \ref{fsconj} and \ref{tropconj}, we expect that when the anticanonical normal fan $\Sigma_{ac}$ is simplicial then the Fukaya-Seidel category $\mathcal{FS}(W_\Sigma)$ is derived equivalent to $D^b \Coh(\textbf{Y})$ where $\textbf{Y}$ is the smooth Deligne-Mumford stack associated to $\Sigma_{ac}$ as also suggested by the conjectures in \cite{BDFKK}. 

\section{Monodromy} \label{4}

In this section, our main goal is to upgrade the action of Corollary \ref{action} to an action on $\mathcal{F}_\Delta (W_\Sigma)$ for a monomial division $\Delta$ adapted to $\Sigma$ that is mirror to the action of $\text{Pic}(X)$ on $D^b\Coh(X)$. Since this action comes from a group of Hamiltonian diffeomorphisms, each functor comes with a natural transformation from the identity functor when the Hamiltonian is sufficiently negative, i.e., is a twisting Hamiltonian for an effective divisor. We compare these natural transformations with the natural transformation from the identity to $(\cdot) \otimes \mathcal{O}(D)$ given by multiplication by a defining section when $D$ is effective. Finally, we explore some applications of the monodromy and natural transformations. 

\subsection{Monodromy on the monomially admissible Fukaya-Seidel category} \label{catmono}

It is well-known that symplectomorphisms act on the Fukaya category. Thus, the action of twisting Hamiltonians on $\mathcal{F}_\Delta(W_\Sigma)$ should be entirely expected, but care must be taken in the precise setup. While the compact Fukaya category is invariant under Hamiltonian isotopy, the twisting Hamiltonians are not compactly supported and act nontrivially on the quasi-isomorphism classes of objects by Corollaries \ref{action} and \ref{hmscor}.

For Lagrangian branes $L_0$ and $L_1$ and a symplectomorphism $\phi$, the complexes $CF^\bullet(L_0, L_1)$ and $CF^\bullet(\phi(L_0), \phi(L_1))$ are naturally identified when the latter is computed with the almost complex structure $\phi_*(J)$. To take advantage of this fact, we would first like to verify that twisting Hamiltonians preserve the class of monomially admissible almost complex structures.

\begin{lem} \label{preservesJ} If $\phi$ is the time-$1$ flow of a twisting Hamiltonian $H$ for some toric divisor with respect to a monomial division $\Delta$ for $W_\Sigma$ then $\phi_*$ preserves the class of monomially admissible almost complex structures with respect to $\Delta$. 
\end{lem}
\begin{proof} Suppose that $J$ is a monomially admissible almost complex structure so that $z^\alpha$ is $J$-holomorphic in $\mu^{-1}(C_\alpha)$ for each $\alpha \in A$ in the complement of a compact sbuset. We will show that each $z^\alpha$ is $\phi_*J$-holomorphic in $\mu^{-1}(C_\alpha)$ outside a compact subset.

A function is $\phi_*J$-holomorphic if and only if its composition with $\phi$ is $J$-holomorphic. However, we can write $z^\alpha = f(\mu) e^{i \theta \cdot \alpha}$ for some smooth function $f$. Thus, we have
\[ z^\alpha \circ \phi = f(\mu) e^{i (\theta + \nabla H) \cdot \alpha} =  z^\alpha \]
in the complement of a compact subset of $\mu^{-1}(C_\alpha)$ where $\nabla H \cdot \alpha$ is an integer.  
\end{proof}

In general, one needs to use continuation morphisms or an algebraic setup as in Section 10b of \cite{SeBook} to construct an action of a group of symplectomorphisms on a Fukaya category. In our setting, we can simplify the setup of the group action if we are careful. First, let $O_\Delta^o$ be a countable set of monomially admissible Lagrangian branes. Then, choose a basis and representatives $\{ H_1, \hdots, H_m \}$ for the free abelian group that is the quotient of the twisting Hamiltonians that exist for $\Delta$ by the subgroup of admissible Hamiltonians. For instance, the basis can be chosen as the Hamiltonians corresponding to the divisors $D_\alpha$ for all $\alpha \in A$ when $\Delta$ is adapted to $\Sigma$. Let $G$ be the free abelian group generated by $\phi_{H_1}, \hdots, \phi_{H_m}$. Note that $G$ is isomorphic to the group of twisting Hamiltonians mod admissible Hamiltonians that we used to begin this construction. For a generic choice of basis, we can assume that $\phi(L) \neq L$ for all $L \in O_\Delta^o$ and all nontrivial $\phi \in G$. Then, we set 
\[ O_\Delta = \{ \phi(L) | L \in O_\Delta^o, \phi \in G \} \]
and construct the category $\mathcal{F}_\Delta(W_\Sigma)$ as in Section \ref{flocalize} except that we choose the almost complex structures to be equivariant with respect to the $G$ action, which is free on objects of $\mathcal{F}_\Delta^\circ$. Note that every $\phi \in G$ commutes with $\psi^\theta$ for any $\theta$ so the action on objects of $\mathcal{F}^\circ_\Delta$ is defined to only change the Lagrangian brane and not the angle. 

As a result of the observation above on the identification of Floer cochain complexes, we obtain a strict action of $G$ on $\mathcal{F}_\Delta^\circ$ with no higher terms. More formally, for every $\phi \in G$, we have an $A_\infty$-functor $\mathscr{F}_\phi\colon~\mathcal{F}_\Delta^\circ \to \mathcal{F}_\Delta^\circ$ defined by $\mathscr{F}_\phi(L, \theta) = (\phi(L), \theta)$, $\mathscr{F}_\phi^1$ is the canonical identification of $CF(\psi^{\theta_0}(L_0), \psi^{\theta_1}(L_1))$ with $CF(\phi \circ \psi^{\theta_0} (L_0),  \phi \circ \psi^{\theta_1}(L_1))$ for $\theta_0 > \theta_1$, $\mathscr{F}_\phi^1(e_{(L, \theta)}) = e_{(\phi(L), \theta)}$,  and $\mathscr{F}_\phi^d = 0$ for $d \geq 2$. These functors satisfy $\mathscr{F}_{\id} = \id_{\mathcal{F}_\Delta^\circ}$ and $\mathscr{F}_{\phi_2 \phi_1} = \mathscr{F}_{\phi_2} \circ \mathscr{F}_{\phi_1}$. 

If we also equivariantly choose the almost complex structures defining quasi-units, the $\mathscr{F}_\phi$ will take quasi-units to quasi-units on the chain level (instead of on the cohomology level for a more general choice). As a consequence, we can define (cf. Proposition 3.1 of \cite{LO}) $A_\infty$-functors 
\[ F_\phi\colon \mathcal{F}_\Delta (W_\Sigma) \to \mathcal{F}_\Delta (W_\Sigma) \]
 for each $\phi \in G$ which agree with $\mathscr{F}_\phi$ on objects and are defined on morphism spaces
\[ \Hom_{\mathcal{F}_\Delta (W_\Sigma)} ( Z_0, Z_1) =  \bigoplus_{\substack{ p \geq 0 \\ Q_1, \hdots, Q_p \in \text{Cones}(Q)}} \Hom_{\text{Tw}\mathcal{F}_\Delta^\circ} (Q_p, Z_1)[1] \otimes \hdots \otimes \Hom_{\text{Tw}\mathcal{F}_\Delta^\circ} (Z_0, Q_1) \]
for $Z_0, Z_1 \in \text{Ob}(\mathcal{F}_\Delta (W_\Sigma))$ by linearly extending
\[ F^1_\phi (x_p \otimes \hdots \otimes x_0) = \mathscr{F}^1_\phi (x_p) \otimes \hdots \otimes \mathscr{F}^1_\phi(x_0) \]
when $x_j \in \Hom_{\text{Tw}\mathcal{F}_\Delta^\circ} (Q_j, Q_{j+1})[1]$ for $j = 1, \hdots, p$ with $Q_{p+1} = Z_1$ and  $x_0 \in \Hom_{\text{Tw}\mathcal{F}_\Delta^\circ} (Z_0, Q_1)$. We are using crucially here that the functors induced by the $\mathscr{F}_\phi$ on $\text{Tw}\mathcal{F}_\Delta^\circ$ preserve $\text{Cones}(Q)$. Without making equivariant choices to ensure the preservation of $\text{Cones}(Q)$, one could construct functors using the universal property of localization up to quasi-isomorphism. However, the result would not be a strict action in the sense that $F_{\phi_2 \phi_1}$ would only be guaranteed to be quasi-isomorphic to $F_{\phi_2} \circ F_{\phi_1}$ and $F_{\id}$ would not have to be chosen to be the identity. With our construction, we have a strict $G$-action on $\mathcal{F}_\Delta (W_\Sigma)$ and the diagram
\[ \begin{tikzcd} \mathcal{F}_\Delta^\circ \arrow{r}{\mathscr{F}_\phi} \arrow{d} & \mathcal{F}_\Delta^\circ \arrow{d} \\
 \mathcal{F}_\Delta(W_\Sigma) \arrow{r}{F_\phi} & \mathcal{F}_\Delta(W_\Sigma) \end{tikzcd} \]
commutes or all $\phi \in G$.
In fact, we could have constructed an action by any countable group of twisting Hamiltonians. In particular, we have shown the following.

\begin{prop} The monodromy in \eqref{monodromyeq} induces a functor on $\mathcal{F}_\Delta (W_\Sigma)$ whenever a twisting Hamiltonian for $D$ exists and $O_\Delta$ is closed under its flow.
\end{prop}

We can also see that our choice of representative twisting Hamiltonian up to admissible Hamiltonian isotopy did not matter. In fact, it is enough for the piecewise linear approximations to differ by a linear map.

\begin{prop} \label{actionclasses} If the difference of two twisting Hamiltonians $H_1, H_2$ is linear up to an admissible Hamiltonian, then the induced functors of $H_1$ and $H_2$ on $\mathcal{F}_\Delta (W_\Sigma)$ are quasi-isomorphic.
\end{prop}
\begin{proof} The Hamiltonian $H_1 - H_2$ is equal to the sum of an admissible Hamiltonian $H$ and an integral linear function on $\bb{R}^n$ (the fact that it is integral follows from $X$ being smooth). Since the time-$1$ flow of an integral linear function is the identity on $(\bb{C}^*)^n$, it is enough to show that the functor induced by $H$ is quasi-isomorphic to the identity. That essentially follows from the fact that continuation maps can be defined for both $H$ and $-H$ using Proposition \ref{comp2} and applying the same arguments that show that functors induced by Hamiltonian isotopies on a compact Fukaya category are quasi-isomorphic to the identity as in Proposition 10.3 of \cite{SeBook}. More accurately, one needs to define a natural transformation from the identity to the functor induced by $H$ for which multiplication by the zeroth order term is the invertible continuation map for $H$. Such natural transformations are defined in Section \ref{defsect}.
\end{proof}

The main subject of the rest of this subsection will be to understand the monodromy action of twisting Hamiltonians on Lagrangian sections described in the example below under the homological mirror symmetry established in Corollary \ref{hmscor}. Of course, we expect our description of the action to generalize to objects other than sections due to the expected generation by sections. 

\begin{ex} \label{tensoronlines} Suppose that $\Delta$ is adapted to the fan. If we take $O_\Delta^o = \{ \mathcal{L} \}$ to only include $\mathcal{L} = (\bb{R}_{>0})^n$ with its canonical grading as a section, then $O_\Delta$ will consist of Lagrangians $\phi_D(\mathcal{L})$ for each toric divisor $D$ on $X$. In this setting, the construction above produces an action of $G$ on $\mathcal{F}^s_\Delta (W_\Sigma)$ given by a functor $F_D = F_{\phi_D}$ for each toric divisor $D$ on $X$. Note that in this case, $O_\Delta$ is not in bijection with the quasi-isomorphism types of objects of $\mathcal{F}^s_\Delta(W_\Sigma)$. If we instead perform the construction with a basis of twisting Hamiltonians up to admissible Hamiltonians and integral linear functions, we obtain $\mathcal{F}^s_\Delta(W_\Sigma)$ from an $O_\Delta$ in bijection with $\Pic(X)$ and with an action of $\Pic(X)$ by Proposition \ref{bij} and Remark \ref{twistandpic}.
\end{ex}

By Proposition \ref{actionclasses}, there is really no difference in the information encoded by either setup given in Example \ref{tensoronlines}. Both have the expected effect under homological mirror symmetry. As a first example, $F_D$ takes $\mathcal{L}(kD)$ to $\mathcal{L}( (k + 1) D)$ in the setup of Section \ref{floercomp} for any $k \in \bb{Z}$. In that setting, $F_D$ acts on morphisms as the canonical identification between $\Hom_{\mathcal{F}^s_\Delta(W_\Sigma)} (\mathcal{L}(kD), \mathcal{L}(\ell D))$ and $\Hom_{\mathcal{F}^s_\Delta(W_\Sigma)} (\mathcal{L}( (k+1) D), \mathcal{L}( (\ell+1) D))$. In other words, the polytopes of $\mathcal{O}(\ell D) \otimes \mathcal{O}(-k D)$ and $\mathcal{O}( (\ell +1) D) \otimes \mathcal{O} ( -(k+1)D)$ are the same. More generally, we have the following.

\begin{thm} \label{mirrorfunctors} The functor $F_D$ from Example \ref{tensoronlines} is mapped by the quasi-equivalence of Corollary \ref{hmscor} to the functor $( \cdot) \otimes \mathcal{O}(D)$ on $\Cech(X)$. 
\end{thm}
\begin{proof} The quasi-equivalence of Corollary \ref{hmscor} extends the correspondence between Hamiltonian isotopy classes of monomially admissible Lagrangian sections with line bundles on $X$ from Proposition \ref{bij} as it takes $\phi_D(\mathcal{L})$ to $\mathcal{O}(D)$. Thus, the functor $F_D$ acts as $(\cdot) \otimes \mathcal{O}(D)$ on objects by Corollary \ref{action}.

On morphisms, $F_D$ is the canonical identification 
\[ \Hom_{\mathcal{F}^s_\Delta(W_\Sigma)} ( L_0, L_1) \cong \Hom_{\mathcal{F}^s_\Delta(W_\Sigma)} ( \phi_D(L_0) , \phi_D(L_1) )\]
which after passing to $\Morse_\Delta$ using Corollary \ref{FOcor} simply becomes the identity as 
\[ CM^\bullet(f_1 - f_0 + (\theta_1 - \theta_0)K_\Delta + m) = CM^\bullet( (f_1 + H_D) - (f_0 + H_D) + (\theta_1 - \theta_0)K_\Delta + m) \]
for any monomially admissible functions $f_0, f_1$, any angles $\theta_1, \theta_0$, and any $m \in \bb{Z}^n$. In \cite{Ab2}, morphisms are defined in $\Cech(X)$ by
\[ \Hom_{\Cech(X)} (L_0, L_1) = \bigoplus_{m \in \bb{Z}^n} \check{C}^\bullet_m (L_1 \otimes L_0^{-1}) \]
where $\check{C}^\bullet_m(L_1 \otimes L_0^{-1})$ is a \v{C}ech complex that satisfies 
\[ \check{C}^\bullet_m(L_1 \otimes L_0^{-1}) = \check{C}^\bullet_m( (L_1 \otimes \mathcal{O}(D)) \otimes (L_0 \otimes \mathcal{O}(D))^{-1}) \]
for every toric divisor $D$ on $X$. Moreover, the sequence of quasi-equivalences from $\Morse_\Delta$ to $\Cech(X)$ in Section \ref{hms} preserves these equalities.
\end{proof}

Note that in Example \ref{tensoronlines}, we can associate a functor to each object of $\mathcal{F}^s_{\Delta}(W_\Sigma)$ by assigning $F_D$ to $( \phi_D(\mathcal{L}), \theta)$ for any $\theta \in \{ \theta^j \}$. Because the $F_D$ come from a strict action of a free abelian group, this induces a symmetric monoidal structure on the homotopy category $H^0 (\mathcal{F}^s_\Delta(W_\Sigma))$ defined by
\[ \otimes \colon \Big( ( \phi_{D_0}( \mathcal{L}) , \theta_0) , (\phi_{D_1} (\mathcal{L}), \theta_1) \Big) \to F_{D_0}  (\phi_{D_1}( \mathcal{L}), \theta_1) =  ( \phi_{D_0 + D_1} (\mathcal{L}),  \theta_1 ) \]
with an analogous definition on morphisms and with natural isomorphisms
\[ ( \phi_{D_0}( \mathcal{L}) , \theta_0) \otimes (\phi_{D_1} (\mathcal{L}), \theta_1) \to ( \phi_{D_1}( \mathcal{L}) , \theta_1) \otimes (\phi_{D_0} (\mathcal{L}), \theta_0) \]
given by multiplication by quasi-units or their inverses. 

As a result of Theorem \ref{mirrorfunctors}, the homological mirror symmetry of Corollary \ref{hmscor} induces an equivalence of monoidal categories.

\begin{cor} \label{monoidalstr} The equivalence of homotopy categories induced by the quasi-equivalence of Corollary \ref{hmscor} is an equivalence of monoidal categories when $H^0 (\mathcal{F}^s_\Delta(W_\Sigma))$ is equipped with the monoidal structure defined above and the category of line bundles on $X$ with its usual monoidal structure.
\end{cor}

It should be noted that in the microlocal homological mirror symmetry for toric varieties introduced by Fang-Liu-Treumann-Zaslow in \cite{FLTZ1, FLTZ2, FLTZ3, FLTZS} and proved in full generality by Kuwagaki in \cite{Ku2}, there is also an induced monoidal equivalence with respect to a monoidal structure on the category of microlocal sheaves. Also, the monoidal structure that we defined above fits well with the work of Subotic in \cite{Subotic} where a monoidal structure for Fukaya categories is described as fiberwise addition in the presence of a Lagrangian torus fibration with a reference section (in our case, these are the moment map $\mu$ and the section $\mathcal{L}$, respectively).

\subsection{Defining sections of line bundles from natural transformations} \label{defsect}

In the compact setting, all functors induced by Hamiltonian isotopies on the Fukaya category come with natural transformations from the identity as in Section 10c of \cite{SeBook}. Moreover, the natural transformations are quasi-isomorphisms with inverse given by the natural transformation for the inverse Hamiltonian isotopy. In the noncompact setting, such natural transformations are only defined when $H$ has certain asymptotic behavior, and the natural transformations that exist have more interesting geometric meaning as we will illustrate in our setting.

Given a twisting Hamiltonian $H_D$ for an effective divisor $D$ on $X$, we wish to define a natural transformation from the identity to the functor $F_D = F_{\phi_D}$ defined in Section \ref{catmono}. In particular, we will be assuming that $\mathcal{F}_\Delta(W_\Sigma)$ is set up as in Section \ref{catmono}. Roughly, we wish to define maps
\[ CF^\bullet(L_{d-1}, L_d) \otimes \hdots \otimes CF^\bullet(L_0, L_1) \to CF^\bullet(L_0, \phi_D(L_d))[-d] \]
for $d \geq 0$ by counting punctured discs with moving Lagrangian boundary condition determined by $H$ using the approach in Section 10c of \cite{SeBook}. The fact that $D$ is effective is a key condition that implies the inequality $\nabla H \cdot \alpha \leq 0$ in the complement of a compact subset of $C_\alpha$ for all $\alpha \in A$ needed to use the compactness results established in Section \ref{control}. We will need to be careful to define the natural transformations compatibly with localization. For example, a natural transformation from the identity to $F_D$ cannot be constructed on the level of $\mathcal{F}^\circ_\Delta$ as the $0$th order term would necessarily vanish.

Now, we begin to make things more precise. We first slightly modify our definition of $\mathcal{F}_\Delta(W_\Sigma)$. For any angle $\zeta \in (0, \pi]$, we define $\mathcal{F}_\Delta^\circ (\zeta)$ and $\mathcal{F}_\Delta (W_\Sigma, \zeta)$ in the same way as $\mathcal{F}_\Delta^\circ$ and $\mathcal{F}_\Delta(W_\Sigma)$ except replacing the sequence $\{ \theta^j \}$ of angles with all angles in $\mathbb{Q} \cap [0, \zeta]$. This modification poses no problem for our transversality assumptions as the set of angles is still countable.  As a consequence of Proposition \ref{localhoms} and the arguments in the proof of Proposition \ref{choices}, the quasi-equivalence type of $\mathcal{F}_\Delta (W_\Sigma, \zeta)$ is independent of $\zeta$, and all the $\mathcal{F}_\Delta (W_\Sigma, \zeta)$ are quasi-equivalent to $\mathcal{F}_\Delta (W_\Sigma)$ when constructed from the same set $O_\Delta$ of Lagrangian branes.

Next, we will define $A_\infty$-functors 
\[ P_\eps\colon \mathcal{F}_\Delta (W_\Sigma, \zeta) \to \mathcal{F}_\Delta (W_\Sigma, \zeta + \eps) \]
for every $\eps \in \mathbb{Q} \cap (0, \pi - \zeta]$ that take an object $(L, \theta)$ to $(L, \theta + \eps)$. To that end, we assume that our monomially admissible almost complex structures are chosen so that
\[ J^{\zeta + \eps}_{(L_0, \theta_0 + \eps), \hdots, (L_d, \theta_d + \eps)} = (\psi^\eps)_* J^{\zeta}_{(L_0, \theta_0), \hdots, (L_d, \theta_d)} \]
where $J^{\zeta}_{(L_0, \theta_0), \hdots, (L_d, \theta_d)}$ is the monomially admissible almost complex structure used in computing $m^d_{\mathcal{F}^\circ_\Delta}$ on the sequence of objects $(L_0, \theta_0), \hdots, (L_d, \theta_d)$ with $\theta_0 > \hdots > \theta_d$. This allows us to define functors
\[ \mathscr{P}_\eps \colon \mathcal{F}_\Delta^\circ(\zeta) \to \mathcal{F}_\Delta^\circ(\zeta + \eps) \]
analogously to how the functors $\mathscr{F}_\phi$ were defined in Section \ref{catmono}. That is, $\mathscr{P}_\eps$ takes $(L, \theta)$ to $(L, \theta + \eps)$, $\mathscr{P}^1_\eps$ is the identification
\[ CF^\bullet(\psi^{\theta_0}(L_0), \psi^{\theta_1}(L_1)) \cong CF^\bullet(\psi^{\theta_0+ \eps}(L_0), \psi^{\theta_1+\eps}(L_1))\]
 through $\psi^\eps$, and there are no higher terms. If the quasi-units in $\mathcal{F}_\Delta^\circ(\zeta+\eps)$ are also defined using pushforward almost complex structures, then the $\mathscr{P}_\eps$ will preserve the quasi-units on the chain level and we can define $P_\eps$ from $\mathscr{P}_\eps$ explicitly in the same way the $F_\phi$ are defined from $\mathscr{F}_\phi$ in Section \ref{catmono}. Without that assumption, we could still define $P_\eps$ from the universal property of localization. 

We now wish to define natural transformations 
\[ \mathcal{T}_{\eps, \phi} \colon \mathscr{P}_\eps \to \mathscr{F}_\phi \]
 where $\phi = \phi_D$. Here and for the rest of this section, we abuse notation to view $\mathscr{F}_\phi$ as a functor from $\mathcal{F}_\Delta^\circ(\zeta)$ to $\mathcal{F}_\Delta^\circ (\zeta + \eps)$ by composing $\mathscr{F}_\phi$ on $\mathcal{F}^\circ_\Delta (\zeta)$ with the inclusion of $\mathcal{F}_\Delta^\circ(\zeta)$ into $\mathcal{F}_\Delta^\circ(\zeta + \eps)$. Following Section 10c of \cite{SeBook}, we let $\mathcal{R}_{d+1, 1}$ be the moduli space of discs with $d+1$ boundary punctures $z_0, \hdots, z_d$ ordered counterclockwise and an interior marked point and let $\mathcal{S}_{d+1, 1}$ be the universal family over $\mathcal{R}_{d+1, 1}$. These spaces have compactifications $\overline{\mathcal{R}}_{d+1,1}$ and $\overline{\mathcal{S}}_{d+1,1}$ that consist of broken discs and are constructed similarly to $\overline{\mathcal{R}}_{d+1}$ and $\overline{\mathcal{S}}_{d+1}$ except with extra degenerations from the interior marked point. We then need to choose a family of data on $\overline{\mathcal{S}}_{d+1,1}$ which amounts to choosing on each fiber $S$ in $\overline{\mathcal{S}}_{d+1, 1}$  strip-like ends at each puncture (outgoing at $z_0$ and incoming at all other boundary punctures), numbers $0 = s_0 \leq s_1 \leq \hdots \leq s_{d+1} = 1$, a function $f: \partial S \to [0,1]$ such that $f$ is constant in the strip-like ends and monotonically increases from $s_j$ to $s_{j+1}$ on the boundary component between $z_j$ and $z_{j+1}$ for $j =0, \hdots , d$ with $z_{d+1} = z_0$, and a one-form $\beta$ that vanishes in the strip-like ends and satisfies $d\beta \geq 0$ and $\beta|_{\partial S} = df$. This family of data needs to be chosen compatibly with degenerations at the boundary. In addition to the more standard requirements, the compatibility includes that the component containing the interior marked point must contain the support of $\beta$ and all boundary components where $f$ is not constant. We also choose for each sequence $(L_0, \theta_0), \hdots, (L_d, \theta_d)$ of objects of $\mathcal{F}_\Delta^\circ(\zeta)$ with $\theta_0 > \hdots > \theta_d$ a generic family of monomially admissible almost complex structures over $\overline{\mathcal{R}}_{d+1,1}$ such that on each fiber $S$ the almost complex structure coincides with
\[ (\phi^{s_j} \psi^{(1-s_j)\eps})_* J^\zeta_{(L_{j-1}, \theta_{j-1}), (L_{j}, \theta_{j})} \]
on the incoming strip-like end near $z_j$ for $j =1 , \hdots, d$ and coincides with
\[ J^{\zeta + \eps}_{(L_0, \theta_0 + \eps), (\phi(L_d), \theta_d)}\]
on the outgoing strip-like end. The family of almost complex structures must also be chosen compatibly with boundary degenerations. 

With those choices in hand, we consider the moduli spaces of pairs of a point $r \in \overline{\mathcal{R}}_{d+1,1}$ and a map $u: S_r \to (\bb{C}^*)^n$ such that $S_r$ is the fiber of $\overline{\mathcal{S}}_{d+1, 1}$ over $r$ and $u$ is a solution to \eqref{perturb} with $H = H_D - \eps K_\Delta$, all other data from the chosen families, and the boundary condition that the boundary component between $z_j$ and $z_{j+1}$ is mapped to $\phi^{f(s)} \circ \psi^{\theta_j + (1 - f(s))\eps} (L_j)$ for $j = 1, \hdots, d$ again with $z_{d+1} = z_0$. Because $D$ is effective and $d\beta \geq 0$, we can use Remark \ref{moving} and Gromov compactness to see that these moduli spaces are compact. Counting the zero-dimensional part of these moduli spaces appropriately, we obtain degree $-d$ maps
\[ \mathcal{T}_{\eps, \phi}^d \colon CF^\bullet( \psi^{\theta_{d-1}} (L_{d-1}) , \psi^{\theta_d}(L_d)) \otimes \hdots \otimes  CF^\bullet (\psi^{\theta_0} (L_0), \psi^{\theta_{1}} (L_1)) \to CF^\bullet (\psi^{\theta_0 + \eps}(L_0), \phi \circ \psi^{\theta_d}(L_d) ) \]
using the identification 
\[ CF^\bullet (\psi^{\theta_{j}} (L_{j}), \psi^{\theta_{j+1}} (L_{j+1})) \cong CF^\bullet (\phi^{s_{j+1}} \circ \psi^{\theta_{j} + (1 - s_{j+1})\eps} (L_{j}), \phi^{s_{j+1}}\circ\psi^{\theta_{j+1} + (1-s_{j+1})\eps} (L_{j+1})) \]
for $j = 0, \hdots, d-1$. These maps satisfy the $A_\infty$ equation for a natural transformation as discussed in Section 10c of \cite{SeBook} and are the only possible non-zero terms in the natural transformation $\mathcal{T}_{\eps, \phi}$. Figure \ref{ntfig} shows the boundary degenerations contributing to the $A_\infty$ equation for $d = 1$. 

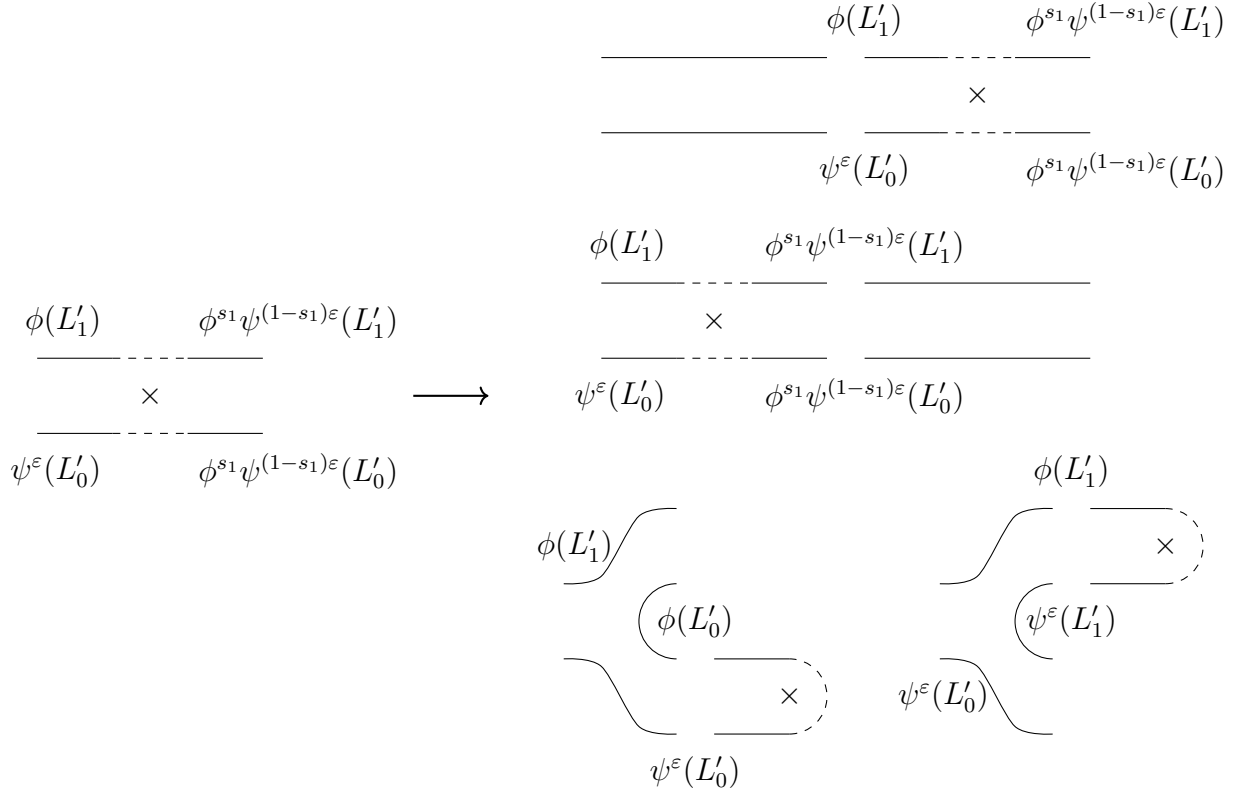
\begin{figure}
\centering 
\begin{tikzpicture}
	\draw (0,2) -- (1,2);
	\draw (0,3) -- (1,3);
	\node[left] at (1, 3.5) {$\phi (L_1') $};
	\node[left] at (1,1.5) {$\psi^{ \eps}(L_0')$};
	\draw[dashed] (1, 2) -- (2, 2);
	\draw[dashed] (1, 3) -- (2, 3) ;
	\draw (2,2) -- (3,2);
	\draw (2,3) -- (3,3);
	\node[right] at (2, 3.5){$\phi^{s_1}\psi^{ (1- s_1) \eps}(L_1') $};
	\node[right] at (2,1.5){$\phi^{s_1}\psi^{ (1- s_1) \eps}(L_0')$};
	\node at (1.5, 2.5) {$\times$}; 
	
	\draw[->, thick] (5 , 2.5) -- (6, 2.5);
	
	\draw (9,-2) -- (10,-2);
	\draw (9,-1) -- (10,-1);
	\node at (8.75, -0.5) {$\phi(L_0') $};
	\node at (8.75, -2.5) {$\psi^{\eps}(L_0')$};
	\node[left] at (7.8, 0.5) {$\phi(L_1')$};
	\draw[dashed] (10, -2) arc (-90:90: 0.5);
	\draw plot [smooth, tension = 0.5] coordinates {(7, 0) (7.5 , 0.1) (8, 0.9) (8.5 ,1)};
	\draw plot [smooth, tension = 0.5] coordinates {(7, -1) (7.5 , -1.1) (8, -1.9) (8.5 ,-2)};
	\draw (8.5 ,0) arc (90: 270 : 0.5);
	\node at (10, -1.5) {$\times$};
	
	\draw (14,0) -- (15,0);
	\draw (14,1) -- (15,1);
	\node at (13.75, -0.5) {$\psi^{\eps}(L_1') $};
	\node[left] at (12.8, -1.5) {$\psi^{\eps}(L_0')$};
	\node at (13.75, 1.5) {$\phi(L_1')$};
	\draw[dashed] (15, 0) arc (-90:90: 0.5);
	\draw plot [smooth, tension = 0.5] coordinates {(12, 0) (12.5 , 0.1) (13, 0.9) (13.5 ,1)};
	\draw plot [smooth, tension = 0.5] coordinates {(12, -1) (12.5 , -1.1) (13, -1.9) (13.5 ,-2)};
	\draw (13.5 ,0) arc (90: 270 : 0.5);
	\node at (15, 0.5) {$\times$};
	
	\draw (7.5,3) -- (8.5,3);
	\draw (7.5,4) -- (8.5,4);
	\node[left] at (8.5, 4.5) {$\phi (L_1') $};
	\node[left] at (8.5 ,2.5) {$\psi^{ \eps}(L_0')$};
	\draw[dashed] (8.5, 3) -- (9.5, 3);
	\draw[dashed] (8.5, 4) -- (9.5, 4) ;
	\draw (9.5,3) -- (10.5,3);
	\draw (9.5,4) -- (10.5,4);
	\node at (11, 4.5){$\phi^{s_1}\psi^{ (1- s_1) \eps}(L_1') $};
	\node at (11,2.5){$\phi^{s_1}\psi^{ (1- s_1) \eps}(L_0')$};
	\node at (9, 3.5) {$\times$}; 
	\draw (11, 3) -- (14, 3); 
	\draw (11, 4) -- (14, 4);
	
	\draw (11,6) -- (12,6);
	\draw (11,7) -- (12,7);
	\node[right] at (13, 7.5) {$\phi^{s_1}\psi^{ (1- s_1) \eps}(L_1') $};
	\node[right] at (13 ,5.5) {$\phi^{s_1}\psi^{ (1- s_1) \eps}(L_0')$};
	\draw[dashed] (12, 6) -- (13, 6);
	\draw[dashed] (12, 7) -- (13, 7) ;
	\draw (13,6) -- (14,6);
	\draw (13,7) -- (14,7);
	\node at (11, 7.5){$\phi(L_1')$};
	\node at (11,5.5){$\psi^\eps(L_0')$};
	\node at (12.5, 6.5) {$\times$}; 
	\draw (7.5, 6) -- (10.5, 6); 
	\draw (7.5, 7) -- (10.5, 7);
	
\end{tikzpicture}
\caption{Boundary degenerations in the $\overline{\mathcal{R}}_{2,1}$ contributing to the $A_\infty$ natural transformation equation. Dashed lines indicate where $df \neq 0$ and $\times$ is the interior marked point. For brevity of notation, we have also set $L_0' = \psi^{\theta_0} (L_0)$ and $L_1' = \psi^{\theta_1}(L_1)$. }
\label{ntfig}
\end{figure} 

As part of the universal property of localization, any natural transformation between functors from $\mathcal{F}_\Delta^\circ(\zeta)$ to $\mathcal{F}_\Delta^\circ(\zeta + \eps)$ taking quasi-units to quasi-units induces a natural transformation between the induced functors from $\mathcal{F}_\Delta(W_\Sigma, \zeta)$ to $\mathcal{F}_\Delta(W_\Sigma, \zeta + \eps)$. An explicit formula, which simplifies somewhat in our case as our functors have no higher terms, is given in Proposition 4.1 of \cite{LO}. As a result of that fact, we obtain the following proposition from the $\mathcal{T}_{\eps, \phi}$.

\begin{prop} \label{nattrans} Given a twisting Hamiltonian $H_D$ with flow $\phi_D$ for an effective toric divisor $D$ on $X$, there is a natural transformation
\[ T_{\eps, \phi_D} \colon P_\eps \to F_{\phi_D} = F_D \]
of functors from $\mathcal{F}_\Delta (W_\Sigma, \zeta)$ to $\mathcal{F}_\Delta(W_\Sigma, \zeta + \eps)$ for any $\zeta \in (0, \pi]$ and $\eps \in \bb{Q} \cap (0, \pi - \zeta]$.
\end{prop}

When $\phi_D = \id$, each zeroth order term $\mathcal{T}^0_{\eps,\id} (L, \theta) \in CF^0 (\psi^{\theta + \eps}(L), \psi^\theta(L))$ for an object $(L, \theta)$ of $\mathcal{F}^\circ_\Delta (\zeta)$ is exactly a quasi-unit $c_{(L, \theta \to \theta + \eps)}$. This, in addition, implies that all the $T^0_{\eps, \id}$ are quasi-units as Proposition 4.1 of \cite{LO} shows that each $T^0_{\eps, \id}$ is the image of the corresponding $\mathcal{T}^0_{\eps, \id}$ under the localization functor. As a result of that observation and Lemma 1.6 in \cite{SeBook}, we have the following.

\begin{prop} \label{transexist} For any $\zeta \in (0, \pi]$ and $\eps \in (0, \pi-\zeta]$, the functor $P_\eps$ is quasi-isomorphic to the inclusion $\mathcal{F}_\Delta(W_\Sigma, \zeta) \to \mathcal{F}_\Delta(W_\Sigma, \zeta + \eps)$. 
\end{prop}

Thus, we see that the natural transformations $T_{\eps, \phi_D}$ define natural transformations from the identity to $F_D$ as functors on $\mathcal{F}_\Delta(W_\Sigma)$, as the inclusion $\mathcal{F}_\Delta(W_\Sigma, \zeta) \to \mathcal{F}_\Delta (W_\Sigma, \zeta + \eps)$ is a quasi-isomorphism for all $\zeta \in (0, \pi]$ and $\eps \in (0, \pi - \zeta]$. To conclude this section, we wish to interpret the natural transformations $T_{\eps, \phi_D}$ under homological mirror symmetry.

\begin{thm} \label{sectionthm} Suppose that $\Delta$ is adapted to $\Sigma$, $D$ is an effective toric divisor, and $F_D$ is constructed as a functor on $\mathcal{F}^s_\Delta(W_\Sigma, \zeta)$ as in Example \ref{tensoronlines}.  The image of the natural transformation $T_{\eps, \phi_D}$ under the quasi-equivalence of Corollary \ref{hmscor} 
is quasi-isomorphic to $(\cdot) \otimes s_D$ where $s_D$ is a toric defining section of $D$. 
\end{thm}
\begin{proof} Applying the quasi-equivalence of Corollary \ref{hmscor} gives us a natural transformation from the identity on $\Cech(X)$ to $(\cdot) \otimes \mathcal{O}(D)$ by Theorem \ref{mirrorfunctors}. Such natural transformations are classified by 
\[ HH^0(\Cech(X), \mathcal{B}) \cong \Hom_{X \times X}(i_* \mathcal{O}, i_* \mathcal{O}(D) ) \cong \Gamma(\mathcal{O}(D)) \]
where $\mathcal{B}$ is the bimodule over $\Cech(X)$ defined using the identity functor on the left and $(\cdot) \otimes \mathcal{O}(D)$ on the right and $i$ is the diagonal map. Therefore, we only need to look at the zeroth order terms $T^0_{\eps, \phi_D}$. 

In fact, it is enough to only look at the cohomology class of
\[ T^0_{\eps, \phi_D}(\mathcal{L}) \equiv T^0_{\eps, \phi_D} (\mathcal{L}, 0) \]
in $HF^0( \psi^{\eps}(\mathcal{L}), \phi_D(\mathcal{L}) ) \cong \Gamma(\mathcal{O}(D))$. With respect to the $\bb{Z}^n$ grading on $CF^0( \psi^{\eps}(\mathcal{L}), \phi_D(\mathcal{L}) )$ induced by the lifts $\eps dK_\Delta$ and $dH_D$ of $\psi^{\eps}(\mathcal{L})$ and $\phi_D(\mathcal{L})$, respectively, $T^0_{\eps, \phi_D}(\mathcal{L})$ has degree zero. This grading is sent to the natural $\bb{Z}^n$ grading on $\Gamma(\mathcal{O}(D))$ under the isomorphism $HF^0( \psi^{\eps}(\mathcal{L}), \phi_D(\mathcal{L}) ) \cong \Gamma(\mathcal{O}(D))$ induced by Corollary \ref{hmscor} (see the proof of Proposition 6.7 in \cite{Ab2}), and the cohomology in degree $0 \in \bb{Z}^n$ has rank one. As a result, the cohomology class of $T^0_{\eps, \phi_D}(\mathcal{L})$ is sent to a multiple of a toric defining section. 

It remains only to show that this multiple is not zero. For that, we can apply an argument similar to the proof of Proposition \ref{invariantqu}(c). Namely, we consider instead $T^0_{\eps, \phi_D}(\mathcal{L}, \theta)$ for some $\theta > 0$, which has cohomologous image to that of $T^0_{\eps,\phi_D}(\mathcal{L})$, and view
\[ m^2(\cdot, T^0_{\eps,\phi_D}(\mathcal{L}, \theta) ) : CF^\bullet(\phi_D\circ \psi^\theta(\mathcal{L}), L) \to CF^\bullet(\psi^{\theta+ \eps}(\mathcal{L}), L) \]
as a continuation map for $H = \eps K_\Delta - H_D$ by gluing and reparameterization where $L$ is any monomially admissible Lagrangian section. We can factor this continuation map as a continuation map for a twisting Hamiltonian $\wt{H}$ for $-D$ that is everywhere nonnegative and zero in an arbitrarily large compact set containing all intersection points of $\phi_D \circ \psi^\theta(\mathcal{L})$ and $L$, a continuation map coming from multiplication by a quasi-unit, and a continuation map of an admissible Hamiltonian. 

The latter two continuation maps are quasi-isomomorphisms by the proof of Proposition \ref{invarianthf} and by Proposition \ref{invariantqu}(c), respectively. Therefore, we only need to show that the first continuation map is nonzero on cohomology. This continuation map satisfies $0 \leq E^{\text{geom}}(u) \leq E^{\text{top}}(u)$ for any solution $u$ of \eqref{perturb} counted to define the continuation map. As a result, the only solutions that preserve action are constant. The result follows if we can show that there is a monomially admissible Lagrangian section $L$ such that there is a nonzero class in $HF^\bullet (\wt{\phi}\circ \phi_D \circ \psi^\theta( \mathcal{L}), L)$, where $\wt{\phi}$ is the flow of $\wt{H}$, represented by elements of $CF^\bullet( \phi_D \circ \psi^\theta(\mathcal{L}), L)$. Consider $L = \phi_D(\mathcal{L})$. Then, the generators of $CF^\bullet(\wt{\phi}\circ \phi_D \circ \psi^\theta( \mathcal{L}), L)$ which have degree $0$ with respect to the $\bb{Z}^n$ grading correspond to points where 
\begin{equation} \label{intersectionproof} d\wt{H} + \theta dK_\Delta = 0. \end{equation}
We can assume that $\nabla K_\Delta \cdot \alpha = 1$ in $C_\alpha$ outside of the compact set where $\wt{H}$ vanishes for all $\alpha \in A$ and that $\nabla \wt{H} \cdot \alpha \geq 0$ in $C_\alpha$ for all $\alpha \in A$. Then, the only solutions to \eqref{intersectionproof} are where $\wt{H}$ vanishes and $d K_\Delta = 0$. That is, they all correspond to generators of $CF^\bullet( \phi_D \circ \psi^\theta(\mathcal{L}), L)$. Since $HF^\bullet (\wt{\phi}\circ \phi_D \circ \psi^\theta( \mathcal{L}), L)$ has rank one in the degree $0$ part of the $\bb{Z}^n$ grading, we have found the desired nonzero cohomology class. 
\end{proof}

In the case that $D$ is ample and we choose $H_D$ as in Section \ref{floercomp}, we have seen that $T^0_{\eps, \phi_D} \in CF^\bullet(\psi^\eps(\mathcal{L}), \phi_D(\mathcal{L}))$ is a non-zero multiple of the generator corresponding to $0 \in P$. In fact, the proof can be somewhat simplified in that case using the concavity of $H_D$.

\subsection{Partially wrapped mirrors to some non-complete toric varieties via localization} \label{openhms}

In \cite{SeNT}, Seidel observed that for a (possibly singular) hypersurface $D \subset X$ one can obtain $D^b \Coh (X \setminus D)$ by localizing $D^b \Coh(X)$ at the natural transformation $(\cdot) \otimes s_D$. Futher, he postulated that there should be a mirror natural transformation so that homological mirror symmetry for $X \setminus D$ can be deduced from homological mirror symmetry for $X$ with the understanding of the mirror natural transformation. Although the picture should hold in great generality, we will show that it holds in the toric setting using the natural transformations that we have already constructed in Section \ref{defsect}. The analogous statement using microlocal sheaf theory was proved in \cite{IK}. 

Before stating the result, we need a bit more notation. For an effective toric divisor $D$, we set $S_D$ to be a set consisting of a chain-level representative for the morphism $1 \otimes s_D \in H^0 (V, V \otimes \mathcal{O}(D))$ for every object $V$ of $\Cech(X)$ where $s_D$ is a defining section of $D$. For instance, we can take $S_D$ to simply be the image under the quasi-equivalence of Corollary \ref{hmscor} of all the $T^0_{\eps, \phi_D}$ by Theorem \ref{sectionthm}. We define $\Cech(X \setminus D)$ to be the localization $S_D^{-1} \Cech(X)$. As a consequence of (1.10) and (1.11) in \cite{SeNT} and Theorem \ref{sectionthm}, $\Cech(X \setminus D)$ is a dg-enhancement of the category of line bundles on $X \setminus D$. Setting $\mathcal{F}^s_\Delta(W_\Sigma, D)$ to be the localization of $\mathcal{F}^s_\Delta(W_\Sigma)$ at all the $T^0_{\eps, \phi_D}$, we have deduced the following.

\begin{thm} \label{removedhms} Suppose that $\Delta$ is adapted to $\Sigma$ and $D$ is an effective divisor. The commutative diagram 
\[ \begin{tikzcd} \mathcal{F}^s_\Delta(W_\Sigma) \arrow{r}{\sim} \arrow{d} & \Cech(X) \arrow{d}  \\
\mathcal{F}^s_\Delta(W_\Sigma, D) \arrow{r}{\sim} & \Cech(X \setminus D) 
\end{tikzcd} \]
of $A_\infty$-categories with the top arrow being the quasi-equivalence from Corollary \ref{hmscor} induces a commutative diagram
\[ \begin{tikzcd} H^0(\text{Tw}^\pi \mathcal{F}^s_\Delta(W_\Sigma)) \arrow{r}{\sim} \arrow{d} & D^b\Coh(X) \arrow{d}  \\
H^0(\text{Tw}^\pi \mathcal{F}^s_\Delta(W_\Sigma, D)) \arrow{r}{\sim} & D^b \Coh(X \setminus D) 
\end{tikzcd} \]
of triangulated categories. 
\end{thm}

\begin{rem} Following the interpretation of wrapping and partial wrapping in Floer theory from \cite{AbSloc}, the localization $\mathcal{F}^s_\Delta(W_\Sigma) \to \mathcal{F}^s_\Delta(W_\Sigma, D)$ should correspond geometrically to wrapping by a twisting Hamiltonian for $D$. Indeed, the case of $D = \sum_{\alpha \in A} D_\alpha$ has $X \setminus D = (\bb{C}^*)^n$ and $H_D$ will give a cofinal wrapping sequence in all directions. In general, the wrapping only occurs in the directions where the twisting Hamiltonian is negative. See Figure \ref{wrappingfig} for a simple example. However, a geometric description of our localization cannot be deduced directly from existing results due to technical differences in the setup of the monomially admissible Fukaya-Seidel category.
\end{rem}

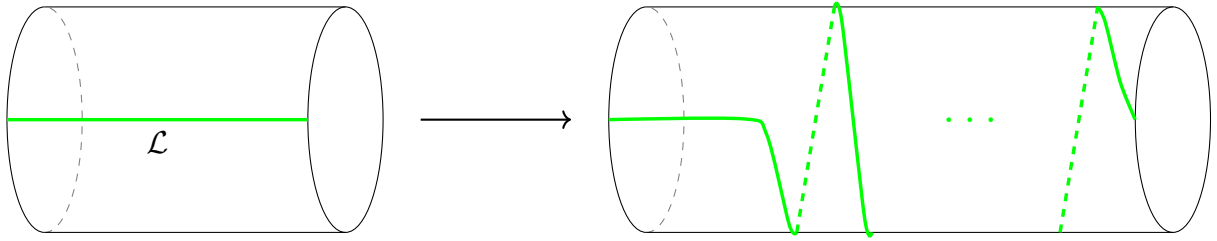
\begin{figure} 
\centering 
\begin{tikzpicture}
	\draw[dashed,color=gray] (-6,0) arc (-90:90:0.5 and 1.5);
	\draw (-6,0) -- (-2,0);
	\draw (-6,3) -- (-2,3);
	\draw (-6,0) arc (270:90:0.5 and 1.5);
	\draw (-2,1.5) ellipse (0.5 and 1.5);
	\draw (-6.5,1.5) -- (-2.5, 1.5) [color = green, line width = 1.3] node[font=\normalsize,midway,below, color= black] {$\mathcal{L}$};

	\draw[->, thick] (-1, 1.5) -- (1,1.5);
	 
	\draw[dashed,color=gray] (2,0) arc (-90:90:0.5 and 1.5);
	\draw (2,0) -- (9,0);
	\draw (2,3) -- (9,3);
	\draw (2,0) arc (270:90:0.5 and 1.5);
	\draw (9,1.5) ellipse (0.5 and 1.5);
	\draw [green, line width = 1.3] plot [smooth, tension = 0.5] coordinates {(1.5, 1.5) (3.3, 1.51) (3.6,1.3)  (3.9, 0.1) (4,0)  };
	\draw[dashed, green, line width = 1.3] plot [smooth, tension = 0.5] coordinates{(4, 0) (4.5,3)};
	\draw[green, line width = 1.3] plot [smooth, tension = 0.5] coordinates{(4.5, 3) (4.6, 2.8) (4.9, 0.2) (5, 0) } ;
	\draw[green] (6.3, 1.5) node[scale= 1.5] {$\hdots$};
	\draw[dashed, green, line width = 1.3] plot [smooth, tension = 0.5] coordinates{(7.5, 0) (8,3)};
	\draw[green, line width = 1.3] plot [smooth, tension = 0.5] coordinates{(8, 3) (8.1, 2.8) (8.3, 2) (8.5, 1.5) }; 	
\end{tikzpicture}
\caption{The image of the zero section in the mirror to $\mathbb{P}^1$ as it is wrapped by the twisting Hamiltonian for $D_{1}$ giving a Lagrangian mirror to the structure sheaf of $\bb{C}$ in the limit. All wrapping occurs in the region $\log |z| \geq 0$.} 
\label{wrappingfig}
\end{figure}

Theorem \ref{removedhms} should be viewed as an instance of Sylvan's stop removal \cite{Syl}. However, his results do not apply directly to our setting. It should also be noted that Katzarkov-Kerr have constructed partially wrapped mirrors to toric varieties in general in \cite{KK} using an entirely different approach. 

\subsection{Line bundles are thimbles} \label{thimblesec}

The monodromy action of \eqref{monodromyeq} can also be seen without the language of monomial admissibility. For instance, if we assume that $W^{\theta, D}_\Sigma$ is a Lefschetz fibration for all $\theta \in [0, 2\pi]$, then we can follow the critical points, Lefschetz thimbles, and their instersections as $\theta$ goes from $0$ to $2\pi$ giving a monodromy functor on the Fukaya-Seidel category generated by thimbles (cf. Figure \ref{pnex}). Note that although the assumption that $W^{\theta, D}_\Sigma$ is a Lefschetz fibration for all $\theta \in [0, 2\pi]$ may not hold in general, it will hold if we allow the norms of the coefficients to be perturbed by small functions of $\theta$ given that $W_\Sigma$ is a Lefschetz fibration as the degenerate Laurent polynomials have real codimension two. In fact, it is shown in \cite{DKK1} and discussed in the toric setting in \cite{BDFKK} that certain choices of coefficients make each $W_\Sigma^{\theta,D}$ a Lefschetz fibration with critical values lying on concentric circles referred to as a radar screen. In light of the monodromy action on Lefschetz thimbles and our understanding of the monodromy in Theorem \ref{mirrorfunctors}, we are lead to expect the following.

\begin{conj} \label{thimbleconj} If $W_\Sigma$ is a Lefschetz fibration, there is a Lagrangian thimble mirror to each line bundle on $X$. 
\end{conj}

Although there is always a full strong exceptional collection of thimbles, Conjecture \ref{thimbleconj} would not guarantee that $D^b\Coh(X)$ admits a full strong exceptional collection of line bundles. In fact, it is known by work of Efimov \cite{Ef2} that even Fano toric varieties do not admit such an exceptional collection in general. A necessary condition for Conjecture \ref{thimbleconj} to produce an exceptional collection of line bundles is that the monodromy action must be transitive on critical points of $W_\Sigma$. However, a transitive action on critical points is not sufficient as we would also need to guarantee that the monodromy can produce thimbles for all critical points that correspond to disjoint paths. Jerby has used a related method to produce full strong exceptional collections of line bundles in $D^b \Coh (X)$ for some low-dimensional examples in \cite{Jerb1, Jerb2, Jerb3} (cf. Section \ref{introexample}).

We now outline how one might attempt to prove Conjecture \ref{thimbleconj} and prove a particular case. First, we assume that $c_\alpha \in \bb{R}_{>0}$ for all $\alpha \in A$. As a result, the Lagrangian $\mathcal{L} = (\bb{R}_{>0})^n$ is both a monomially admissible Lagrangian section and the thimble for the critical point of $W_\Sigma$ corresponding to its minimum value when restricted to $\mathcal{L}$. Suppose there is an admissibility condition preserved by $\phi_D$ and such that $\phi_D^{-\theta/2\pi}(\mathcal{L}^\theta)$ is admissible for all $\theta \in [0, 2\pi]$ where $\mathcal{L}^\theta$ is a thimble for $W_\Sigma^{\theta, D}$ obtained from following $\mathcal{L}$. Then, we have an isotopy, which must be Hamiltonian as our Lagrangians are simply connected, of admissible Lagrangians from $\mathcal{L}$ to $\phi_D^{-1}(\mathcal{L}^{2\pi})$. It follows that $\phi_D(\mathcal{L})$ and the thimble $\mathcal{L}^{2\pi}$ are quasi-isomorphic in a Fukaya-Seidel category of Lagrangians subject to this hypothetical admissibility condition. Therefore, Conjecture \ref{thimbleconj} follows from Theorem \ref{mirrorfunctors} given the existence of such an admissibility condition. 

When we have an adapted division with $k_\alpha = 1$ in the second condition of Definition \ref{div}, e.g., when $X$ is Fano by Corollary \ref{tropicalfanoadapts}, admissibility with respect to Abouzaid's tropical localization of $W_\Sigma^{\theta,D}$ as defined in \cite{Ab1} provides a candidate admissibility condition. In that case, the flow of $H_D$ commutes with the tropically localized superpotential $W_{TL}^{\theta, D}$ in the sense that 
\[ W_{TL}^{\theta, D} \circ \phi^{\theta'/2\pi} = W_{TL}^{\theta + \theta', D} \]
when the compact set where the derivatives of $H_D$ are not controlled is small enough. However, $W_{TL}^{\theta, D}$ is not a Lefschetz fibration. Thus, one needs to be able to map the Lagrangians $\mathcal{L}^\theta$ to Lagrangians admissible with respect to $W_{TL}^{\theta, D}$ in a way that preserves Floer cohomology to carry out the proof. Although it is shown in \cite{Ab1} that the pairs $\left( (\bb{C}^*)^n, (W_{\Sigma}^{\theta,D})^{-1}(1)\right)$ and $\left((\bb{C}^*)^n, (W_{TL}^{\theta,D})^{-1}(1)\right)$ are symplectomorphic, this is not enough to construct a quasi-equivalence of the Fukaya-Seidel categories with boundaries in these hypersurfaces (defined in \cite{Ab2}). It is also not clear how to define the tropical localization when working with a nonstandard toric K\"{a}hler form as needed in the proof of Corollary \ref{tropicalfanoadapts}.

If we make further assumptions, it is possible to get around such difficulties by modifying the tropical localization construction in Section 4.1 of \cite{Ab1}. The main difference from Abouzaid's construction will be that our set $A$ does not contain the origin. We assume that $X$ is Fano. We look at a family of Lefschetz fibrations
\[ W_t^{\theta} = \frac{1}{t} \sum_{\alpha \in A} c_\alpha e^{in_\alpha \theta} z^\alpha \]
with $c_\alpha = c_\alpha(\theta) \in \bb{R}_{>0}$ and $ \frac{c_{\alpha}(\theta)}{c_\beta(\theta)} <  e^{\eps_0}$ for all $\alpha, \beta \in A$ and $\theta \in [0, 2\pi]$ and for some $\eps_0> 0$. We assume that the toric K\"{a}hler form is the standard one.  For all $\alpha, \beta \in A$, let 
\[ V_\alpha = \{ \Log(z) \, | \, | z^\alpha| \geq | z^\gamma| \text{ for all } \gamma \in A \} \]
and $\mathcal{H}(\alpha, \beta)$ the hyperplane given by 
\[ \alpha \cdot u  = \beta \cdot u .\]
Before proceeding further, we will need the following lemma, which is almost exactly Lemma 4.1 in \cite{Ab1} and follows from the same argument. 
\begin{lem} \label{Ab141} There is a constant $c > 0$ such that
\[ d(p, V_\alpha) \geq \eps \implies d(p, \mathcal{H}(\alpha, \beta)) \geq 2c\eps \]
for all sufficiently small $\eps$, every $p \in V_\beta$, and all pairs $\alpha \neq \beta \in A$ and where $d$ is the Euclidean metric.
\end{lem}

We now fix such a sufficiently small $\eps$ and choose smooth functions $\psi_\alpha \colon \bb{R}^n \to [0,1]$ for all $\alpha \in A$ such that 
\begin{equation} \label{cutoff1} d(p, V_\alpha) \leq \frac{\eps \log(t)}{2} \iff \psi_\alpha(p) = 0, \end{equation}
\begin{equation} \label{cutoff2} d(p, V_\alpha) \geq \eps \log(t) \iff \psi_\alpha (p) = 1, \end{equation}
and
\begin{equation} \label{cutoff3} \sum_{i=1}^n \left| \frac{\partial\psi_{\alpha}(p)}{\partial u_i} \right| < \frac{4}{\eps \log(t)} \end{equation}
for all $p \in \bb{R}^n$ as in equations (4-3) to (4-5) in \cite{Ab1}. As in \cite{Ab1}, we will write $\psi_\alpha(z)$ for $\psi_\alpha(\Log(z))$. With all that in hand, we set
\[ W_{t,1}^\theta = \frac{1}{t} \sum_{\alpha \in A} c_\alpha e^{i n_\alpha \theta} (1- \psi_\alpha(z))z^\alpha .\]
Our goal will be to prove that this is a family of symplectic Lefschetz fibrations. We will first need a few calculations. The following lemma is essentially the same as Lemma 4.5 of \cite{Ab1} and the same proof applies and uses Lemma \ref{Ab141}.
\begin{lem} \label{Ab145} If $\Log(z) \in V_\beta$ and $\psi_\alpha(z) \neq 0$, then 
\[ c_\alpha |z^\alpha | < e^{\eps_0 - c\eps \log(t) | \alpha - \beta|} c_\beta | z^\beta| .\]
\end{lem}

The next lemma is required due to the lack of a constant term in our setup. 

\begin{lem} \label{notinab} There is a constant $\kappa$ such that $\eps \geq \kappa > 0$ and if $p = \Log(z) \not \in B_{\eps /2}(0)$, then there is a cone $\sigma \in \Sigma$ such that for all $\alpha \in A$, either $\alpha \in \sigma$ or $d(p, V_\alpha) > \kappa/2$.

Moreover, if $p \not \in B_{\eps \log(t)/2}, p \in V_\beta,$ and $\alpha \not\in \sigma$, then 
\[ c_\alpha |z^\alpha| < e^{\eps_0 - c\kappa \log(t)|\alpha - \beta|} c_\beta |z^\beta| .\]
\end{lem}
\begin{proof} Outside of the origin, $V_{\alpha_1} \cap \hdots \cap V_{\alpha_k} \neq \emptyset$ if and only if $\langle \alpha_1, \hdots, \alpha_k \rangle$ is a cone of $\Sigma$ because we have assumed that $X$ is Fano and the equations $u\cdot \alpha \leq 1$ for all $\alpha \in A$ define the polytope of the anticanonical bundle.  For any $q \in \bb{R}^n$ and $r\geq 0$, let $U_r(q)$ be the set of $\alpha \in A$ such that $V_\alpha \cap B_{r/2}(p) \neq \emptyset$. Take $\kappa$ to be the smaller of $\eps$ and
\[ \min_{\{\alpha_1, \hdots, \alpha_k\} \in Z} \ \inf_{q \in \bb{R}^n \setminus B_{\eps/2}(0)} \left\{ r \, | \, \{ \alpha_1, \hdots, \alpha_k \} \subset U_r(q) \right\} \] 
where $Z$ consists of all sets of elements of $A$ which do not generate a cone of $\Sigma$. Each infimum is nonzero as the cones $V_{\alpha_1}, \hdots, V_{\alpha_k}$ have empty intersection away from the origin, and so do their $r/2$-neighborhoods for $r$ sufficiently small compared to $\eps$. This choice of $\kappa$ ensures that the set $U_\kappa(p)$ of $\alpha \in A$ for which $d(p,V_\alpha) \leq \kappa/2$ is not in $Z$, i.e., generates a cone of $\Sigma$.

The last part of the lemma follows from the same computation that proves Lemma \ref{Ab145}.
\end{proof}

Let $N$ be the maximum $\ell^1$-norm of a vector in $A$, $|A|$ be the number of elements in $A$, and $\rho > 1$ be the maximal length distortion of the injective linear maps taking $\{ \alpha_1, \hdots, \alpha_k \}$ to standard basis vectors in $\bb{R}^n$ among cones $\langle \alpha_1, \hdots, \alpha_k \rangle$ in $\Sigma$. We assume that $t$ is large enough so that
\begin{equation} \label{constants1} e^{\eps_0 - c\kappa\log(t)} < \frac{1}{2|A|N \rho} \end{equation}
and
\begin{equation} \label{constants2} \frac{e^{\eps_0 - c\eps \log(t)}}{\eps \log(t)} < \frac{1}{16|A| \rho} .\end{equation} 

We are now prepared to put these estimates to use in a similar fashion to the proof of Proposition 4.2 in \cite{Ab1}.

\begin{prop} \label{lefschetz} For $t$ sufficiently large, $W_{t,1}^\theta$ is a symplectic Lefschetz fibration for all $\theta \in [0, 2\pi]$ given that the same is true for $W_{t}^\theta$. 
\end{prop}
\begin{proof} Note that if $\Log(z) \in B_{\eps \log(t)/2}(0)$, we have that $W_{t,1}^\theta(z) = W_t^\theta(z)$ is a symplectic Lefschetz fibration by \eqref{cutoff1}. Outside of $\Log^{-1}(B_{\eps \log(t)/2})$, we will deduce the result by showing that 
\[ \left| \partial W_{t,1}^\theta \right| > \left| \overline{\partial} W_{t,1}^\theta \right| \]
which implies that $W_{t,1}^\theta$ is a symplectic fibration by an observation of Donaldson in \cite{Don} and that it has no critical points outside of $B_{\eps \log(t)/2}(0)$. We have 
\[ \partial W_{t,1}^\theta = \sum_{\alpha \in A} c_\alpha e^{i n_\alpha \theta} (1-\psi_\alpha(z)) \partial z^\alpha - \sum_{\alpha \in A} c_\alpha e^{i n_\alpha \theta} z^\alpha \partial \psi_\alpha(z)  \]
and
\[ \overline{\partial} W_{t,1}^\theta = - \sum_{\alpha \in A} c_\alpha e^{i n_\alpha \theta} z^\alpha \overline{\partial} \psi_\alpha(z) .\]
Since $\psi_\alpha(z)$ is a function only of $|z|$, it is enough to show that
\begin{equation} \label{Ab14-6}  \left| \sum_{\alpha \in A} c_\alpha e^{i n_\alpha \theta} (1-\psi_\alpha(z)) \partial z^\alpha \right| > 2 \left| \sum_{\alpha \in A} c_\alpha e^{i n_\alpha \theta} z^\alpha \partial \psi_\alpha(z) \right|.\end{equation}
Suppose that $\Log(z) \in V_\beta$ and $\sigma$ is as in Lemma \ref{notinab}. We write
\[ \sum_{\alpha \in A} c_\alpha e^{i n_\alpha \theta} (1-\psi_\alpha(z)) \partial z^\alpha = \sum_{\gamma \in \sigma} c_\gamma e^{i n_\gamma \theta} (1-\psi_\gamma(z)) \partial z^\gamma +  \sum_{\alpha \in A \setminus \sigma} c_\alpha e^{i n_\alpha \theta} (1-\psi_\alpha(z)) \partial z^\alpha. \]
Then, \eqref{Ab14-6} follows from
\begin{equation} \label{thimbleestimate}  \left| \sum_{\gamma \in \sigma} c_\gamma e^{i n_\gamma \theta} (1-\psi_\gamma(z)) \partial z^\gamma \right| > \left| \sum_{\alpha \in A\setminus \sigma} c_\alpha e^{i n_\alpha \theta} (1-\psi_\alpha(z)) \partial z^\alpha \right| + 2 \left| \sum_{\alpha \in A} c_\alpha e^{i n_\alpha \theta} z^\alpha \partial \psi_\alpha(z) \right|. \end{equation}
We first bound the left-hand side of \eqref{thimbleestimate}. We have
\begin{align*}  \left| \sum_{\gamma \in \sigma} c_\gamma e^{i n_\gamma \theta} (1-\psi_\gamma(z)) \partial z^\gamma \right|^2 &=  \left| \sum_{\gamma \in \sigma} c_\gamma e^{i n_\gamma \theta} (1-\psi_\gamma(z)) \sum_{i=1}^n \gamma_i z^{\gamma - e_i} \, dz_i  \right|^2 \\
&= \sum_{i=1}^n  \left| \sum_{\gamma \in \sigma} c_\gamma e^{i n_\gamma \theta} (1-\psi_\gamma(z)) \gamma_i z^\gamma \right|^2 \\
&\geq \frac{c_\beta^2 \left| z^\beta \right|^2}{\rho^2}
\end{align*}
using that $\frac{dz_i}{z_i}$ is an orthonormal basis and the same logic as in the proof of Proposition 4.2 in \cite{Ab1} for the last inequality. We now move to the right-hand side of \eqref{thimbleestimate}. For the first term, we have
\begin{align*} \left| \sum_{\alpha \in A\setminus \sigma} c_\alpha e^{i n_\alpha \theta} (1-\psi_\alpha(z)) \partial z^\alpha \right|  & \leq \sum_{\alpha \in A \setminus \sigma} c_\alpha | \partial z^\alpha | \\
& \leq \sum_{\alpha \in A \setminus \sigma} c_\alpha \sum_{i=1}^n |\alpha_i| |z^\alpha| \\
& \leq N |A| e^{\eps_0 -c\kappa \log(t)} c_\beta \left| z^\beta \right| \\
&< \frac{c_\beta \left| z^\beta \right|}{2\rho} 
\end{align*}
using Lemma \ref{notinab} and \eqref{constants1}. For the second term, we have
\begin{align*} 
\left| \sum_{\alpha \in A} c_\alpha e^{i n_\alpha \theta} z^\alpha \partial \psi_\alpha(z) \right| &\leq \sum_{\psi_\alpha(z) \neq 0} c_\alpha \left| z^\alpha \right| \left|\partial\psi_\alpha(z) \right| \\
&\leq \sum_{\psi_\alpha(z) \neq 0} c_\alpha \left| z^\alpha \right| \sum_{i=1}^n \left| \frac{\partial \psi_\alpha}{\partial u_i} \right| \\ 
& < \frac{4|A| e^{\eps_0-c\eps \log(t)}}{\eps \log(t)} c_\beta \left| z^\beta \right| \\
& < \frac{c_\beta \left| z^\beta \right|}{4 \rho}
\end{align*}
using Lemma \ref{Ab145}, \eqref{cutoff1}, \eqref{cutoff3}, and \eqref{constants2}. Combining these estimates gives \eqref{thimbleestimate} as desired.
\end{proof}

As a result, we get the following special case of Conjecture \ref{thimbleconj}. 

\begin{prop} \label{thimbleexample} Suppose that $X$ is Fano, $c_\alpha \in \bb{R}_{>0}$ make $W_\Sigma$ a Lefschetz fibration, and $\frac{c_\alpha}{c_\beta} < e^{\eps_0}$ for all $\alpha, \beta \in A$ and $\eps_0$ sufficiently small. If there is a monomial division $\Delta$ for $W_\Sigma$ and $\mu = \Log$ that is adapted to the fan and has all $k_\alpha = 1$, then for every line bundle $\mathcal{O}(D)$ on $X$ there is a Lagrangian thimble of $W_\Sigma$ quasi-isomorphic to a monomially admissible Lagrangian section mapped to $\mathcal{O}(D)$ by the quasi-equivalence of Corollary \ref{hmscor}.
\end{prop}
\begin{proof} Assume that $t$ is large enough so that Proposition \ref{lefschetz} applies to $W_{t,1}^\theta$ (as defined above for any toric divisor $D$ on $X$). Then, note that Lagrangian thimbles for $W_\Sigma^{D, \theta}$ can be identified with Lagrangian thimbles for $W_{t,1}^{\theta}$ as $W_{t,1}^{\theta} = \frac{1}{t} W_\Sigma^{D, \theta}$ on the open set $B_{\eps \log(t)/2}(0)$. Let $\mathcal{L}^\theta$ be a family of Lagrangian thimbles for $W_\Sigma^{D,\theta}$ such that $\mathcal{L}^\theta$ fibers over $\bb{R}_{>0}$ away from a compact subset and is obtained by following $\mathcal{L}$ in the family of Lefschetz fibrations, and let $\mathcal{L}^{\theta,1}$ be the corresponding family of Lagrangian thimbles for $W_{t,1}^\theta$.

Further, we can assume $\Delta$ is a combinatorial division. As a result, $\Delta$ will be a monomial division for $W_{\Sigma}^{D,\theta}$ for all $\theta \in [0, 2\pi]$ given that $\eps_0$ is sufficiently small. It follows that outside of a compact subset, $\phi_D$ satisfies
\[ W_{t,1}^{\theta} \circ \phi_D^{\theta'/2\pi} = W_{t,1}^{\theta + \theta'} \]
when $\eps_0$ is sufficiently small. 

Therefore, $\psi^{-\theta/2\pi}(\mathcal{L}^{\theta,1})$ is a family of simply-connected Lagrangians that are admissible in the sense of Definition 2.4 of \cite{Ab1} (given that the reference fiber is taken sufficiently far out along the $\bb{R}_{>0}$ axis). Thus, all Lagrangians in the family are Hamiltonian isotopic among admissible Lagrangians. In particular, $\phi_D^{-1}(\mathcal{L}^{2\pi,1})$ is Hamiltonian isotopic to $\mathcal{L}$. It follows that $\phi_D(\mathcal{L})$ and $\mathcal{L}^{2\pi,1}$ are quasi-isomorphic.
\end{proof}

It is likely that the arguments used to prove Proposition \ref{thimbleexample} can be applied in a more general setting. However, a significantly modified argument is needed in general, particularly to go beyond the Fano case.  

\appendix

\section{$A_\infty$-pre-categories} \label{precat}

In this appendix, we recall some relevant definitions for $A_\infty$-pre-categories and show that $A_\infty$-pre-categories and localized $A_\infty$-categories give equivalent definitions of Fukaya-Seidel categories. Our terminology and definitions are closer to those in \cite{KoS} than in \cite{Ab2}. 

\begin{df} A non-unital $A_\infty$-pre-category $\mathcal{C}$ is given by
\begin{itemize} 
\item A class of objects $\text{Ob}(\mathcal{C})$;
\item A subclass $\text{Ob}_n^{tr}(\mathcal{C}) \subset \text{Ob}(\mathcal{C})^n$ for each $n \in \bb{N}$ called the transverse sequences of objects of length $n$;\footnote{The word transverse here has a purely formal meaning. Although it is related to transversality of the Lagrangian objects in a Fukaya category, the notion of transverse objects does not always exactly correspond with geometric transversality.}
\item A $\bb{Z}$-graded vector space $\Hom(Z_0,Z_1)$ for every $(Z_0,Z_1) \in \text{Ob}_2^{tr}(\mathcal{C})$; 
\item For every $d \geq 1$ and every transverse sequence $(Z_0, \hdots, Z_d)$ of length $d + 1$, a degree $2-d$ map
\[ m^d\colon \Hom(Z_{d-1}, Z_d) \otimes \hdots \otimes \Hom(Z_0, Z_1) \to \Hom(Z_0, Z_d); \]
\end{itemize}
such that $\text{Ob}_1^{tr} = \text{Ob}(\mathcal{C})$, every subsequence of a transverse sequence is transverse, and the maps $m^d$ satisfy the $A_\infty$-relations. 
\end{df}

In order to get a reasonably behaved theory, it is necessary to make some further requirements on these categories. 

\begin{df} \label{preqiso} A degree zero morphism $p \in \Hom(Z_0, Z_1)$ is a pre-quasi-isomorphism if for every transverse sequence $(Z_0, Z_1, Y)$ the map
\[ m^2(\cdot, p) \colon \Hom(Z_1, Y) \to \Hom(Z_0, Y) \]
is a quasi-isomorphism and for every transverse sequence of the form $(Y, Z_0, Z_1)$ the map
\[ m^2(p, \cdot)\colon \Hom(Y, Z_0) \to \Hom(Y, Z_1) \]
is a quasi-isomorphism. When such a morphism exists, we will say that $Z_0$ and $Z_1$ are pre-quasi-isomorphic.
\end{df} 

The morphisms satisfying the conditions of Definition \ref{preqiso} are called quasi-isomorphisms in \cite{Ab2, KoS}, but we prefer the term pre-quasi-isomorphism due to the relationship with localization in Proposition \ref{prelocal} below. The pre-quasi-isomorphisms play an important role in restricting non-unital $A_\infty$-pre-categories to a more well-behaved class of objects.

\begin{df} \label{preunital} An $A_\infty$-pre-category is a non-unital $A_\infty$-pre-category such that for every object $Z$ and finite set $\{S_i\}_{i \in I}$ of transverse sequences, there are objects $Z_{-}$ and $Z_+$ that are pre-quasi-isomorphic to $Z$ and $(Z_{-}, S_i, Z_{+})$ is a transverse sequence for all $i \in I$.
\end{df}

Finally, we need the notions of functors and quasi-equivalences between pre-$A_\infty$-categories. 

\begin{df} A functor $F\colon \mathcal{C} \to \mathcal{D}$ between $A_\infty$-pre-categories is given by the following.
\begin{itemize}
\item A map of objects $F\colon \text{Ob}(\mathcal{C}) \to \text{Ob}(\mathcal{D})$ such that the image of any transverse sequence is a transverse sequence.
\item For every $d \geq 1$ and every transverse sequence $(Z_0, \hdots, Z_d)$, a degree $1 - d$ map of vector spaces 
\[ F^d \colon \Hom_\mathcal{C} (Z_{d-1}, Z_d) \otimes \hdots \otimes \Hom_\mathcal{C} (Z_0, Z_1) \to \Hom_\mathcal{D} (F(Z_0), F(Z_d) ) \]
satisfying the $A_\infty$-functor equations and such that $F^1$ sends pre-quasi-isomorphisms to pre-quasi-isomorphisms. 
\end{itemize}
\end{df}

\begin{df} A functor $F\colon \mathcal{C} \to \mathcal{D}$ between $A_\infty$-pre-categories is a quasi-equivalence if every object of $\mathcal{D}$ is pre-quasi-isomorphic to the image of an object in $\mathcal{C}$ and the functor is a quasi-isomorphism on all transverse sequences. 

Two $A_\infty$-pre-categories are quasi-equivalent when they can be related by a sequence of quasi-equivalences (in both directions). 
\end{df}

With all the background terminology in hand, we want to show that setting up a Fukaya-Seidel $A_\infty$-category using localization as done here is equivalent to setting one up instead as an $A_\infty$-pre-category as done in \cite{Ab2}. The following definition captures the key feature that makes this equivalence apparent.

\begin{df} An $A_\infty$-pre-category $\mathcal{C}$ is left-ordered if there is a subset $\text{Ob}^\circ (\mathcal{C}) \subset \text{Ob}(\mathcal{C})$ such that for every $Z \in \text{Ob}^\circ (\mathcal{C})$, there is a sequence $\{ Z^j \}_{j \in \bb{Z}} \subset \text{Ob}(\mathcal{C})$ with $Z^0 = Z$ such that all of the following hold.
\begin{itemize} 
\item A sequence of the form $(Z_0^{j_0}, \hdots, Z_d^{j_d})$ is transverse if and only if $j_0 > j_1 > \hdots > j_d$. 
\item There is a pre-quasi-isomorphism $c_{Z, k \to j} \in \Hom_\mathcal{C} (Z^j, Z^k)$ for all $Z \in \text{Ob}^\circ(\mathcal{C})$ and $j > k \in \bb{Z}$. 
\item The inclusion of the sub-$A_\infty$-pre-category consisting of all $Z^j$ for $Z \in \text{Ob}^\circ (\mathcal{C})$  and $j \in \bb{Z}$ is a quasi-equivalence.
\end{itemize}
\end{df}

Given a left-ordered $A_\infty$-pre-category $\mathcal{C}$, we can construct an associated $A_\infty$-category via localization. First, we have an $A_\infty$-category $\mathcal{A}(\mathcal{C})^\circ$ whose objects are $Z^j$ for $Z \in \text{Ob}^\circ(\mathcal{C})$ and $j \in \bb{Z}$. Morphisms are given by
\[ \Hom_{\mathcal{A}(\mathcal{C})^\circ} (Z_0^j, Z_1^k) = \begin{cases} \Hom_{\mathcal{C}} (Z_0^j, Z_1^k) & j > k \\
 \bb{K} \cdot e_{Z_0^j} & Z_0^j = Z_1^k \\ 0 & \text{otherwise} \end{cases} \]
where $e_{Z_0^j}$ is formal element and $\bb{K}$ is our ground field. The nontrivial $A_\infty$ structure maps are defined to make the $e_{Z_0^j}$ strict units and coincide with those of $\mathcal{C}$ otherwise. We define $\mathcal{A}(\mathcal{C})$ to be the localization of $\mathcal{A}(\mathcal{C})$ at the set of all the pre-quasi-isomorphisms $c_{Z, k \to j}$. 

\begin{rem} One may notice that this setup does not match exactly with the definition $\mathcal{F}_\Delta(W)$ given in Section \ref{flocalize}. Namely, we set up $\mathcal{F}_\Delta(W)$ using only a sequence of angles in $\bb{N}$. This is simply a matter of convention as we could have also flowed the Lagrangian branes in $O_\Delta$ backwards using $K_\Delta$ approaching $-\pi$ to get a sequence in $\bb{Z}$ and obtain a quasi-equivalent $A_\infty$-category. 
\end{rem}

The following proposition was the goal of this appendix. 

\begin{prop} \label{prelocal} If $\mathcal{C}$ is a left-ordered $A_\infty$-pre-category and $\mathcal{A}(\mathcal{C})$ its associated $A_\infty$-category, then $\mathcal{C}$ and $\mathcal{A}(\mathcal{C})$ are quasi-equivalent as $A_\infty$-pre-categories.
\end{prop}
\begin{proof} Let $\mathcal{D}$ be the sub-$A_\infty$-pre-category of $\mathcal{C}$ consisting of all $Z^j$ for $Z \in \text{Ob}^\circ (\mathcal{C})$ and $j \in \bb{Z}$. It is enough to show that $\mathcal{D}$ is quasi-equivalent to $\mathcal{A}(\mathcal{C})$. 

There is a ``functor" $G\colon \mathcal{D} \to \mathcal{A}(\mathcal{C})^\circ$ which is the identity on objects and morphisms with no higher terms. This is not a functor of $A_\infty$-pre-categories as it does not send pre-quasi-isomorphisms to pre-quasi-isomorphisms. However, we claim that the composition $F\colon \mathcal{D} \to \mathcal{A}(\mathcal{C})$ of $G$ with the localization functor is a functor of $A_\infty$-pre-categories which is a quasi-equivalence.

The only property that we need to check for $F$ to be a functor of $A_\infty$-pre-categories that is not immediate is that $F^1$ takes pre-quasi-isomorphisms to pre-quasi-isomorphisms. To see that this property holds, suppose that $e \in \Hom_{\mathcal{D}} (Z_0^j, Z_1^k)$ with $j > k$ is a pre-quasi-isomorphism. Choose some $\ell$ such that $k > \ell$. Since
\[ m^2( \cdot, e) \colon \Hom_{\mathcal{D}} (Z_1^k, Z_0^\ell) \to \Hom_{\mathcal{D}} (Z_0^j, Z_0^\ell) \]
is a quasi-isomorphism, we can find a closed morphism $x$ such that $[m^2(x,e)] = [c_{Z_0 , \ell \to j}]$ in $H^\bullet \Hom_{\mathcal{D}} (Z_0^j, Z_0^\ell)$. Similarly, we can take $m > j$ find a closed morphism $y$ such that $[m^2(e,y)] = [c_{Z_1, k \to m}]$ in $H^\bullet \Hom_\mathcal{D} (Z_1^m, Z_1^k)$. In particular, we conclude that $e$ is a quasi-isomorphism in $\mathcal{A}(\mathcal{C})$ and hence a pre-quasi-isomorphism.

Finally, we can conclude that $F$ is a quasi-equivalence by the fact that the inclusion $\Hom_\mathcal{C} (Z_0^j,  Z_1^k) \to \Hom_{\mathcal{A}(\mathcal{C})} (Z_0^j, Z_1^k)$ is a quasi-isomorphism for $j > k$ by Lemma 7.18 of \cite{SeNotes}. 
\end{proof}

It should be noted that Efimov has shown in \cite{Ef} that in general the quasi-equivalence classes of essentially small $A_\infty$-categories are in bijection with the quasi-equivalence classes of essentially small $A_\infty$-pre-categories.

\small Department of Mathematics, UC Berkeley, Berkeley CA 94720-3840, USA 

\emph{E-mail:} \href{mailto:a.hanlon@berkeley.edu}{a.hanlon@berkeley.edu}

\end{document}